\documentclass[sn-mathphys-num]{sn-jnl}

\usepackage{graphicx}%
\usepackage{multirow}%
\usepackage{amsmath,amssymb,amsfonts}%
\usepackage{amsthm}%
\usepackage{mathrsfs}%
\usepackage[title]{appendix}%
\usepackage{xcolor}%
\usepackage{textcomp}%
\usepackage{manyfoot}%
\usepackage{booktabs}%
\usepackage{algorithm}%
\usepackage{algorithmicx}%
\usepackage{algpseudocode}%
\usepackage{listings}%
\usepackage{subfigure}
\theoremstyle{thmstyleone}%
\newtheorem{theorem}{Theorem}
\theoremstyle{thmstyletwo}%
\newtheorem{example}{Example}%
\newtheorem{remark}{Remark}%
\newtheorem{lemma}{Lemma}%
\theoremstyle{thmstylethree}%
\allowdisplaybreaks
\begin{document}

\title[Article Title]{Optimal error estimates of the stochastic parabolic optimal control problem with integral state constraint}


\author[1]{\fnm{Qiming} \sur{Wang}}\email{202231130040@mail.bnu.edu.cn}
\author[2]{\fnm{Wanfang} \sur{Shen}}\email{wfshen@sdufe.edu.cn}
\author*[3,4]{\fnm{Wenbin} \sur{Liu}}\email{wbliu@uic.edu.cn}


\affil[1]{\orgdiv{School of Mathematical Sciences}, \orgname{Beijing Normal University}, \orgaddress{\street{19 Xinjiekou Outer Street}, \city{Beijing}, \postcode{100875}, \state{Beijing}, \country{China}}}

\affil[2]{\orgdiv{Shandong Key Laboratory of Blockchain Finance}, \orgname{Shandong University of Finance and Economics}, \orgaddress{\street{7366 Erhuan East Road}, \city{Jinan}, \postcode{250014}, \state{Shandong}, \country{China}}}

\affil[3]{\orgdiv{Research Center for Mathematics}, \orgname{Beijing Normal University}, \orgaddress{\street{18 Jinfeng Road}, \city{Zhuhai}, \postcode{519087}, \state{Guangdong}, \country{China}}}

\affil[4]{\orgdiv{Faculty of Business and Management}, \orgname{Beijing Normal University-Hong Kong Baptist University United International College}, \orgaddress{\street{2000 Jindong Road}, \city{Zhuhai}, \postcode{519087}, \state{Guangdong}, \country{China}}}


\abstract{In this paper, the optimal strong error estimates for stochastic parabolic optimal control problem with additive noise and integral state constraint are derived based on time-implicit and finite element discretization.
The continuous and discrete first-order optimality conditions are deduced by constructing the Lagrange functional, which contains forward-backward stochastic parabolic equations and a variational equation. 
The fully discrete version of forward-backward stochastic parabolic equations is introduced as an auxiliary problem and the optimal strong convergence orders in time and space are estimated, which further allows the optimal a priori error estimates for control, state, adjoint state and multiplier to be derived.
Then, a simple and yet efficient gradient projection algorithm is proposed to solve stochastic control problem with state constraint and the convergence rate is proved. 
Numerical experiments are carried out to illustrate the theoretical findings.}

\keywords{Stochastic parabolic optimal control, integral state constraint, forward-backward stochastic parabolic equations, optimal a priori error estimate, gradient projection algorithm}


\pacs[MSC Classification]{49J20, 65M60, 93E20}

\maketitle

\section{Introduction}\label{sec1}
Let $T \in (0, \infty)$ be a fixed time horizon and $(\Omega, \mathcal{F}, \mathbb{F}, \mbox{P})$ be a complete probability space with the natural filtration $\mathbb{F}:=\{\mathcal{F}_t\}_{0\leq t \leq T}$ generated by one-dimensional standard Brownian motion $W_t$. We denote by $\mathcal{D}\subset\mathbb{R}^d (d=1,2,3)$ and $\Delta$ a bounded domain with smooth boundary and the Laplace operator with homogeneous Dirichlet boundary condition in space $L^2(\mathcal{D})$, respectively.
The set of all $L^2(\mathcal{D})$-valued, $\mathbb{F}$-adapted processes is denoted as $L^2_{\mathbb{F}}\left(\Omega;L^2\left(0,T;L^2(\mathcal{D})\right)\right)$, which is abbreviated to $L^2_{\mathbb{F}}(\Omega;0,T;\mathcal{D})$.
The following stochastic optimal control problem (SOCP) with integral state constraint is consider:
\begin{equation}\label{object}
\begin{aligned}
\min\limits_{\substack{ X(t)\in K\\ U(t) \in L^2_{\mathbb{F}}(\Omega;0,T;\mathcal{D}) }}
J(X(t),U(t))=\frac{1}{2}\int_{0}^{T}\mathbb{E}\left[\|X(t)-X_d(t)\|^2+\alpha\|U(t)\|^2\right]dt
\end{aligned}
\end{equation}
subject to
\begin{equation}\label{state}
\left\{\begin{aligned}
dX(t)&=[\Delta X(t)+U(t)]dt+\sigma(t)dW_t,\ t\in(0,T],\\
X(0)&=X_0.
\end{aligned}\right.
\end{equation}
Here $U(t)$ is the control process and $\alpha>0$ is the regularization parameter. Initial value $X_0\in L^2(\mathcal{D})$ and $X_d(t),\sigma(t)\in L_{\mathbb{F}}^2(\Omega;0,T;\mathcal{D})$ are the given desired state and diffusion coefficient.
$X(t)\in K$ is the state process and the constraint set $K$ is given as follows with $\delta\in\mathbb{R}$
\begin{equation}\label{K_constraint}
\begin{aligned}
K:=\left\{X(t)\in L_{\mathbb{F}}^2(\Omega;0,T;\mathcal{D})|\int_{0}^{T}\int_{\mathcal{D}}\mathbb{E}\left[X(t)\right]dxdt\leq\delta\right\}.
\end{aligned}
\end{equation}

Deterministic optimal control problems (OCPs) governed by parabolic equation have been studied so extensively that it is impossible to present a comprehensive review on its development here. 
We refer to \cite{deter1,deter2,deter3,deter4,deter5,deter6,deter7,deter8} and their references. 
In recent years, it has become clear that many phenomena, which are commonly described by deterministic OCPs, may be more fully modeled by SOCPs instead. However, the complexity of the SOCP model is carried over to the solutions themselves, which are no longer
simple functions, but instead stochastic processes and the effective numerical methods play a crucial role.
The most common method for solving OCPs governed by PDEs is the maximization principle method, which converts the directional derivative of the objective functional into a more efficiently computed form called variational inequality by introducing a adjoint state equation. Then the optimality conditional system containing state equation, adjoint state equation and variational inequality is used to solve OCPs.
Analogously, solving SOCP (\ref{object})-(\ref{K_constraint}) via the stochastic maximization principle consists of three main parts, the numerical analysis of state equation (\ref{state}) and the adjoint state equation (\ref{adjoint}), which is a backward stochastic partial differential equation (BSPDE), and the handling of state constraint.
Both strong and weak error estimates of forward stochastic partial differential equation (FSPDE) have been extensively developed. We can refer to \cite{forward1,forward2,forward3,forward4,forward5,forward6,forward7,forward8,forward9} and their references for more details. However, as stated in \cite{finite_tra}, unlike deterministic OCPs and individual FSPDE, numerical scheme should keep the adaptedness of state and control with respect to the filtration in stochastic control system, which restricts the application of many implicit schemes.
Meanwhile, the numerical analysis of the backward stochastic partial differential equation (BSPDE) is relatively difficult, such as the analysis of optimal strong error estimate and the calculation of conditional expectations in numerical simulation, which are also challenges faced in stochastic optimal control problems.
For recent research on BSPDEs, refer to \cite{backward1,backward2,backward3,backward4,backward5}.

Current numerical analyses of stochastic parabolic optimal control problems focus on unconstrained and control constrained problems. In \cite{control2016}, the authors conduct spatial semi-discretization of forward-backward stochastic heat equations and derive the error estimates. They performed the numerical simulations for the forward-backward stochastic partial differential equations (FBSPDEs) and SOCP governed by stochastic heat equation by using the least square Monte Carlo method to approximate the conditional expectations. The space-time full discretization and strong error estimates of FBSPDEs with additive noise as well as the analysis of corresponding SOCP are further complemented in \cite{control2021}. The SOCP with the stochastic parabolic equation driven by linear noise is studied in \cite{control2022,control2024}.
In \cite{control2021Li}, the optimal convergence order of stochastic parabolic optimal control problem with multiplicative noise is obtained based on the semi-discretization of $Z(t)$, that is, a space discretization of only $Z(t)$. The discretization of a Neumann boundary control problem with a stochastic parabolic equation is analyzed in \cite{control2022Li}, where an additive noise occurs in the Neumann boundary condition. For more details on unconstrained or control constrained SOCPs, we can refer to \cite{control2024_1} and its references.

There is a gap in the numerical analysis of state-constrained SOCPs governed by SPDEs and this paper attempts to investigate such problems, where the state constraint is given in the form of integral shown in (\ref{K_constraint}). We first derive the optimality condition system by constructing the Lagrange functional. The optimal strong convergence order of the FBSPDEs is deduced based on the time implicit discretization and piecewise linear finite element discretization in space.
Compared to the unconstrained SOCP with additive noise in \cite{control2021}, the optimal strong convergence order in time is recovered by strengthening the time regularity of functions $\sigma(t)$ and $X_d(t)$.
Then the optimal a prior error estimates of control, state, adjoint state and multiplier are derived, and it's found that the error estimate of multiplier depends on the weak error estimates. Further, an efficient gradient projection algorithm is proposed, which ensures that the state constraint holds by choosing the specific multiplier at each step, and the convergence analysis of the algorithm is given. Finally, numerical examples are performed to verify the theoretical analysis.

The rest of this paper is organized as follows. In Section 2, the continuous first-order optimality condition and spatial semi-discrete SOCP are presented. Optimal spatial semi-discrete error estimates of FBSPDEs are introduced. 
In Section 3, the fully discrete optimality condition system is first given. Subsequently, the optimal strong convergence order of fully discrete FBSPDEs is derived, resulting in the optimal error estimates of the control, state, adjoint state, and multiplier.
In Section 4, an efficient gradient projection algorithm is proposed to solve SOCP and the convergence rate is deduced.
In Section 5, numerical examples are given to verify the theoretical findings. Finally, we close with a few brief remarks in Section 6.

\section{Optimality condition system and space semi-discretization}
In this section, the first-order optimality condition and the regularities of solutions are deduced. Then the SCOP is discretized in space and the spatial semi-discrete error estimates of FBSPDEs are given. We begin by introducing the following notation.

The standard notation $\|\cdot\|_{W^{m,p}(\mathcal{D})}$ is used to denote the norm of Sobolev space $W^{m,p}(\mathcal{D})$. 
In particular, we use $L^{p}(\mathcal{D})$ and $H^{m}(\mathcal{D})$ to denote the cases $m=0$ and $p=2$. We denote by $\|\cdot\|$ and $(\cdot,\cdot)$ the norm and inner product of space $L^{2}(\mathcal{D})$.
The inner product of space $L^2\left(0,T;L^2(\mathcal{D})\right)$ is denoted as $[\cdot,\cdot]$, i.e., $[U(t),V(t)]=\int_{0}^{T}\int_{\mathcal{D}}U(t)V(t)dxdt$. 
The space of all $\mathbb{F}$-adapted processes $X:\Omega\times[0,T]\rightarrow W^{m,p}(\mathcal{D})$ satisfying $\mathbb{E}\left[\int_{0}^{T}\|X(t)\|_{W^{m,p}(D)}^2dt\right]<\infty$ is denoted as $L^2_{\mathbb{F}}\left(\Omega;L^2\left(0,T;W^{m,p}(\mathcal{D})\right)\right)$.
$L^\infty_{\mathbb{F}}\left(0,T;L^2\left(\Omega;W^{m,p}(\mathcal{D})\right)\right)$ and $L^2_{\mathbb{F}}\left(\Omega;C\left(0,T;W^{m,p}(\mathcal{D})\right)\right)$ are the set of all $\mathbb{F}$-adapted processes $X:\Omega\times[0,T]\rightarrow W^{m,p}(\mathcal{D})$ satisfying $\sup_{t\in[0,T]}\mathbb{E}\left[\|X(t)\|^2_{W^{m,p}(\mathcal{D})}\right]<\infty$ and $\mathbb{E}\left[\sup_{t\in[0,T]}\|X(t)\|^2_{W^{m,p}(\mathcal{D})}\right]<\infty$, respectively.

In order to show the dependence of the state variable on the control variable, we use $X(U)$ to denote the state process solved by (\ref{state}) and the reduced SOCP of (\ref{object})-(\ref{K_constraint}) is as follows:
\begin{equation}\label{reduced_object}
\begin{aligned}
\min\limits_{U(t)\in \mathbb{U}_{\delta}}
\mathcal{J}(U(t))=\frac{1}{2}\int_{0}^{T}\mathbb{E}\left[\|X(U)-X_d(t)\|^2+\alpha\|U(t)\|^2\right]dt.
\end{aligned}
\end{equation}
Here $\mathbb{U}_{\delta}$ is the control domain given by
\begin{equation}\label{reduced_domain}
\begin{aligned}
\mathbb{U}_{\delta}:=\left\{U(t)\in L^2_{\mathbb{F}}(\Omega;0,T;\mathcal{D})|G(U(t))\leq 0\right\},
\end{aligned}
\end{equation}
where $G(U(t))=\int_{0}^{T}\int_{\mathcal{D}}\mathbb{E}\left[X(U)\right]dxdt-\delta$. 
According to \cite{control2024_1}, there exists a unique solution $X(t)\in L^2_{\mathbb{F}}\left(\Omega;C\left(0,T;L^2(\mathcal{D})\right)\cap L^2\left(0,T;H_{0}^{1}(\mathcal{D})\right)\right)$ for the state equation (\ref{state}), which implies that SOCP (\ref{object})-(\ref{K_constraint}) is well defined if the control domain $\mathbb{U}_{\delta}$ is non-empty. 
Then by introducing Lagrangian functional, the following optimality condition system is derived.
\begin{theorem}\label{lianxu_first_condition}
Suppose that the Slater's condition holds, i.e., there exists a $\tilde{U}(t)\in L^2_{\mathbb{F}}(\Omega;0,T;\mathcal{D})$ such that $G(\tilde{U}(t))<0$.
Let $(X(t),U(t))$ be the solution of stochastic optimal control problem (\ref{reduced_object})-(\ref{reduced_domain}), then there exists a real number $\mu\geq0$ and a pair of process $(Y(t),Z(t))$ such that $(X(t),Y(t),Z(t),\mu,U(t))$ satisfies the following optimality condition system:
\begin{equation}\label{first_continuous}
\left\{\begin{aligned}
&dX(t)=[\Delta X(t)+U(t)]dt+\sigma(t)dW_t,\ t\in(0,T],\\
&X(0)=X_0,\\
&dY(t)=\left[-\Delta Y(t)-\left(X(t)-X_d(t)+\mu\right)\right]dt+Z(t)dW_t,\ t\in[0,T),\\
&Y(T)=0,\\
&\mathbb{E}\left[\left[\mu,W(t)-X(t)\right]\right]\leq0,\ \forall W(t)\in K,\\
&Y(t)+\alpha U(t)=0.
\end{aligned}\right.
\end{equation}
\end{theorem}
\begin{proof}
From the Theorem 1.56 in \cite{deter2}, there exist a non-negative constant $\mu$ such that
\begin{equation*}
\begin{aligned}
\mu G(U(t))=0,\quad \mathcal{L}_{U}^{\prime}(U(t),\mu)(V(t)-U(t))=0,\ \forall V(t)\in L^2_{\mathbb{F}}(\Omega;0,T;\mathcal{D}),
\end{aligned}
\end{equation*}
where $\mathcal{L}(U(t),\mu):=\mathcal{J}(U(t))+\mu G(U(t))$ denotes the Lagrange functional with $\mu$ being the Lagrange multiplier and $\mathcal{L}_{U}^{\prime}(U(t),\mu)(V(t))$ is the directional derivative with respect to $U(t)$ in direction $V(t)$.

Note that for any $W(t)\in K$, it holds that
\begin{equation*}
\begin{aligned}
0=\mu G(U(t))=\mathbb{E}\left[[\mu,X(U)]\right]-\mu\delta=\mathbb{E}\left[[\mu,X(U)-W(t)]\right]+\mathbb{E}\left[[\mu,W(t)]\right]-\mu\delta,
\end{aligned}
\end{equation*}
which implies that
\begin{equation}\label{first_continuous_third}
\begin{aligned}
\mathbb{E}\left[[\mu,W(t)-X(U)]\right]=\mathbb{E}\left[[\mu,W(t)]\right]-\mu\delta\leq0.
\end{aligned}
\end{equation}
By the definitions of $\mathcal{J}(U(t))$ and $G(U(t))$, we arrive at
\begin{equation}\label{derivative_L}
\begin{aligned}
&\mathcal{L}_{U}^{\prime}(U(t),\mu)(V(t)-U(t))\\
&=\lim_{\kappa\rightarrow 0}\frac{\mathcal{J}\left(U(t)+\kappa(V(t)-U(t))\right)-\mathcal{J}(U(t))}{\kappa}\\
&\quad+\mu\lim_{\kappa\rightarrow 0}\frac{G\left(U(t)+\kappa(V(t)-U(t))\right)-G(U(t))}{\kappa}\\
&=\mathbb{E}\left[\left[X(U)-X_d(t),X^{\prime}(U)(V(t)-U(t)) \right] +\alpha\left[U(t),V(t)-U(t) \right]\right]\\
&\quad+\mathbb{E}\left[\left[\mu,X^{\prime}(U)(V(t)-U(t))\right]\right].
\end{aligned}
\end{equation}
Let $Q(t):=X^{\prime}(U)(V(t)-U(t))$ and it satisfies the following equation from the state equation (\ref{state}):
\begin{equation*}
\left\{\begin{aligned}
&dQ(t)=\left[\Delta Q(t)+V(t)-U(t)\right]dt,\ t\in(0,T],\\
&Q(0)=0.
\end{aligned}\right.
\end{equation*}
In order to rid $Q(t)$ in (\ref{derivative_L}), the following BSPDE is introduced:
\begin{equation}\label{adjoint}
\left\{\begin{aligned}
&dY(t)=\left[-\Delta Y(t)-(X(t)-X_d(t)+\mu)\right]dt+Z(t)dW_t,\ t\in[0,T),\\
&Y(T)=0.
\end{aligned}\right.
\end{equation}
By It\^{o} formula for the processes $Q(t)$ and $Y(t)$, we deduce
\begin{equation*}
\begin{aligned}
0=\mathbb{E}\left[\int_{\mathcal{D}}\int_{0}^{T}Q(t)dY(t)dx\right]+\mathbb{E}\left[\int_{\mathcal{D}}\int_{0}^{T}Y(t)dQ(t)dx\right],
\end{aligned}
\end{equation*}
which implies that
\begin{equation}\label{ito_formul}
\begin{aligned}
\mathbb{E}\left[\left[X(U)-X_d(t)+\mu,Q(t)\right]\right]
=\mathbb{E}\left[\left[Y(t),V(t)-U(t)\right]\right].
\end{aligned}
\end{equation}
For any $V(t)\in L^2_{\mathbb{F}}(\Omega;0,T;\mathcal{D})$, combining (\ref{derivative_L}) and (\ref{ito_formul}) yields 
\begin{equation*}
\begin{aligned}
\mathcal{L}_{U}^{\prime}(U(t),\mu)(V(t)-U(t))=\mathbb{E}\left[\left[Y(t)+\alpha U(t),V(t)-U(t)   \right]\right]=0,
\end{aligned}
\end{equation*}
which implies that $Y(t)+\alpha U(t)=0$. The first-order optimality condition is obtained by (\ref{state}), (\ref{first_continuous_third}), (\ref{adjoint}) and the fact that $Y(t)+\alpha U(t)=0$.
\end{proof}
\begin{remark}
From the fact that $\mu G(U(t))=0$, we can further derive that
\begin{equation}\label{mu_property}
\mu=\left\{\begin{aligned}
&\mathrm{constant}\geq 0,\ &&\mathrm{if}\ \mathbb{E}\left[[1,X(t)]\right]=\delta,\\
&0,\ &&\mathrm{if}\ \mathbb{E}\left[[1,X(t)]\right]<\delta.
\end{aligned}\right.
\end{equation}
\end{remark}
Based on the optimality condition system (\ref{first_continuous}), we have the following regularity results of the solutions to SOCPs.
\begin{lemma}\label{regularity}
Assume that $X_0\in H_0^1(\mathcal{D})\cap H^2(\mathcal{D})$, $X_d(t)\in L_{\mathbb{F}}^2\left(\Omega;L^2(0,T;H_0^1(\mathcal{D}))\right)$ and $\sigma(t)\in L_{\mathbb{F}}^2\left(\Omega;L^2(0,T;H_0^1(\mathcal{D})\cap H^2(\mathcal{D}))\right)$, the solutions solved by optimality condition system (\ref{first_continuous}) satisfy
\begin{equation*}
\begin{aligned}
U(t),X(t),Y(t)\in L^2_{\mathbb{F}}\left(\Omega;C\left(0,T;H_0^1(\mathcal{D})\cap H^2(\mathcal{D})\right)\cap L^2\left(0,T;H_{0}^{1}(\mathcal{D})\cap H^{3}(\mathcal{D})\right)\right)
\end{aligned}
\end{equation*}
and $Z(t)\in L^2_{\mathbb{F}}\left(\Omega;L^2\left(0,T;H_{0}^{1}(\mathcal{D})\cap H^{2}(\mathcal{D}) \right)\right)$.
\end{lemma}
\begin{proof}
Based on the regularity assumptions of $X_0$, $\sigma(t)$ and $U(t)\in L_{\mathbb{F}}^{2}(\Omega;0,T;\mathcal{D})$, from \cite{control2021Li,control2024_1} it follows that $X(t)\in L^2_{\mathbb{F}}\left(\Omega;C\left(0,T;H_0^1(\mathcal{D})\right)\cap L^2\left(0,T;H_{0}^{1}(\mathcal{D})\cap H^{2}(\mathcal{D})\right)\right)$. Combining the regularities of $X(t)$ and $X_d(t)$ gives
$Y(t)\in L^2_{\mathbb{F}}\left(\Omega;C\left(0,T;H_0^1(\mathcal{D})\cap H^2(\mathcal{D})\right)\cap L^2\left(0,T;H_{0}^{1}(\mathcal{D})\cap H^{3}(\mathcal{D})\right)\right)$ and $Z(t)\in L^2_{\mathbb{F}}\left(\Omega;L^2\left(0,T;H_{0}^{1}(\mathcal{D})\cap H^{2}(\mathcal{D})\right)\right)$.
Using the fact $U(t)=-\alpha Y(t)$ gives the regularity of control, which further leads to the regularity of state variable.
\end{proof}

Next the space semi-discretization for SOCP (\ref{reduced_object})-(\ref{reduced_domain}) is performed.
Let $\mathcal{T}_h=\{E\}$ be a family of decompositions of the bounded domain $\mathcal{D}$ into regular triangles and the maximum mesh size is defined as $h:=\max\{\mathrm{diam}(E): E\in\mathcal{T}_h\}$. For a positive integer $N$, a uniform time partition $\Pi=\{t_0,...,t_N\}$ over $[0,T]$ is introduced:
$$0=t_0<t_1<...<t_N=T,\ t_{n+1}-t_n=T/N=\tau\leq 1.$$
We denote by $\mathbb{P}_1(E)$ the space of polynomials of degree $\leq 1$ on element $E$ and 
the finite element space consisting of continuous piecewise linear functions over the triangulation $\mathcal{T}_h$ is given by
\begin{equation*}
\begin{aligned}
\mathbb{V}_h:=\{\phi\in C(\overline{\mathcal{D}}): \phi|_{E}\in\mathbb{P}_1(E),\ \forall E\in\mathcal{T}_h\}.
\end{aligned}
\end{equation*}
Let $\mathbb{V}_{h}^{1}:=\mathbb{V}_h\cap H_{0}^{1}(\mathcal{D})$ and the $L^2$-projection $\Pi_{h}^{1}:L^2(\mathcal{D})\rightarrow\mathbb{V}_h^1$ is defined by $(\Pi_{h}^{1}\xi-\xi,\phi_h)=0$ for all $\phi_h\in\mathbb{V}_{h}^{1}$, which implies that $\|\Pi_{h}^{1}\xi\|\leq\|\xi\|$.
The discrete Laplace $\Delta_h:\mathbb{V}_{h}^{1}\rightarrow\mathbb{V}_{h}^{1}$ is defined as $(-\Delta_h \xi_h,\phi_h)=(\nabla\xi_h,\nabla\phi_h)$ for all $\xi_h,\phi_h\in\mathbb{V}_{h}^{1}$. We denote by $\mathcal{R}_h:H_{0}^{1}(\mathcal{D})\rightarrow\mathbb{V}_{h}^{1}$ the Ritz-projection, which satisfies that $(\nabla\mathcal{R}_h\xi,\nabla\phi_h)=(\nabla\xi,\nabla\phi_h)$ for all $\phi_h\in\mathbb{V}_{h}^{1}$.
Throughout the paper, $C$ is a positive constant, not necessarily the same at each occurrence, but it's independent of the discretization parameters $h$ and $N$. The projection estimates and inverse inequality are given as follows.
\begin{lemma}(\cite{touying_error,control2024_1,inverse})\label{pro_estimate}
For the projections $\Pi_{h}^{1}$, $\mathcal{R}_h$ and functions $\phi_1\in H^2(\mathcal{D})$, $\phi_2\in H_{0}^{1}(\mathcal{D})\cap H^2(\mathcal{D})$, there exists a positive constant $C$ such that
\begin{equation*}
\begin{aligned}
&\|\phi_1-\Pi_h^1\phi_1\|+\|\phi_2-\mathcal{R}_h\phi_2\|\leq Ch(\|\phi_1\|_{H^1(\mathcal{D})}+\|\phi_2\|_{H^1(\mathcal{D})}).\\
&\|\phi_1-\Pi_h^1\phi_1\|+\|\phi_2-\mathcal{R}_h\phi_2\|\leq Ch^2(\|\phi_1\|_{H^2(\mathcal{D})}+\|\phi_2\|_{H^2(\mathcal{D})}),\\
&\|\phi_1-\Pi_{h}^{1}\phi_1\|_{H^1(\mathcal{D})}+\|\phi_2-\mathcal{R}_h\phi_2\|_{H^1(\mathcal{D})}\leq Ch(\|\phi_1\|_{H^2(\mathcal{D})}+\|\phi_2\|_{H^2(\mathcal{D})}).
\end{aligned}
\end{equation*}
\end{lemma}
\begin{lemma}(\cite{backward5,touying_error,inverse})\label{inver_ertimate}
For the regular triangulation $\mathcal{T}_h$, there exists a positive constant $C$
such that for any $\phi_h\in\mathbb{V}_{h}^{1}$
 \begin{equation*}
\begin{aligned}
\|\phi_h\|_{H^{1}(\mathcal{D})}^2\leq Ch^{-2}\|\phi_h\|^2.
\end{aligned}
\end{equation*}
\end{lemma}
Then the spatial semi-discrete SOCP is obtained as following:
\begin{equation*}\label{semi_object}
\begin{aligned}
\min\limits_{\substack{ X_h(t)\in K_h\\ U_h(t) \in L^2_{\mathbb{F}}\left(\Omega;L^2(0,T;\mathbb{V}_{h}^{1})\right) }}
J(X_h(t),U_h(t))=\frac{1}{2}\int_{0}^{T}\mathbb{E}\left[\|X_h(t)-X_d(t)\|^2+\alpha\|U_h(t)\|^2\right]dt
\end{aligned}
\end{equation*}
subject to
\begin{equation*}\label{semi_state}
\left\{\begin{aligned}
dX_h(t)&=[\Delta_h X_h(t)+U_h(t)]dt+\Pi_{h}^{1}\sigma(t)dW_t,\ t\in(0,T],\\
X_h(0)&=\Pi_{h}^{1}X_0.
\end{aligned}\right.
\end{equation*}
Here the constraint set $K_h$ is given by
\begin{equation*}
\begin{aligned}
K_h:=\left\{V_h(t)\in L^2_{\mathbb{F}}\left(\Omega;L^2(0,T;\mathbb{V}_{h}^{1})\right)|\int_{0}^{T}\int_{\mathcal{D}}\mathbb{E}\left[V_h(t)\right]dxdt\leq\delta\right\}.
\end{aligned}
\end{equation*}
Performing the analysis like Theorem \ref{lianxu_first_condition}, the following spatial semi-discrete first-order optimal condition is obtained:
\begin{equation*}\label{semi_first_continuous}
\left\{\begin{aligned}
&dX_h(t)=[\Delta_h X_h(t)+U_h(t)]dt+\Pi_{h}^{1}\sigma(t)dW_t,\ t\in(0,T],\\
&X_h(0)=\Pi_{h}^{1}X_0,\\
&dY_h(t)=\left[-\Delta_h Y_h(t)-\Pi_{h}^{1}(X_h(t)-X_d(t)+\mu_h)\right]dt+Z_h(t)dW_t,\ t\in[0,T),\\
&Y_h(T)=0,\\
&\mathbb{E}\left[\left[\mu_h,W_h(t)-X_h(t)\right]\right]\leq0,\ \forall W_h(t)\in K_{h},\\
&Y_h(t)+\alpha U_h=0.
\end{aligned}\right.
\end{equation*}
In order to give the optimal error estimates of fully discretized FBSPDEs in the next section, the following auxiliary problem is introduced:
\begin{equation}\label{semifu_first_continuous}
\left\{\begin{aligned}
&dX_h^U(t)=\left[\Delta_h X_h^U(t)+\Pi_{h}^{1}U(t)\right]dt+\Pi_{h}^{1}\sigma(t)dW_t,\ t\in(0,T],\\
&X_h^U(0)=\Pi_{h}^{1}X_0,\\
&dY_h^U(t)=\left[-\Delta_h Y_h^U(t)-\Pi_{h}^{1}\left(X(t)-X_d(t)+\mu\right)\right]dt
+Z_h^U(t)dW_t,\ t\in[0,T),\\
&Y_h^U(T)=0.
\end{aligned}\right.
\end{equation}
It is obvious that (\ref{semifu_first_continuous}) is the finite element approximation of FBSPDEs in continuous optimality condition system (\ref{first_continuous}). According to Theorems 3.2 and 3.6 in \cite{control2024_1},
the following optimal semi-discrete error estimates are deduced.
\begin{lemma}\label{optimal_state}
Under the regularities in Lemma \ref{regularity} and we assume that 
$X_d(t)\in L_{\mathbb{F}}^2\left(\Omega;L^2\left(0,T;H_0^1(\mathcal{D})\cap H^2(\mathcal{D})\right)\right)$, let $(X(t),Y(t),Z(t))$ and $\left(X_h^U(t),Y_h^U(t),Z_h^U(t)\right)$ be the solutions of FBSPDEs in (\ref{first_continuous}) and (\ref{semifu_first_continuous}), respectively, then the following estimates hold:
\begin{equation*}\label{optimal_fspde_semi}
\begin{aligned}
&\mathbb{E}\left[\sup_{t\in[0,T]}\|X(t)-X_h^U(t)\|^2\right]+\mathbb{E}\left[\int_{0}^{T}h^2\|\nabla X(t)-\nabla X_{h}^U(t)\|^2dt\right]\leq Ch^4,\\
&\sup_{t\in[0,T]}\mathbb{E}\left[\|Y(t)-Y_h^U(t)\|^2\right]\leq Ch^4,\\
&\mathbb{E}\left[\sup_{t\in[0,T]}\|Y(t)-Y_h^U(t)\|^2\right]+\mathbb{E}\left[\int_{0}^{T}\|\nabla Y(t)-\nabla Y_{h}^U(t)\|^2+\|Z(t)-Z_{h}^U(t)\|^2dt\right]\leq Ch^2.
\end{aligned}
\end{equation*}
\end{lemma}
Here, we want to show that the optimal convergence rate of $\int_{0}^{T}\mathbb{E}\left[\|Z(t)-Z_h^U(t)\|^2\right]dt$ 
can be achieved by the improved  regularity
\begin{equation}\label{data_assume}
\begin{aligned}
Y(t)\in L^2_{\mathbb{F}}\left(\Omega; L^2\left(0,T;H_{0}^{1}(\mathcal{D})\cap H^4(\mathcal{D})\right)\right).
\end{aligned}
\end{equation}
Then the following error estimate between $Z(t)$ and $Z_h^U(t)$ is given.
\begin{lemma} 
Under the assumptions in Lemma \ref{optimal_state} and condition (\ref{data_assume}), it holds that
\begin{equation*}
\begin{aligned}
\mathbb{E}\left[\int_{0}^{T}\|Z(t)-Z_{h}^U(t)\|^2dt\right]\leq Ch^4.
\end{aligned}
\end{equation*}
\end{lemma}
\begin{proof}
Let $e_Y:=Y_{h}^{U}(t)-Y(t)=\Theta_Y(t)+\Lambda_Y(t)$ and $e_Z:=Z_{h}^{U}(t)-Z(t)=\Theta_Z(t)+\Lambda_Z(t)$, where 
\begin{equation*}
\begin{aligned}
&\Theta_Y(t)=Y_h^{U}(t)-\mathcal{R}_h Y(t),\quad\Lambda_Y(t)=\mathcal{R}_h Y(t)-Y(t),\\
&\Theta_Z(t)=Z_h^{U}(t)-\mathcal{R}_h Z(t),\quad\Lambda_Z(t)=\mathcal{R}_h Z(t)-Z(t).
\end{aligned}
\end{equation*}
Here $\Theta_Y(t)$ and $\Lambda_Y(t)$ are the solutions of BSPDEs
\begin{equation}\label{adjoint_theta}\left\{
\begin{aligned}
d\Theta_Y(t)&=
-\left[\Delta_h Y_h^{U}-\mathcal{R}_h\Delta Y(t)+(\Pi_{h}^{1}-\mathcal{R}_h)(X(t)-X_d(t)+\mu)\right]dt\\
&\quad+\Theta_Z(t)dW_t,\ t\in[0,T),\\
\Theta_Y(T)&=0
\end{aligned}\right.
\end{equation} 
and
\begin{equation}\label{adjoint_lambda}\left\{
\begin{aligned}
d\Lambda_Y(t)&=\left[(\mathcal{R}_h-I)\left(-\Delta Y(t)-(X(t)-X_d(t)+\mu)\right)\right]dt\\
&\quad+\Lambda_Z(t)dW_t,\ t\in[0,T),\\
\Lambda_Y(T)&=0.
\end{aligned}\right.
\end{equation} 
From the properties of projections $\Pi_{h}^{1}$ and $\mathcal{R}_h$, for any $\phi_h(x)\in\mathbb{V}_{h}^{1}$ it holds that
\begin{equation}
\begin{aligned}\label{semi_adjoint1}
&\int_{\mathcal{D}}\phi_h(x)d\Theta_Y(t)dx \\
&=-\int_{\mathcal{D}}\phi_h(x)d\Lambda_Y(t)dx-\int_{\mathcal{D}}\left(\Delta_hY_h^{U}(t)-\Delta Y(t)\right)\phi_h(x)dxdt  \\
&\quad+\int_{\mathcal{D}}\left(X(t)-X_d(t)+\mu-\Pi_h^1(X(t)-X_d(t)+\mu)\right)\phi_h(x)dtdx 
+\int_{\mathcal{D}}e_Z\phi_h(x)dxdW_t  \\
&=-\int_{\mathcal{D}}\phi_h(x)d\Lambda_Y(t)dx+\int_{\mathcal{D}}\nabla\left(Y_h^{U}(t)-Y(t)\right)\cdot\nabla\phi_h(x)dxdt  
+\int_{\mathcal{D}}e_Z\phi_h(x)dxdW_t \\
&=-\int_{\mathcal{D}}\phi_h(x)d\Lambda_Y(t)dx+\int_{\mathcal{D}}\nabla\Theta_Y(t)\cdot\nabla\phi_h(x)dxdt
+\int_{\mathcal{D}}e_Z\phi_h(x)dxdW_t.
\end{aligned}
\end{equation}
For the fixed $(\omega,t)\in\Omega\times[0,T]$, taking $\phi_h=\Theta_Y$ in (\ref{semi_adjoint1}) yields
\begin{equation}\label{semi_adjoint2}
\begin{aligned}
\int_{\mathcal{D}}\Theta_Y(t)d\Theta_Y(t)dx
=-\int_{\mathcal{D}}\Theta_Y(t)d\Lambda_Y(t)dx+\int_{\mathcal{D}}|\nabla\Theta_Y(t)|^2dxdt
+\int_{\mathcal{D}}e_Z\Theta_Y(t)dxdW_t.
\end{aligned}
\end{equation}
Using the It\^{o}'s formula yields
\begin{equation}\label{error_equ2}
\begin{aligned}
d(\Theta_Y(t))^2=2\Theta_Y(t)d\Theta_Y(t)+d\Theta_Y(t)d\Theta_Y(t),
\end{aligned}
\end{equation}
From (\ref{adjoint_theta}), (\ref{semi_adjoint2}) and (\ref{error_equ2}), it holds that
\begin{equation}\label{semi_adjoint3}
\begin{aligned}
&\mathbb{E}\left[\|\Theta_Y(t)\|^2\right]+2\mathbb{E}\left[\int_{t}^{T}\|\nabla\Theta_Y(s)\|^2ds\right]
+\mathbb{E}\left[\int_{t}^{T}\|\Theta_Z(s)\|^2ds\right]\\
&=\mathbb{E}\left[\|\Theta_Y(T)\|^2\right]+2\mathbb{E}\left[\int_{t}^{T}\int_{D}\Theta_Y(s)d\Lambda_{Y}(s)dx  \right]\\
&=2\mathbb{E}\left[\int_{t}^{T}\int_{D}\Theta_Y(s)d\Lambda_{Y}(s)dx  \right].
\end{aligned}
\end{equation}
From (\ref{adjoint_lambda}), Lemma \ref{pro_estimate}, Cauchy-Schwartz inequality and the regularity of $Y(t)$, it holds that
\begin{align*}\label{semi_adjoint5}
&\mathbb{E}\left[\int_{t}^{T}\int_{\mathcal{D}}\Theta_Y(s)d\Lambda_Y(s)dx\right]\nonumber\\
&\leq\mathbb{E}\left[\int_{t}^{T}\frac{\|\Theta_Y(s)\|^2}{2}+\|(\mathcal{R}_h-I)\Delta Y(s)\|^2+\|(\mathcal{R}_h-I)\left(X(s)-X_d(s)+\mu\right)\|^2ds\right]\nonumber\\
&\leq\mathbb{E}\left[\int_{t}^{T}\frac{\|\Theta_Y(s)\|^2}{2}ds\right]+Ch^4\mathbb{E}\left[\int_{t}^{T}\|Y(s)\|_{H^4(\mathcal{D})}^{2}+\|X(s)-X_d(s)+\mu\|_{H^2(\mathcal{D})}^{2}ds\right]\nonumber\\
&\leq Ch^4+\mathbb{E}\left[\int_{t}^{T}\frac{\|\Theta_Y(s)\|^2}{2}ds\right].
\end{align*}
Taking the above estimate into (\ref{semi_adjoint3}) gives
\begin{equation*}
\begin{aligned}
&\mathbb{E}\left[\|\Theta_Y(t)\|^2\right]+\mathbb{E}\left[\int_{t}^{T}\|\nabla\Theta_Y(s)\|^2ds\right]
+\mathbb{E}\left[\int_{t}^{T}\|\Theta_Z(s)\|^2ds\right]\\
&\leq Ch^4+\mathbb{E}\left[\int_{t}^{T}\|\Theta_Y(s)\|^2ds\right].
\end{aligned}
\end{equation*}
Then by the Gronwall inequality one gets
\begin{equation}\label{semi_adjoint6}
\begin{aligned}
\mathbb{E}\left[\|\Theta_Y(t)\|^2\right]+\mathbb{E}\left[\int_{0}^{T}\|\nabla\Theta_Y(t)\|^2dt\right]
+\mathbb{E}\left[\int_{0}^{T}\|\Theta_Z(t)\|^2dt\right]\leq Ch^4.
\end{aligned}
\end{equation}
Employing the regularity $Z(t)$ and Lemma \ref{pro_estimate} we arrive at
\begin{equation}\label{semi_adjoint7}
\begin{aligned}
\mathbb{E}\left[\int_{0}^{T}\|\Lambda_Z(t)\|^2dt\right]\leq Ch^4
\mathbb{E}\left[\int_{0}^{T}\|Z(t)\|_{H^{2}(\mathcal{D})}^2dt\right]\leq Ch^4.
\end{aligned}
\end{equation}
Combining (\ref{semi_adjoint6}) and (\ref{semi_adjoint7}) yields
\begin{equation*}\label{estimate_eyez}
\begin{aligned}
\mathbb{E}\left[\int_{0}^{T}\|e_Z\|^2dt\right]\leq 2\mathbb{E}\left[\int_{0}^{T}\|\Theta_Z(t)\|^2dt\right]
+2\mathbb{E}\left[\int_{0}^{T}\|\Lambda_Z(t)\|^2dt\right]\leq Ch^4.
\end{aligned}
\end{equation*}
\end{proof}

\section{Space-time discretization and optimal error estimates}
In this section, the space-time discretization of SOCP is first given and the fully discrete first-order optimality condition is rigorously derived. We then give the fully discrete version of FBSPDEs in (\ref{first_continuous}) and derive the optimal strong convergence orders of time and space in subsection \ref{errror_fbspdes}. Optimal a priori error estimates of control, state, adjoint state and multiplier are deduced in subsection \ref{error_control} based on the results obtained in subsection \ref{errror_fbspdes}.

\subsection{Space-time discretization and error estimates of FBSPDEs}\label{errror_fbspdes}
We begin this part by introducing the following space-time fully discrete spaces for $n=0,1,...,N-1$:
\begin{equation*}
\begin{aligned}
&\mathbb{X}_{h\tau}:=\left\{X_{h\tau}\in L^2_{\mathbb{F}}\left(\Omega;L^2(0,T;\mathbb{V}_{h}^{1})\right)|X_{h\tau}(t)=\mathbb{E}[X_{h\tau}(t_{n+1})|\mathcal{F}_{t_n}],\ \forall t\in(t_n,t_{n+1})\right\},\\
&\mathbb{U}_{h\tau}:=\left\{U_{h\tau}\in L^2_{\mathbb{F}}\left(\Omega;L^2(0,T;\mathbb{V}_{h}^{1})\right)|U_{h\tau}(t)=U_{h\tau}(t_{n}),\ \forall t\in[t_n,t_{n+1})\right\},\\
&\mathbb{R}_{h\tau}:=\left\{R_{h\tau}\in L^2(0,T;\mathbb{V}_{h}^{1})|R_{h\tau}(t)=R_{h\tau}(t_{n+1}),\ \forall t\in(t_n,t_{n+1}]\right\},\\
&\mathbb{L}_{h\tau}:=\left\{L_{h\tau}\in L^2(0,T;\mathbb{V}_{h}^{1})|L_{h\tau}(t)=L_{h\tau}(t_{n}),\ \forall t\in[t_n,t_{n+1})\right\}.
\end{aligned}
\end{equation*}
For $t_n\in\Pi$, let $\tilde{X}_{d}(t_n)=X_d(t_n)$ and 
define 
$\tilde{X}_{d}(t):=\mathbb{E}[X_d(t_{n+1})|\mathcal{F}_{t_n}],t\in(t_n,t_{n+1})$ for $n=0,1,...,N-1$. Let $K_{h\tau}:=K\cap\mathbb{X}_{h\tau}$, then the fully discrete SOCP reads: Find an optimal pair $\left(X_{h\tau}(t),U_{h\tau}(t)\right)\in K_{h\tau}\times \mathbb{U}_{h\tau}$ which minimizes the cost functional
\begin{equation}\label{fully_object}
\begin{aligned}
J_{h\tau}(X_{h\tau},U_{h\tau})
&=\frac{1}{2}\int_{0}^{T}\mathbb{E}\left[\|X_{h\tau}-\tilde{X}_{d}\|^2+\alpha\|U_{h\tau}\|^2\right]dt
\end{aligned}
\end{equation}
subject to
\begin{equation}\label{fully_state}
\left\{\begin{aligned}
&X_{h\tau}(t_{n+1})-X_{h\tau}(t_{n})=\tau\left[\Delta_hX_{h\tau}(t_{n+1})+U_{h\tau}(t_n)\right]+\Pi_{h}^{1}\sigma(t_n)\Delta W_{n+1},\\
&X_{h\tau}(0)=\Pi_{h}^{1}X_0,
\end{aligned}\right.
\end{equation}
where $\Delta W_{n+1}=W_{t_{n+1}}-W_{t_n}$. We define $G_{h\tau}(U_{h\tau}(t)):=\int_{0}^{T}\int_{\mathcal{D}}\mathbb{E}\left[X_{h\tau}(U_{h\tau})\right]dxdt-\delta$ and the reduced cost functional of fully discrete SOCP (\ref{fully_object})-(\ref{fully_state}) with the control domain 
\begin{equation}\label{lisan_control_domain}
\begin{aligned}
\mathbb{U}_{h\tau,\delta}:=\{U_{h\tau}(t)\in\mathbb{U}_{h\tau}|X_{h\tau}(U_{h\tau})\in\mathbb{X}_{h\tau},G_{h\tau}\left(U_{h\tau}(t)\right)\leq0\}
\end{aligned}
\end{equation}
is denoted as 
\begin{equation*}
\begin{aligned}
\mathcal{J}_{h\tau}\left(U_{h\tau}(t)\right)=\frac{1}{2}\int_{0}^{T}\mathbb{E}\left[\|X_{h\tau}(U_{h\tau})-\tilde{X}_{d}\|^2+\alpha\|U_{h\tau}\|^2\right]dt.
\end{aligned}
\end{equation*}
Then the following fully discrete first-order optimality condition of (\ref{fully_object})-(\ref{lisan_control_domain}) is derived.
\begin{theorem}\label{fully_first_theorem}
Under the Slater's condition, 
let $\left(X_{h\tau}(t),U_{h\tau}(t)\right)\in K_{h\tau}\times \mathbb{U}_{h\tau}$ be the solution of SOCP (\ref{fully_object})-(\ref{fully_state}), then there exists a constant $\mu_{h\tau}\geq0$ and a pair of process $(Y_{h\tau}(t),Z_{h\tau}(t))\in \left(\mathbb{U}_{h\tau}\right)^2$ such that $(X_{h\tau}(t),Y_{h\tau}(t),Z_{h\tau}(t),\mu_{h\tau},U_{h\tau}(t))$ satisfies the following optimality condition system for $n=0,1,...,N-1$:
\begin{equation}\label{fully_first_continuous}
\left\{\begin{aligned}
&X_{h\tau}(t_{n+1})-X_{h\tau}(t_{n})=\tau\left[\Delta_hX_{h\tau}(t_{n+1})+U_{h\tau}(t_n)\right]+\Pi_{h}^{1}\sigma(t_n)\Delta W_{n+1},\\
&X_{h\tau}(0)=\Pi_{h}^{1}X_0,\\
&Z_{h\tau}(t_n)=\mathbb{E}\left[\left(\tau^{-1}Y_{h\tau}(t_{n+1})+\Pi_{h}^{1}\left(X_{h\tau}(t_{n+1})- X_{d}(t_{n+1})\right) \right)\Delta W_{n+1}  |\mathcal{F}_{t_n}\right],\\
&Y_{h\tau}(t_{n})=\mathbb{E}\left[Y_{h\tau}(t_{n+1})+\tau\Delta_hY_{h\tau}(t_{n})+\tau  \Pi_{h}^{1}\left(X_{h\tau}(t_{n+1})- X_{d}(t_{n+1})+\mu_{h\tau}\right)   |\mathcal{F}_{t_n}\right],\\
&Y_{h\tau}(T)=0,\\
&\mathbb{E}\left[\left[\mu_{h\tau}, W_{h\tau}(t)-X_{h\tau}(t)\right]\right]\leq0,\ \forall W_{h\tau}(t)\in K_{h\tau},\\
&Y_{h\tau}(t)+\alpha U_{h\tau}(t)=0,
\end{aligned}\right.
\end{equation}
where $\mu_{h\tau}$ satisfies 
\begin{equation}\label{lisan_mu_property}
\mu_{h\tau}=\left\{\begin{aligned}
&\mathrm{constant}\geq0,\ &&\mathrm{if}\ \mathbb{E}\left[[1,X_{h\tau}(t)]\right]=\delta,\\
&0,\ &&\mathrm{if}\ \mathbb{E}\left[[1,X_{h\tau}(t)]\right]<\delta.
\end{aligned}\right.
\end{equation}
\end{theorem}
\begin{proof}
Following the Theorem \ref{lianxu_first_condition}, for any $V_{h\tau}(t)\in\mathbb{U}_{h\tau}$ there exist a real number $\mu_{h\tau}\geq0$ such that
\begin{equation*}
\begin{aligned}
\mu_{h\tau}G_{h\tau}\left(U_{h\tau}(t)\right)=0,\ \mathcal{\tilde{L}}_{U_{h\tau}}^{\prime}\left(U_{h\tau}(t),\mu_{h\tau}\right)\left(V_{h\tau}(t)-U_{h\tau}(t)\right)=0,
\end{aligned}
\end{equation*}
where $\mathcal{\tilde{L}}\left(U_{h\tau}(t),\mu_{h\tau}\right):=\mathcal{J}_{h\tau}\left(U_{h\tau}(t)\right)+\mu_{h\tau}G_{h\tau}\left(U_{h\tau}(t)\right)$ is the Lagrangian functional.
Performing the analysis like (\ref{first_continuous_third}) and (\ref{derivative_L}), for any $W_{h\tau}(t)\in K_{h\tau}$ and $V_{h\tau}(t)\in\mathbb{U}_{h\tau}$, we have
\begin{equation}\label{full_first_continuous_third}
\begin{aligned}
\mathbb{E}\left[[\mu_{h\tau},W_{h\tau}(t)-X_{h\tau}(U_{h\tau})]\right]=\mathbb{E}\left[[\mu_{h\tau},W_{h\tau}(t)]\right]-\mu_{h\tau}\delta\leq0
\end{aligned}
\end{equation}
and
\begin{equation}\label{lisan_derivative_L}
\begin{aligned}
&\mathcal{\tilde{L}}_{U_{h\tau}}^{\prime}\left(U_{h\tau}(t),\mu_{h\tau}\right)\left(V_{h\tau}(t)-U_{h\tau}(t)\right)\\
&=\mathbb{E}\left[\left[X_{h\tau}(U_{h\tau})-\tilde{X}_d(t)+\mu_{h\tau},Q_{h\tau}(t)\right]
+\alpha\left[U_{h\tau}(t),V_{h\tau}(t)-U_{h\tau}(t)\right]    \right],
\end{aligned}
\end{equation}
where $Q_{h\tau}(t)\in\mathbb{X}_{h\tau}$ is solved by the following equation:
\begin{equation}\label{DX_ht}
\left\{\begin{aligned}
&Q_{h\tau}(t_{n+1})-Q_{h\tau}(t_n)
=\tau\left[\Delta_hQ_{h\tau}(t_{n+1})+V_{h\tau}(t_n)-U_{h\tau}(t_n)\right],\\
&Q_{h\tau}(0)=0.
\end{aligned}\right.
\end{equation}
In order to eliminate $Q_{h\tau}(t)$ in (\ref{lisan_derivative_L}), the following backward SPDE is introduced: Find $(Y_{h\tau}(t),Z_{h\tau}(t))\in\left(\mathbb{U}_{h\tau}\right)^2$ such that
\begin{equation}\label{fully_adjoint}
\left\{\begin{aligned}
&Z_{h\tau}(t_n)=\mathbb{E}\left[\left(\tau^{-1}Y_{h\tau}(t_{n+1})+\Pi_{h}^{1}\left(X_{h\tau}(t_{n+1})- X_{d}(t_{n+1})\right) \right)\Delta W_{n+1}  |\mathcal{F}_{t_n}\right],\\
&Y_{h\tau}(t_{n})=\mathbb{E}\left[Y_{h\tau}(t_{n+1})+\tau\Delta_hY_{h\tau}(t_{n})+\tau  \Pi_{h}^{1}\left(X_{h\tau}(t_{n+1})- X_{d}(t_{n+1})+\mu_{h\tau}\right)   |\mathcal{F}_{t_n}\right],\\
&Y_{h\tau}(T)=0.
\end{aligned}\right.
\end{equation}
Employing the discrete It\^{o}'s formula for (\ref{DX_ht}) and (\ref{fully_adjoint}) as well as the $\mathcal{F}_{t_n}$-measurable of $Q_{h\tau}(t_{n+1})$ gives
\begin{equation}\label{discrete_ito}
\begin{aligned}
&\int_{\mathcal{D}}\tau\sum_{n=0}^{N-1}\mathbb{E}\left[\left(X_{h\tau}(t_{n+1})-X_d(t_{n+1})+\mu_{h\tau}\right)Q_{h\tau}(t_{n+1})\right]dx\\
&=\int_{\mathcal{D}}\tau\sum_{n=0}^{N-1}\mathbb{E}\left[Y_{h\tau}(t_{n})\left(V_{h\tau}(t_n)-U_{h\tau}(t_n)\right)\right] dx.
\end{aligned}
\end{equation}
By the property of conditional expectation, (\ref{discrete_ito}) is equivalent to the following form:
\begin{equation*}
\begin{aligned}
\mathbb{E}\left[ \left[X_{h\tau}(t)-\tilde{X}_d(t)+\mu_{h\tau},Q_{h\tau}(t)                     \right]\right]=
\mathbb{E}\left[\left[Y_{h\tau}(t),V_{h\tau}(t)-U_{h\tau}(t)\right]\right].
\end{aligned}
\end{equation*}
Therefore, for any $V_{h\tau}(t)\in\mathbb{U}_{h\tau}$, (\ref{lisan_derivative_L}) can be rewritten as follows:
\begin{equation*}
\begin{aligned}
\mathcal{\tilde{L}}_{U_{h\tau}}^{\prime}\left(U_{h\tau}(t),\mu_{h\tau}\right)\left(V_{h\tau}(t)-U_{h\tau}(t)\right)
=\mathbb{E}\left[\left[Y_{h\tau}(t)+\alpha U_{h\tau}(t),V_{h\tau}(t)-U_{h\tau}(t)\right]\right]=0,
\end{aligned}
\end{equation*}
which implies that $Y_{h\tau}(t)+\alpha U_{h\tau}(t)=0$. 

Combining (\ref{fully_state}), (\ref{full_first_continuous_third}) and (\ref{fully_adjoint}) yields (\ref{fully_first_continuous}), and (\ref{lisan_mu_property}) is obtained directly from the fact that $\mu_{h\tau}\geq 0$ and $\mu_{h\tau}G_{h\tau}\left(U_{h\tau}(t)\right)=0$.
\end{proof}

Next the following fully discrete auxiliary problem is given: Find $\left(  X_{h\tau}^{U}(t),Y_{h\tau}^{U}(t),Z_{h\tau}^{U}(t)\right)\in\mathbb{X}_{h\tau}\times\left(\mathbb{U}_{h\tau}\right)^2$ such that
\begin{equation}\label{fullyfu_first_continuous}
\left\{\begin{aligned}
&X_{h\tau}^{U}(t_{n+1})-X_{h\tau}^{U}(t_{n})=\tau\left[\Delta_hX_{h\tau}^{U}(t_{n+1})+\Pi_{h}^{1}U(t_n)\right]+\Pi_{h}^{1}\sigma(t_n)\Delta W_{n+1},\\
&X_{h\tau}^{U}(0)=\Pi_{h}^{1}X_0,\\
&Z_{h\tau}^U(t_n)=\mathbb{E}\left[\left(\tau^{-1}Y_{h\tau}^U(t_{n+1})+\Pi_{h}^{1}\left(X(t_{n+1})- X_{d}(t_{n+1})\right) \right)\Delta W_{n+1}  |\mathcal{F}_{t_n}\right],\\
&Y_{h\tau}^U(t_{n})=\mathbb{E}\left[Y_{h\tau}^U(t_{n+1})+\tau\Delta_hY_{h\tau}^U(t_{n})+\tau  \Pi_{h}^{1}\left(X(t_{n+1})- X_{d}(t_{n+1})+\mu\right)   |\mathcal{F}_{t_n}\right],\\
&Y_{h\tau}^{U}(T)=0.
\end{aligned}\right.
\end{equation}
Note that (\ref{fullyfu_first_continuous}) is the space-time fully discrete approximation of the FBSPDEs in (\ref{first_continuous}). In order to deduce the optimal strong convergence orders in space and time, the following equation is introduced:
\begin{equation}
\left\{\begin{aligned}
&d\hat{Y}_{h\tau}(t)=\left[-\Delta_hY_{h\tau}^{U}(\pi(t))-\Pi_{h}^{1}\left(X(\tau(t))- X_{d}(\tau(t))+\mu\right) \right]dt+\hat{Z}_{h\tau}(t)dW_t,\\
&\hat{Y}_{h\tau}(T)=0,
\end{aligned}\right.
\end{equation}
where $\pi(t)=t_n$ and $\tau(t)=t_{n+1}$ with $t\in[t_n,t_{n+1})$. According to the Lemma 3.4 in \cite{control2021}, it follows that
\begin{equation}\label{Fuzhu2}
\begin{aligned}
Y_{h\tau}^U(t_n)=\hat{Y}_{h\tau}(t_n),\ Z_{h\tau}^U(t_n)=\tau^{-1}\mathbb{E}\left[\int_{t_n}^{t_{n+1}}\hat{Z}_{h\tau}(s)ds |\mathcal{F}_{t_n}\right].
\end{aligned}
\end{equation}
Then, the following estimates are obtained with the results of Lemma \ref{optimal_state} and (\ref{Fuzhu2}).
\begin{theorem}\label{fbspde_exact_fu}
Under the regularities in Lemma \ref{optimal_state}, we assume $X_d(t)$ and $\sigma(t)$ satisfy
\begin{equation}\label{sigma_xingzhi}
\begin{aligned}
\sum_{n=0}^{N-1}\int_{t_n}^{t_{n+1}}\mathbb{E}\left[\|X_d(t)-X_d(t_{n+1})\|^2+\|\sigma(t)-\sigma(t_n)\|^2\right]dt\leq C\tau^2.
\end{aligned}
\end{equation}
Let $\left(X(t),Y(t)\right)$ and $\left(X_{h\tau}^{U}(t),Y_{h\tau}^{U}(t)\right)$ be the solutions solved by FBSPDEs in (\ref{first_continuous}) and (\ref{fullyfu_first_continuous}), respectively, then we have the following optimal estimates:
\begin{equation*}
\begin{aligned}
&\max_{0\leq n \leq N}\mathbb{E}\left[\|X(t_n)-X_{h\tau}^U(t_n)\|^2\right]+\max_{0\leq n\leq N}\mathbb{E}\left[\|Y(t_n)-Y_{h\tau}^U(t_n)\|^2\right]
\leq C(h^4+\tau^2),\\
&\tau\sum_{n=0}^{N-1}\mathbb{E}\left[\|\nabla\left(X(t_{n+1})-X_{h\tau}^U(t_{n+1})\right)\|^2+\|\nabla\left(Y(t_{n})-Y_{h\tau}^U(t_n)\right)\|^2\right]\leq C(h^2+\tau^2).
\end{aligned}
\end{equation*} 
\end{theorem}
\begin{proof}
By triangle inequality and Lemma \ref{optimal_state}, it holds that
\begin{equation*}
\begin{aligned}
&\max_{0\leq n \leq N}\mathbb{E}\left[\|X(t_n)-X_{h\tau}^U(t_n)\|^2\right]+\max_{0\leq n\leq N}\mathbb{E}\left[\|Y(t_n)-Y_{h\tau}^U(t_n)\|^2\right]\\
&\leq 2\max_{0\leq n \leq N}\mathbb{E}\left[\|X(t_n)-X_{h}^U(t_n)\|^2\right]+2\max_{0\leq n \leq N}\mathbb{E}\left[\|X_{h}^U(t_n)-X_{h\tau}^U(t_n)\|^2\right]\\
&\quad+2\max_{0\leq n\leq N}\mathbb{E}\left[\|Y(t_n)-Y_{h}^U(t_n)\|^2\right]+2\max_{0\leq n\leq N}\mathbb{E}\left[\|Y_{h}^U(t_n)-Y_{h\tau}^U(t_n)\|^2\right]\\
&\leq C\left(h^4+\max_{0\leq n \leq N}\mathbb{E}\left[\|X_{h}^U(t_n)-X_{h\tau}^U(t_n)\|^2\right]+\max_{0\leq n\leq N}\mathbb{E}\left[\|Y_{h}^U(t_n)-Y_{h\tau}^U(t_n)\|^2\right]  \right).
\end{aligned}
\end{equation*}
Using the regularity assumptions of $X(t),Y(t)$ and employing the property of projection $\Pi_{h}^{1}$, Lemmas \ref{pro_estimate}, \ref{inver_ertimate}, \ref{optimal_state} again yields
\begin{align*}
&\tau\sum_{n=0}^{N-1}\mathbb{E}\left[\|\nabla\left(X(t_{n+1})-X_{h\tau}^U(t_{n+1})\right)\|^2+\|\nabla\left(Y(t_{n})-Y_{h\tau}^U(t_n)\right)\|^2\right]\\
&\leq 3\tau\sum_{n=0}^{N-1}\mathbb{E}\left[\|\nabla\left(X(t_{n+1})-\Pi_{h}^{1}X(t_{n+1})\right)\|^2+\|\nabla\left(\Pi_{h}^{1}X(t_{n+1})-X_{h}^U(t_{n+1})\right)\|^2\right]\\
&\quad+ 3\tau\sum_{n=0}^{N-1}\mathbb{E}\left[\|\nabla\left(Y(t_{n})-\Pi_{h}^{1}Y(t_n)\right)\|^2+\|\nabla\left(\Pi_{h}^{1}Y(t_n)-Y_{h}^U(t_n)\right)\|^2\right]\\
&\quad+
3\tau\sum_{n=0}^{N-1}\mathbb{E}\left[\|\nabla\left(X_{h}^U(t_{n+1})-X_{h\tau}^U(t_{n+1})\right)\|^2+\|\nabla\left(Y_{h}^U(t_n)-Y_{h\tau}^U(t_n)\right)\|^2\right]\\
&\leq Ch^{2}+C\tau\sum_{n=0}^{N-1}h^{-2}\mathbb{E}\left[\|X(t_{n+1})-X_{h}^{U}(t_{n+1})\|^2+\|Y(t_{n})-Y_{h}^{U}(t_{n})\|^2\right]\\
&\quad+3\tau\sum_{n=0}^{N-1}\mathbb{E}\left[\|\nabla\left(X_{h}^U(t_{n+1})-X_{h\tau}^U(t_{n+1})\right)\|^2+\|\nabla\left(Y_{h}^U(t_n)-Y_{h\tau}^U(t_n)\right)\|^2\right]\\
&\leq C\left(h^2+\tau\sum_{n=0}^{N-1}\mathbb{E}\left[\|\nabla\left(X_{h}^U(t_{n+1})-X_{h\tau}^U(t_{n+1})\right)\|^2+\|\nabla\left(Y_{h}^U(t_n)-Y_{h\tau}^U(t_n)\right)\|^2\right]\right).
\end{align*}
The proof is completed in two steps: the first step gives
\begin{equation*}\label{fbspde_step1}
\begin{aligned}
\max_{0\leq n \leq N}\mathbb{E}\left[\|X_{h}^U(t_n)-X_{h\tau}^U(t_n)\|^2\right] +\tau\sum_{n=0}^{N-1}\mathbb{E}\left[\|\nabla\left(X_{h}^U(t_{n+1})-X_{h\tau}^U(t_{n+1})\right)\|^2\right]\leq C\tau^2,
\end{aligned}
\end{equation*}
and the second step is to prove
\begin{equation}\label{fbspde_step2}
\begin{aligned}
\max_{0\leq n\leq N}\mathbb{E}\left[\|Y_{h}^U(t_n)-Y_{h\tau}^U(t_n)\|^2\right] +\tau\sum_{n=0}^{N-1}\mathbb{E}\left[\|\nabla\left(Y_{h}^U(t_n)-Y_{h\tau}^U(t_n)\right)\|^2\right]\leq C\tau^2.
\end{aligned}
\end{equation}

\textbf{Step 1.} Let $e_{X}^{n}:=X_{h}^{U}(t_n)-X_{h\tau}^{U}(t_n),n=0,1,...,N$ and subtracting the FSPDE in (\ref{fullyfu_first_continuous}) from FSPDE in (\ref{semifu_first_continuous}) yields
\begin{equation*}
\begin{aligned}
&e_{X}^{n+1}-e_{X}^{n}-\int_{t_n}^{t_{n+1}} \Delta_he_{X}^{n+1}dt\\
&=\int_{t_n}^{t_{n+1}}\Delta_h\left(X_{h}^{U}(t)-X_{h}^{U}(t_{n+1})\right)dt
+\int_{t_n}^{t_{n+1}}\Pi_{h}^{1}\left(U(t)-U(t_n)\right)dt\\
&\quad+\int_{t_n}^{t_{n+1}}\Pi_{h}^{1}\left(\sigma(t)-\sigma(t_n)\right)dW_t.
\end{aligned}
\end{equation*}
For a fixed $\omega\in\Omega$, performing the inner product of both sides of the above equation with $e_{X}^{n+1}$ on $\mathcal{D}$ and using the fact that 
\begin{equation*}
\begin{aligned}
(e_{X}^{n+1},e_{X}^{n+1})-(e_{X}^{n},e_{X}^{n+1})=\frac{1}{2}\left(\|e_{X}^{n+1}\|^2-\|e_{X}^{n}\|^2+\|e_{X}^{n+1}-e_{X}^{n}\|^2\right)
\end{aligned}
\end{equation*}
yields
\begin{equation}\label{estimate_eX0}
\begin{aligned}
&\frac{1}{2}\mathbb{E}\left[\|e_{X}^{n+1}\|^2-\|e_{X}^{n}\|^2+\|e_{X}^{n+1}-e_{X}^{n}\|^2+2\tau\|\nabla e_{X}^{n+1}\|^2\right]\\
&=\mathbb{E}\left[\int_{\mathcal{D}}\int_{t_n}^{t_{n+1}}\nabla\left(X_{h}^{U}(t_{n+1})-X_{h}^{U}(t)\right)\cdot\nabla e_{X}^{n+1}dt dx\right]\\
&\quad+\mathbb{E}\left[\int_{\mathcal{D}}\int_{t_n}^{t_{n+1}}\left(U(t)-U(t_n)\right)dte_{X}^{n+1}dx\right]\\
&\quad+\mathbb{E}\left[\int_{\mathcal{D}}\int_{t_n}^{t_{n+1}}\sigma(t)-\sigma(t_n)dW_te_{X}^{n+1}  dx\right]\\
&=T_{11}+T_{12}+T_{13}.
\end{aligned}
\end{equation}
Like the analysis in Lemma \ref{regularity}, the solution $X_h^U(t)$ of FSPDE in (\ref{semifu_first_continuous}) belongs to $L^2_{\mathbb{F}}\left(\Omega;C\left(0,T;H_0^1(\mathcal{D})\cap H^2(\mathcal{D}) \right)\cap L^2\left(0,T;H_{0}^{1}(\mathcal{D})\cap H^{3}(\mathcal{D})\right)\right)$ and for $t\in[t_n,t_{n+1}]$ it follows that
\begin{equation*}
\begin{aligned}
\nabla X_{h}^{U}(t_{n+1})-\nabla X_{h}^{U}(t)=\int_{t}^{t_{n+1}}\nabla \Delta_hX_{h}^{U}(s)+\nabla \Pi_{h}^{1}U(s)ds+\int_{t}^{t_{n+1}}\nabla\Pi_{h}^{1}\sigma(s)dW_s,
\end{aligned}
\end{equation*}
which implies that
\begin{equation}\label{estimate_eX1}
\begin{aligned}
&\nabla \left(X_{h}^{U}(t_{n+1})-X_{h}^{U}(t)\right)\cdot\nabla e_{X}^{n+1}\\
&=\int_{t}^{t_{n+1}}\nabla \Delta_hX_{h}^{U}(s)+\nabla \Pi_{h}^{1}U(s)ds\cdot\nabla e_{X}^{n+1}+\int_{t}^{t_{n+1}}\nabla\Pi_{h}^{1}\sigma(s)dW_s\cdot\nabla e_{X}^{n+1}.
\end{aligned}
\end{equation}
By (\ref{estimate_eX1}), the Cauchy-Schwarz inequality, Young's inequality and the independence, the following estimate for $T_{11}$ holds:
\begin{align*}
T_{11}&=\mathbb{E}\left[\int_{\mathcal{D}}\int_{t_n}^{t_{n+1}}\int_{t}^{t_{n+1}}\nabla \Delta_hX_{h}^{U}(s)+\nabla \Pi_{h}^{1}U(s)dsdt\cdot\nabla e_{X}^{n+1}dx\right]\\
&\quad+\mathbb{E}\left[\int_{\mathcal{D}}\int_{t_n}^{t_{n+1}}\int_{t}^{t_{n+1}}\nabla \Pi_{h}^{1}\sigma(s)dW_sdt\cdot\nabla\left(e_{X}^{n+1}-e_{X}^{n}\right)dx\right]\\
&\leq \frac{1}{2}\tau\mathbb{E}\left[\int_{t_n}^{t_{n+1}}\int_{t}^{t_{n+1}}\left\|\nabla \Delta_hX_{h}^{U}(s)+\nabla \Pi_{h}^{1}U(s)\right\|^2dsdt\right]+\frac{\tau}{2}\mathbb{E}\left[\|\nabla e_{X}^{n+1}\|^2\right]\\
&\quad+2\tau\mathbb{E}\left[\int_{t_n}^{t_{n+1}}\int_{t}^{t_{n+1}}\left\|\Delta_h \Pi_{h}^{1}\sigma(s)\right\|^2dsdt   \right]
+\frac{1}{8}\mathbb{E}\left[\|e_{X}^{n+1}-e_{X}^{n}\|^2\right]\\
&\leq\tau^{2}\mathbb{E}\left[\int_{t_n}^{t_{n+1}}\|\nabla \Delta_hX_{h}^{U}(t)\|^2+\|\nabla \Pi_{h}^{1}U(t)\|^2  dt\right]+\frac{\tau}{2}\mathbb{E}\left[\|\nabla e_{X}^{n+1}\|^2\right]\\
&\quad+2\tau^2\mathbb{E}\left[\int_{t_n}^{t_{n+1}}\|\Delta_h \Pi_{h}^{1}\sigma(t)\|^2dt\right]
+\frac{1}{8}\mathbb{E}\left[\|e_{X}^{n+1}-e_{X}^{n}\|^2\right].
\end{align*}
Analogously, term $T_{13}$ can be estimated as
\begin{equation*}
\begin{aligned}
T_{13}&=\mathbb{E}\left[\int_{\mathcal{D}}\int_{t_n}^{t_{n+1}}\sigma(t)-\sigma(t_n)dW_te_{X}^{n+1}  dx\right]\\
&\leq2\mathbb{E}\left[\int_{t_n}^{t_{n+1}}\|\sigma(t)-\sigma(t_n)\|^2dt\right]
+\frac{1}{8}\mathbb{E}\left[\|e_{X}^{n+1}-e_{X}^{n}\|^2\right].
\end{aligned}
\end{equation*}
Using the definition of BSPDE in (\ref{adjoint}) and $U(t)=-\frac{1}{\alpha}Y(t)$ in (\ref{first_continuous}) gives
\begin{equation*}
\begin{aligned}
T_{12}&=\frac{1}{\alpha}\mathbb{E}\left[\int_{\mathcal{D}}\int_{t_n}^{t_{n+1}}\left(Y(t_n)-Y(t)\right)e_{X}^{n+1}dtdx\right]\\
&=\frac{1}{\alpha}\mathbb{E}\left[\int_{\mathcal{D}}\int_{t_n}^{t_{n+1}}\int_{t_n}^{t}\Delta Y(s)+X(s)-X_d(s)+\mu dsdte_{X}^{n+1}dx\right]\\
&\quad-\frac{1}{\alpha}\mathbb{E}\left[\int_{\mathcal{D}}\int_{t_n}^{t_{n+1}}\int_{t_n}^{t}Z(s)dW_sdte_{X}^{n+1}dx   \right]\\
&\leq\frac{1}{\alpha}\left(\frac{4\tau^2}{\alpha}\mathbb{E}\left[\int_{t_n}^{t_{n+1}}\|\Delta Y(t)+X(t)-X_d(t)+\mu\|^2dt\right]+\frac{\alpha}{8}\mathbb{E}\left[\|e_{X}^{n+1}-e_{X}^{n}\|^2\right] \right.\\
&\left.\quad+\frac{\tau\alpha}{8}\mathbb{E}\left[\|e_{X}^{n+1}-e_{X}^{n}\|^2+\|e_{X}^{n}\|^2\right] 
+\frac{2\tau^2}{\alpha}\mathbb{E}\left[\int_{t_n}^{t_{n+1}}\|Z(t)\|^2dt\right]\right).
\end{aligned}
\end{equation*}
Combining the estimates of $T_{11},T_{12}$ and $T_{13}$, (\ref{estimate_eX0}) can be bounded as
\begin{align*}
&\frac{1}{2}\mathbb{E}\left[\|e_{X}^{n+1}\|^2-\|e_{X}^{n}\|^2+\|e_{X}^{n+1}-e_{X}^{n}\|^2+2\tau\|\nabla e_{X}^{n+1}\|^2\right]\\
&\leq\tau^{2}\mathbb{E}\left[\int_{t_n}^{t_{n+1}}\|\nabla \Delta_hX_{h}^{U}(t)\|^2+\|\nabla \Pi_{h}^{1}U(t)\|^2  dt\right]+\frac{\tau}{2}\mathbb{E}\left[\|\nabla e_{X}^{n+1}\|^2\right]\\
&\quad+2\tau^2\mathbb{E}\left[\int_{t_n}^{t_{n+1}}\|\Delta_h \Pi_{h}^{1}\sigma(t)\|^2dt\right]
+\frac{3+\tau}{8}\mathbb{E}\left[\|e_{X}^{n+1}-e_{X}^{n}\|^2\right]\\
&\quad+\frac{4\tau^2}{\alpha^2}\mathbb{E}\left[\int_{t_n}^{t_{n+1}}\|\Delta Y(t)+X(t)-X_d(t)+\mu\|^2dt\right]
+\frac{\tau}{8}\mathbb{E}\left[\|e_{X}^{n}\|^2\right] \\
&\quad+\frac{2\tau^2}{\alpha^2}\mathbb{E}\left[\int_{t_n}^{t_{n+1}}\|Z(t)\|^2dt\right]
+2\mathbb{E}\left[\int_{t_n}^{t_{n+1}}\|\sigma(t)-\sigma(t_n)\|^2dt\right],
\end{align*}
which implies that 
\begin{equation}
\begin{aligned}\label{estimate_eX2}
&\mathbb{E}\left[\|e_{X}^{n+1}\|^2\right]+\tau\mathbb{E}\left[\|\nabla e_{X}^{n+1}\|^2\right]\\
&\leq\left(4+\frac{8}{\alpha^2}\right)\left(\tau^2\mathbb{E}\left[\int_{t_n}^{t_{n+1}}\|\nabla \Delta_hX_{h}^{U}(t)\|^2+\|\nabla \Pi_{h}^{1}U(t)\|^2  dt\right] \right. \\
&\left.\quad+\tau^2\mathbb{E}\left[\int_{t_n}^{t_{n+1}}\|\Delta_h \Pi_{h}^{1}\sigma(t)\|^2+\|\Delta Y(t)+X(t)-X_d(t)+\mu\|^2dt\right] \right. \\
&\left.\quad+\tau^2\mathbb{E}\left[\int_{t_n}^{t_{n+1}}\|Z(t)\|^2dt\right]+\mathbb{E}\left[\int_{t_n}^{t_{n+1}}\|\sigma(t)-\sigma(t_n)\|^2dt\right]\right) \\ 
&\quad+(1+\tau)\mathbb{E}\left[\|e_{X}^{n}\|^2\right].
\end{aligned}
\end{equation}
Subsequently, the discrete Gronwall inequality leads to
\begin{equation}\label{estimate_eX3}
\begin{aligned}
&\max_{0\leq n\leq N}\mathbb{E}\left[\|e_{X}^{n}\|^2\right]\\
&\leq\left(4+\frac{8}{\alpha^2}\right)e^{T}\left(\tau^2\mathbb{E}\left[\int_{0}^{T}\|\nabla \Delta_hX_{h}^{U}(t)\|^2+\|\nabla \Pi_{h}^{1}U(t)\|^2+\|\Delta_h \Pi_{h}^{1}\sigma(t)\|^2dt\right] \right.\\
&\left.\quad+\tau^2\mathbb{E}\left[\int_{0}^{T}\|\Delta Y(t)+X(t)-X_d(t)+\mu\|^2+\|Z(t)\|^2dt\right]\right.\\
&\left.\quad+\sum_{n=0}^{N-1}\mathbb{E}\left[\int_{t_n}^{t_{n+1}}\|\sigma(t)-\sigma(t_n)\|^2dt\right]\right).
\end{aligned}
\end{equation}
Applying It\^{o} formula for $\|\Delta_hX_h^{U}(t)\|^2$ gives
\begin{equation*}
\begin{aligned}\label{estimate_eX4}
\int_{0}^{T}\mathbb{E}\left[\|\nabla\Delta_h X_{h}^{U}(t)\|^2\right]dt
\leq\|\Delta_h\Pi_{h}^{1}X_0\|^2+\int_{0}^{T}\mathbb{E}\left[\|\Delta_h\Pi_{h}^{1}\sigma(t)\|^2+\|\nabla\Pi_{h}^{1}U(t)\|^2\right]dt.
\end{aligned}
\end{equation*}
Based on the above inequality, (\ref{sigma_xingzhi}), (\ref{estimate_eX3}) and the regularities of all variables, it follows that
\begin{equation*}
\begin{aligned}
&\max_{0\leq n\leq N}\mathbb{E}\left[\|e_{X}^{n}\|^2\right]\\
&\leq\left(8+\frac{16}{\alpha^2}\right)e^{T}\left(\tau^2\|\Delta_h\Pi_{h}^{1}X_0\|^2+\tau^2\mathbb{E}\left[\int_{0}^{T}\|\nabla \Pi_{h}^{1}U(t)\|^2+\|\Delta_h \Pi_{h}^{1}\sigma(t)\|^2 \right.\right.\\
&\left.\left.\quad+\|\Delta Y(t)+X(t)-X_d(t)+\mu\|^2+\|Z(t)\|^2dt\right]
+\sum_{n=0}^{N-1}\mathbb{E}\left[\int_{t_n}^{t_{n+1}}\|\sigma(t)-\sigma(t_n)\|^2dt\right]\right)\\
&\leq C\tau^2.
\end{aligned}
\end{equation*}
Then, summing up over all steps of (\ref{estimate_eX2}) yields
\begin{align*}
&\tau\sum_{n=0}^{N-1}\mathbb{E}\left[\|\nabla e_{X}^{n+1}\|^2\right]\\
&\leq\left(4+\frac{8}{\alpha^2}\right)\left(\tau^2\mathbb{E}\left[\int_{0}^{T}\|\nabla \Delta_hX_{h}^{U}(t)\|^2+\|\nabla \Pi_{h}^{1}U(t)\|^2+\|\Delta_h \Pi_{h}^{1}\sigma(t)\|^2 \right.\right.\\
&\left.\left.\quad+\|\Delta Y(t)+X(t)-X_d(t)+\mu\|^2+\|Z(t)\|^2dt\right]
+\sum_{n=0}^{N-1}\mathbb{E}\left[\int_{t_n}^{t_{n+1}}\|\sigma(t)-\sigma(t_n)\|^2dt\right]\right)\\
&\quad+\max_{0\leq n \leq N}\mathbb{E}\left[\|e_{X}^{n}\|^2\right]\\
&\leq C\tau^2.
\end{align*}
Combining the estimates of $\max\limits_{0\leq n\leq N}\mathbb{E}\left[\|e_{X}^{n}\|^2\right]$ and $\tau\sum\limits_{n=0}^{N-1}\mathbb{E}\left[\|\nabla e_{X}^{n+1}\|^2\right]$ yields the result.

\textbf{Step 2.} The proof of (\ref{fbspde_step2}) follows a similar process to that in \textbf{Step 1}. Here, we present a simplified proof for clarity. Let $e_{Y}^{n}:=Y_{h}^{U}(t_n)-\hat{Y}_{h\tau}(t_n)$ and it holds that
\begin{equation*}
\begin{aligned}
&e_{Y}^{n}-e_{Y}^{n+1}-\int_{t_n}^{t_{n+1}}\Delta_he_{Y}^{n}dt\\
&=\int_{t_n}^{t_{n+1}}\Delta_h\left(Y_h^{U}(t)-Y_h^{U}(t_n)\right)+
\Pi_{h}^{1}\left(X(t)-X_d(t)-X(t_{n+1})+X_d(t_{n+1})\right)dt\\
&\quad+\int_{t_n}^{t_{n+1}}\hat{Z}_{h\tau}(t)-Z_h^{U}(t)dW_t.
\end{aligned}
\end{equation*}
For a fixed $\omega\in\Omega$, testing with the admissible $e_{Y}^{n}(\omega)$, taking expectation and Young's inequality leads to
\begin{equation*}
\begin{aligned}
&\frac{1}{2}\mathbb{E}\left[\|e_{Y}^{n}\|^2-\|e_{Y}^{n+1}\|^2+\|e_{Y}^{n}-e_{Y}^{n+1}\|^2+2\tau\|\nabla e_{Y}^{n}\|^2 \right]\\
&=\mathbb{E}\left[\int_{\mathcal{D}}\int_{t_n}^{t_{n+1}}\nabla\left(Y_{h}^{U}(t_n)-Y_{h}^{U}(t)\right)\cdot\nabla e_{Y}^{n}dtdx\right]+\mathbb{E}\left[\int_{\mathcal{D}}\int_{t_n}^{t_{n+1}}X(t)-X(t_{n+1})dte_{Y}^{n}dx\right]\\
&\quad+\mathbb{E}\left[\int_{\mathcal{D}}\int_{t_n}^{t_{n+1}}X_d(t_{n+1})-X_d(t)dte_{Y}^{n}dx\right]\\
&=\mathbb{E}\left[\int_{\mathcal{D}}\int_{t_n}^{t_{n+1}}\int_{t_n}^{t}\nabla\Delta_hY_{h}^{U}(s)+\nabla\Pi_{h}^{1}\left(X(s)-X_d(s)+\mu\right)dsdt\cdot\nabla e_{Y}^{n}dx\right]\\
&\quad+\mathbb{E}\left[\int_{\mathcal{D}}\int_{t_{n}}^{t_{n+1}}\int_{t}^{t_{n+1}}-\Delta X(s)-U(s)dsdte_{Y}^{n}dx\right]\\
&\quad+\mathbb{E}\left[\int_{\mathcal{D}}\int_{t_n}^{t_{n+1}}X_d(t_{n+1})-X_d(t)dte_{Y}^{n}dx\right]\\
&=T_{21}+T_{22}+T_{23}.
\end{aligned}
\end{equation*}
By the Cauchy-Schwarz and Young's inequalities, we have
\begin{align*}
T_{21}&\leq\tau^2\mathbb{E}\left[\int_{t_n}^{t_{n+1}}\|\nabla\Delta_hY_{h}^{U}(t)\|^2+\|\nabla\Pi_{h}^{1}\left(X(t)-X_d(t)+\mu\right)\|^2dt\right]+\frac{\tau}{2}\mathbb{E}\left[\|\nabla e_{Y}^{n}\|^2\right],\\
T_{22}&\leq 2\tau^2\mathbb{E}\left[\int_{t_n}^{t_{n+1}}\|\Delta X(t)+U(t)\|^2dt\right]+\frac{\tau}{4}\mathbb{E}\left[\|e_{Y}^{n}-e_{Y}^{n+1}\|^2+\|e_{Y}^{n+1}\|^2\right],\\
T_{23}&\leq2\mathbb{E}\left[\int_{t_n}^{t_{n+1}}\|X_d(t_{n+1})-X_d(t)\|^2dt\right]
+\frac{\tau}{4}\mathbb{E}\left[\|e_{Y}^{n}-e_{Y}^{n+1}\|^2+\|e_{Y}^{n+1}\|^2\right],
\end{align*}
which implies that 
\begin{equation}\label{estimate_eY0}
\begin{aligned}
&\mathbb{E}\left[\|e_{Y}^{n}\|^2\right]+\tau\mathbb{E}\left[\|\nabla e_{Y}^{n}\|^2\right]\\
&\leq 2\tau^2\mathbb{E}\left[\int_{t_n}^{t_{n+1}}\|\nabla\Delta_hY_{h}^{U}(t)\|^2+\|\nabla\Pi_{h}^{1}\left(X(t)-X_d(t)+\mu\right)\|^2dt\right]\\
&\quad+4\tau^2\mathbb{E}\left[\int_{t_n}^{t_{n+1}}\|\Delta X(t)+U(t)\|^2dt\right]+4\mathbb{E}\left[\int_{t_n}^{t_{n+1}}\|X_d(t_{n+1})-X_d(t)\|^2dt\right]\\
&\quad+(1+\tau)\mathbb{E}\left[\|e_{Y}^{n+1}\|^2\right].
\end{aligned}
\end{equation}
Note that applying It\^{o} formula for $\|\Delta_hY_h^{U}(t)\|^2$ yields
\begin{equation}\label{estimate_eY1}
\begin{aligned}
\int_{0}^{T}\mathbb{E}\left[\|\nabla\Delta_h Y_{h}^{U}(t)\|^2\right]dt
\leq\int_{0}^{T}\mathbb{E}\left[\|\nabla\Pi_{h}^{1}\left(X(t)-X_d(t)+\mu\right)\|^2 \right]dt.
\end{aligned}
\end{equation}
Performing the analysis like \textbf{Step 1} and combining (\ref{estimate_eY0}), (\ref{estimate_eY1}), (\ref{Fuzhu2}) as well as (\ref{sigma_xingzhi}) yields the result (\ref{fbspde_step2}).
\end{proof}

\subsection{Optimal error estimates of control problem}\label{error_control}
In this subsection, the optimal error estimates for SOCP are deduced. Previously, for the FBSPDEs in fully discrete first-order condition (\ref{fully_first_continuous}) and auxiliary problem (\ref{fullyfu_first_continuous}), the following estimates are given.
\begin{theorem}\label{full_fullfu_XY}
Let $\left(X_{h\tau}(t),Y_{h\tau}(t),Z_{h\tau}(t)\right)$ and $\left(X_{h\tau}^U(t),Y_{h\tau}^U(t),Z_{h\tau}^U(t)\right)$ be the solutions of FBSPDEs in (\ref{fully_first_continuous}) and (\ref{fullyfu_first_continuous}), respectively, then the following estimates hold:
\begin{subequations}
\begin{align}
&\max_{0\leq n\leq N}\mathbb{E}\left[\|X_{h\tau}(t_{n})-X_{h\tau}^U(t_{n})\|^2\right]+\tau\sum_{n=0}^{N-1}\mathbb{E}\left[\|\nabla\left(X_{h\tau}(t_{n+1})-X_{h\tau}^U(t_{n+1})\right)\|^2\right] \label{full_fullfu_XY1}\\
&\leq 2\tau\sum_{n=0}^{N-1}\mathbb{E}\left[\|U_{h\tau}(t_n)-U(t_n)\|^2\right], \nonumber\\
&\max_{0\leq n\leq N}\mathbb{E}\left[\|Y_{h\tau}(t_{n})-Y_{h\tau}^U(t_{n})\|^2\right]+\tau\sum_{n=0}^{N-1}\mathbb{E}\left[\|\nabla\left(Y_{h\tau}(t_{n})-Y_{h\tau}^U(t_{n})\right)\|^2\right] \label{full_fullfu_XY2}\\
&\leq C\left(\tau\sum_{n=0}^{N-1}\mathbb{E}\left[\|U_{h\tau}(t_{n})-U(t_{n})\|^2+\|X_{h\tau}^{U}(t_{n+1})-X(t_{n+1})\|^2 \right]+|\mu_{h\tau}-\mu|^2\right). \nonumber
\end{align}
\end{subequations}
\end{theorem}
\begin{proof}
The difference between (\ref{fully_first_continuous}) and (\ref{fullyfu_first_continuous}) has the following weak form for any $\phi_h,\psi_h\in\mathbb{V}_h^1$:
\begin{equation}\label{lisan_fuzhu}
\left\{\begin{aligned}
&\left(X_{h\tau}(t_{n+1})-X_{h\tau}^U(t_{n+1})-\left(X_{h\tau}(t_{n})-X_{h\tau}^U(t_{n})\right), \phi_h\right)\\
&=\tau\left(\nabla X_{h\tau}^U(t_{n+1})-\nabla X_{h\tau}(t_{n+1}),\nabla\phi_h \right)
+\tau\left( U_{h\tau}(t_n)-U(t_n),\phi_h\right),\\
&X_{h\tau}(0)-X_{h\tau}^U(0)=0,\\
&\left(\mathbb{E}\left[Y_{h\tau}(t_{n+1})-Y_{h\tau}^U(t_{n+1}) |\mathcal{F}_{t_n}\right]-\left(Y_{h\tau}(t_{n})-Y_{h\tau}^U(t_{n})\right), \psi_h\right)\\
&=\tau\left(\nabla Y_{h\tau}(t_{n})-\nabla Y_{h\tau}^U(t_{n}),\nabla\psi_h\right)+\tau\left(\mathbb{E}\left[X(t_{n+1})+\mu-X_{h\tau}(t_{n+1})-\mu_{h\tau}|\mathcal{F}_{t_n} \right],\psi_h\right),\\
&Y_{h\tau}(T)-Y_{h\tau}^U(T)=0.
\end{aligned}\right.
\end{equation}
We set $\theta_{X}^{n+1}:=X_{h\tau}(t_{n+1})-X_{h\tau}^U(t_{n+1})$ and $\phi_h=\theta_{X}^{n+1}$ in (\ref{lisan_fuzhu}) for a fixed $\omega\in\Omega$. By the Cauchy-Schwarz inequality it holds that
\begin{equation*}
\begin{aligned}
\|\theta_{X}^{n+1}\|^2+\tau\|\nabla\theta_{X}^{n+1}\|^2
\leq \|\theta_{X}^{n+1}\|\|\theta_{X}^{n}\|+\tau\|U_{h\tau}(t_n)-U(t_n)\| \|\theta_{X}^{n+1}\|,
\end{aligned}
\end{equation*}
which implies that
\begin{equation*}
\begin{aligned}
\|\theta_{X}^{n+1}\|&\leq \|\theta_{X}^{n}\|+\tau\|U_{h\tau}(t_n)-U(t_n)\|
\leq \sum_{j=0}^{n}\tau \|U_{h\tau}(t_j)-U(t_j)\|.
\end{aligned}
\end{equation*}
Therefore, we have
\begin{equation}\label{full_fullfu_X}
\begin{aligned}
\max_{0\leq n\leq N}\mathbb{E}\left[\|X_{h\tau}(t_{n})-X_{h\tau}^U(t_{n})\|^2\right]\leq\tau\sum_{n=0}^{N-1}\mathbb{E}\left[\|U_{h\tau}(t_n)-U(t_n)\|^2 \right] .
\end{aligned}
\end{equation}
Let $\phi_h=\theta_{X}^{n+1}-\theta_{X}^{n}$, it follows that
\begin{equation*}
\begin{aligned}
\|\theta_{X}^{n+1}-\theta_{X}^{n}\|^2+\tau(\nabla\theta_{X}^{n+1},\nabla\theta_{X}^{n+1}-\nabla\theta_{X}^{n} )=\tau(U_{h\tau}(t_n)-U(t_n),\theta_{X}^{n+1}-\theta_{X}^{n}).
\end{aligned}
\end{equation*}
Using the fact that 
\begin{equation*}
\begin{aligned}
(\nabla\theta_{X}^{n+1},\nabla\theta_{X}^{n+1}-\nabla\theta_{X}^{n})=\frac{1}{2}\left[\|\nabla\theta_{X}^{n+1} \|^2 -\|\nabla\theta_{X}^{n}\|^2+\|\nabla\theta_{X}^{n+1}-\nabla\theta_{X}^{n}\|^2 \right]
\end{aligned}
\end{equation*}
yields
\begin{equation*}
\begin{aligned}
&\|\theta_{X}^{n+1}-\theta_{X}^{n}\|^2+\frac{\tau}{2}\left[\|\nabla\theta_{X}^{n+1} \|^2 -\|\nabla\theta_{X}^{n}\|^2+\|\nabla\theta_{X}^{n+1}-\nabla\theta_{X}^{n}\|^2 \right]\\
&\leq \frac{\tau^2}{4}\|U_{h\tau}(t_n)-U(t_n)\|^2+\|\theta_{X}^{n+1}-\theta_{X}^{n}\|^2,
\end{aligned}
\end{equation*}
which implies that 
\begin{equation*}
\begin{aligned}
\|\nabla\theta_{X}^{n+1}\|^2\leq \|\nabla\theta_{X}^{n}\|^2+\frac{\tau}{2}\|U_{h\tau}(t_n)-U(t_n)\|^2\leq \sum_{j=0}^{n}\frac{\tau}{2} \|U_{h\tau}(t_j)-U(t_j)\|^2,
\end{aligned}
\end{equation*}
i.e.,
\begin{equation}\label{full_fullfu_X1}
\begin{aligned}
\tau\sum_{n=0}^{N-1}\mathbb{E}\left[\|\nabla\left(X_{h\tau}(t_{n+1})-X_{h\tau}^U(t_{n+1})\right)\|^2   \right]\leq\tau\sum_{n=0}^{N-1}\mathbb{E}\left[\|U_{h\tau}(t_n)-U(t_n)\|^2\right].
\end{aligned}
\end{equation}
Combining (\ref{full_fullfu_X}) and (\ref{full_fullfu_X1}) gives the result (\ref{full_fullfu_XY1}).

We define $\eta_{Y}^{n+1}:=Y_{h\tau}(t_{n+1})-Y_{h\tau}^U(t_{n+1})$ and set $\psi_h=\eta_{Y}^{n}$ in (\ref{lisan_fuzhu}), leading to
\begin{equation*}
\begin{aligned}
&\mathbb{E}\left[(\eta_{Y}^{n+1}-\eta_{Y}^{n},\eta_{Y}^{n})|\mathcal{F}_{t_n} \right]\\
&=\tau(\nabla\eta_{Y}^{n},\nabla\eta_{Y}^{n})
+\tau\mathbb{E}\left[\left(X(t_{n+1})+\mu-X_{h\tau}(t_{n+1})-\mu_{h\tau},\eta_{Y}^{n}\right)|\mathcal{F}_{t_n} \right].
\end{aligned}
\end{equation*}
By the fact that $\left(\eta_{Y}^{n+1}-\eta_{Y}^{n},\eta_{Y}^{n}\right)=\frac{1}{2}\left[\|\eta_{Y}^{n+1} \|^2 -\|\eta_{Y}^{n}\|^2-\|\eta_{Y}^{n+1}-\eta_{Y}^{n}\|^2 \right]$,
it holds that
\begin{align*}
&\frac{1}{2}\mathbb{E}\left[2\tau\|\nabla\eta_{Y}^n\|^2+\|\eta_{Y}^n\|^2+\|\eta_{Y}^{n+1}-\eta_{Y}^{n}\|^2 |\mathcal{F}_{t_n} \right]\\
&=\tau\mathbb{E}\left[\left(X_{h\tau}(t_{n+1})+\mu_{h\tau}-X(t_{n+1})-\mu,\eta_{Y}^{n}\right)|\mathcal{F}_{t_n} \right]
+\frac{1}{2}\mathbb{E}\left[\|\eta_{Y}^{n+1}\|^2|\mathcal{F}_{t_n} \right]\\
&\leq\tau \mathbb{E}\left[\|X_{h\tau}(t_{n+1})+\mu_{h\tau}-X(t_{n+1})-\mu\|^2 |\mathcal{F}_{t_n} \right]+\frac{\tau}{4} \|\eta_{Y}^{n}\|^2+\frac{1}{2}\mathbb{E}\left[\|\eta_{Y}^{n+1}\|^2|\mathcal{F}_{t_n} \right]\\
&\leq \tau \mathbb{E}\left[\|X_{h\tau}(t_{n+1})+\mu_{h\tau}-X(t_{n+1})-\mu\|^2|\mathcal{F}_{t_n} \right]+\frac{\tau}{2}\mathbb{E}\left[\|\eta_{Y}^{n+1}-\eta_{Y}^{n}\|^2|\mathcal{F}_{t_n} \right]\\
&\quad+\frac{1+\tau}{2}\mathbb{E}\left[\|\eta_{Y}^{n+1}\|^{2}|\mathcal{F}_{t_n} \right].
\end{align*}
Taking the expectation on both sides of the above equation yields
\begin{equation}\label{dis_gronwall}
\begin{aligned}
&\mathbb{E}\left[2\tau\|\nabla\eta_{Y}^n\|^2+\|\eta_{Y}^n\|^2+\|\eta_{Y}^{n+1}-\eta_{Y}^{n}\|^2\right]\\
&\leq 2\tau \mathbb{E}\left[\|X_{h\tau}(t_{n+1})+\mu_{h\tau}-X(t_{n+1})-\mu\|^2 \right]
+\tau\mathbb{E}\left[\|\eta_{Y}^{n+1}-\eta_{Y}^{n}\|^2 \right]\\
&\quad+(1+\tau)\mathbb{E}\left[\|\eta_{Y}^{n+1}\|^2\right].
\end{aligned}
\end{equation}
Applying the discrete Gronwall inequality and (\ref{full_fullfu_X}) leads to
\begin{align*}
&\max_{0\leq n\leq N}\mathbb{E}\left[\|\eta_{Y}^n\|^2\right]\\
&\leq 2e^{T}\tau\sum_{n=0}^{N-1}\mathbb{E}\left[\|X_{h\tau}(t_{n+1})+\mu_{h\tau}-X(t_{n+1})-\mu\|^2 \right]\\
&\leq C\left(\tau\sum_{n=0}^{N-1}\mathbb{E}\left[\|X_{h\tau}(t_{n+1})-X(t_{n+1})\|^2 \right] + |\mu_{h\tau}-\mu|^2\right)\\
&\leq C\left(\tau\sum_{n=0}^{N-1}\mathbb{E}\left[\|X_{h\tau}(t_{n+1})-X_{h\tau}^{U}(t_{n+1})\|^2+\|X_{h\tau}^{U}(t_{n+1})-X(t_{n+1})\|^2 \right]+|\mu_{h\tau}-\mu|^2   \right)\\
&\leq C\left(\tau\sum_{n=0}^{N-1}\mathbb{E}\left[\|U_{h\tau}(t_{n})-U(t_{n})\|^2+\|X_{h\tau}^{U}(t_{n+1})-X(t_{n+1})\|^2 \right]+|\mu_{h\tau}-\mu|^2\right).
\end{align*}
Summing from $n=0$ to $N-1$ for (\ref{dis_gronwall}) yields
\begin{align*}
&\tau\sum_{n=0}^{N-1}\mathbb{E}\left[\|\nabla \eta_{Y}^n\|^2\right]\\
&\leq\tau\sum_{n=0}^{N-1}\mathbb{E}\left[\|\eta_{Y}^{n+1}\|^2\right]
+2\tau\sum_{n=0}^{N-1}\mathbb{E}\left[\|X_{h\tau}(t_{n+1})+\mu_{h\tau}-X(t_{n+1})-\mu\|^2 \right]\\
&\leq C\left(\tau\sum_{n=0}^{N-1}\mathbb{E}\left[\|U_{h\tau}(t_{n})-U(t_{n})\|^2+\|X_{h\tau}^{U}(t_{n+1})-X(t_{n+1})\|^2 \right]+|\mu_{h\tau}-\mu|^2\right).
\end{align*}
The desired result (\ref{full_fullfu_XY2}) is obtained by combining the above estimates.
\end{proof}
Next the upper bound on the multiplier error is estimated. Before that, the following deterministic parabolic equation is introduced:
\begin{equation}\label{varphi_strong}
\left\{\begin{aligned}
&\varphi_{t}(t)=\Delta \varphi(t)+1,\ t\in(0,T],\\
&\varphi(0)=0.
\end{aligned}\right.
\end{equation}
The weak form of (\ref{varphi_strong}) is as follows:
\begin{equation}\label{varphi_weak}
\left\{\begin{aligned}
&\left(\varphi_{t}(t),\phi\right)+\left(\nabla\varphi(t),\nabla\phi\right)=(1,\phi),\ \forall\phi\in H_{0}^{1}(\mathcal{D}),\\
&\varphi(0)=0.
\end{aligned}\right.
\end{equation}
According to Chapter 7.1.3 in \cite{evans}, there exists a unique weak solution and it holds that $\varphi(t)\in L^2\left(0,T;H^4(\mathcal{D})\right)\cap L^{\infty}\left(0,T;H^2(\mathcal{D})\right)$, $\varphi_{t}(t)\in L^2\left(0,T;H^2(\mathcal{D})\right)$ and $\varphi_{t,t}(t)\in L^2\left(0,T;L^2(\mathcal{D})\right)$. Further, for $t\in[t_n,t_{n+1}]$, we have
\begin{equation}\label{varphi_error1}
\begin{aligned}
\sum_{n=0}^{N-1}\int_{t_n}^{t_{n+1}}\|\varphi(t_{n+1})-\varphi(t)\|^2dt\leq2\tau^2\left(\int_{0}^{T}\|\Delta \varphi(t)\|^2dt+T|\mathcal{D}|\right)\leq C\tau^2.
\end{aligned}
\end{equation}
The space-time fully discrete approximation of (\ref{varphi_weak}) reads: Find $\varphi_{h\tau}(t)\in\mathbb{R}_{h\tau}$ such that
\begin{equation}\label{varphi_fully}
\left\{\begin{aligned}
&\left(\varphi_{h\tau}(t_{n+1})-\varphi_{h\tau}(t_{n}),\phi_h\right)+\tau\left(\nabla\varphi_{h\tau}(t_{n+1}),\nabla\phi_h  \right)=\tau(1,\phi_h),\ \forall\phi_h\in \mathbb{V}_{h}^{1},\\
&\varphi_{h\tau}(0)=0.
\end{aligned}\right.
\end{equation}
From Theorem 7.10 in \cite{touying_error}, the following fully discrete error estimate is obtained:
\begin{equation}\label{varphi_error2}
\begin{aligned}
\|\varphi_{h\tau}(t_{n})-\varphi(t_n)\|
\leq Ch^2\int_{0}^{t_n}\|\varphi_{t}(s)\|_{H^2(\mathcal{D})}ds
+C\tau\int_{0}^{t_n}\|\varphi_{t,t}(s)\|ds
\leq C(h^2+\tau).
\end{aligned}
\end{equation}
Based the above estimate and the regularity of $\varphi(t)$, it follows that
\begin{equation}\label{varphi_error3}
\begin{aligned}
\max\limits_{0\leq n\leq N}\|\varphi_{h\tau}(t_{n})\|\leq \sup_{t\in[0,T]}\|\varphi(t)\|+C(h^2+\tau)
\leq C.
\end{aligned}
\end{equation}
Then the error estimate between $\mu$ and $\mu_{h\tau}$ is given by the following theorem.
\begin{theorem}\label{estimate_mu}
Let $\mu$ and $\mu_{h\tau}$ be the non-negative constants obtained by (\ref{first_continuous}) and (\ref{fully_first_continuous}), then the following estimate holds for sufficiently small $h$ and $\tau$:
\begin{equation*}
\begin{aligned}
|\mu-\mu_{h\tau}|^2
&\leq C_{\mu}\left(\tau\sum_{n=0}^{N-1}\|\mathbb{E}\left[U_{h\tau}(t_n)-U(t_n)\right]\|^2+\|\mathbb{E}\left[X_{h\tau}^{U}(t_{n+1})-X(t_{n+1})\right]\|^2 \right.\\
&\quad\left.+\|\mathbb{E}\left[Y_{h\tau}^U(t_{n})-Y(t_n)\right]\|^2\right),
\end{aligned}
\end{equation*}
where $C_{\mu}$ is a positive constant independence of $h$ and $\tau$.
\end{theorem}
\begin{proof}
Fixing one realization $\omega\in\Omega$ and setting $\phi_h=Y_{h\tau}(t_{n})-Y_{h\tau}^U(t_{n})$ in (\ref{varphi_fully}) and $\psi_h=\varphi_{h\tau}(t_{n+1})$ in (\ref{lisan_fuzhu}) leads to
\begin{align*}
&\left(\mathbb{E}\left[Y_{h\tau}(t_{n+1})-Y_{h\tau}^U(t_{n+1})|\mathcal{F}_{t_n}\right],\varphi_{h\tau}(t_{n+1})\right)-
\left(Y_{h\tau}(t_{n})-Y_{h\tau}^U(t_{n}),\varphi_{h\tau}(t_{n})\right)\\
&=\left(\mathbb{E}\left[Y_{h\tau}(t_{n+1})-Y_{h\tau}^U(t_{n+1})|\mathcal{F}_{t_n}\right]-\left(Y_{h\tau}(t_{n})-Y_{h\tau}^U(t_{n})\right),\varphi_{h\tau}(t_{n+1})\right)\\
&\quad+\left(\varphi_{h\tau}(t_{n+1})-\varphi_{h\tau}(t_{n}),Y_{h\tau}(t_{n})-Y_{h\tau}^U(t_{n})\right)\\
&=\left(\tau\Delta_h \left(Y_{h\tau}^{U}(t_n)-Y_{h\tau}(t_n)\right)+\tau\mathbb{E}\left[X(t_{n+1})+\mu-X_{h\tau}(t_{n+1})-\mu_{h\tau}|\mathcal{F}_{t_n}\right],\varphi_{h\tau}(t_{n+1})\right)\\
&\quad+\tau\left(\Delta_h\varphi_{h\tau}(t_{n+1})+1,Y_{h\tau}(t_{n})-Y_{h\tau}^U(t_{n})\right)\\
&=\left(\tau\mathbb{E}\left[X(t_{n+1})+\mu-X_{h\tau}(t_{n+1})-\mu_{h\tau}|\mathcal{F}_{t_n}\right],\varphi_{h\tau}(t_{n+1})\right)+\tau\left(1,Y_{h\tau}(t_{n})-Y_{h\tau}^U(t_{n})\right).
\end{align*}
Taking the expectation and summing from $n=0$ to $N-1$ yields
\begin{equation*}
\begin{aligned}
&\sum_{n=0}^{N-1}\tau\left(\mu-\mu_{h\tau},\varphi_{h\tau}(t_{n+1})\right)\\
&=\sum_{n=0}^{N-1}\tau\left(\mathbb{E}\left[X_{h\tau}(t_{n+1})-X(t_{n+1})\right],\varphi_{h\tau}(t_{n+1})\right)+\sum_{n=0}^{N-1}\tau\left(1,\mathbb{E}\left[Y_{h\tau}^U(t_{n})-Y_{h\tau}(t_{n})\right]\right).
\end{aligned}
\end{equation*}
Further, it follows that
\begin{align*}
&\mu-\mu_{h\tau}\\
&=C_\varphi\left(\sum_{n=0}^{N-1}\tau\left(\mathbb{E}\left[X_{h\tau}(t_{n+1})-X(t_{n+1})\right],\varphi_{h\tau}(t_{n+1})\right)  \right.\\
&\left.\quad+\sum_{n=0}^{N-1}\tau\left(1,\mathbb{E}\left[Y_{h\tau}^U(t_{n})-Y_{h\tau}(t_{n})\right]\right) +\left[\mu-\mu_{h\tau},\varphi(t)-\varphi_{h\tau}(t)\right]\right)\\
&=C_\varphi\left(\sum_{n=0}^{N-1}\tau\left(\mathbb{E}\left[X_{h\tau}(t_{n+1})-X(t_{n+1})\right],\varphi_{h\tau}(t_{n+1})\right)  \right.\\
&\left.\quad+\sum_{n=0}^{N-1}\tau\left(1,\mathbb{E}\left[Y_{h\tau}^U(t_{n})-Y_{h\tau}(t_{n})\right]\right) +(\mu-\mu_{h\tau})\sum_{n=0}^{N-1}\int_{t_n}^{t_{n+1}}\int_{\mathcal{D}}\varphi(t)-\varphi(t_{n+1})dxdt \right.\\
&\left.\quad+(\mu-\mu_{h\tau})\tau\sum_{n=0}^{N-1}\int_{\mathcal{D}}\varphi(t_{n+1})-\varphi_{h\tau}(t_{n+1})dx\right),
\end{align*}
where $C_\varphi:=\frac{1}{\int_{0}^{T}\int_{\mathcal{D}}\varphi(t)dxdt}>0$ supposed by the maximum principle (\cite{evans}).
Employing (\ref{varphi_error1}), (\ref{varphi_error2}), (\ref{varphi_error3}) and Cauchy-Schwartz inequality gives
\begin{align*}
&|\mu-\mu_{h\tau}|^2\\
&\leq 4C_\varphi^2\left[\left(\tau\sum_{n=0}^{N-1}\int_{\mathcal{D}}\mathbb{E}\left[X_{h\tau}(t_{n+1})-X(t_{n+1})\right]\varphi_{h\tau}(t_{n+1})dx\right)^2 \right.\\
&\quad\left.+\left(\tau\sum_{n=0}^{N-1}\int_{\mathcal{D}}\mathbb{E}\left[Y_{h\tau}^U(t_{n})-Y_{h\tau}(t_{n})   \right]dx\right)^2 \right.\\
&\left.\quad+|\mu-\mu_{h\tau}|^2\left(\sum_{n=0}^{N-1}\int_{t_n}^{t_{n+1}}\int_{\mathcal{D}}\varphi(t)-\varphi(t_{n+1})dxdt\right)^2 \right.\\
&\left.\quad+|\mu-\mu_{h\tau}|^2\left(\tau\sum_{n=0}^{N-1}\int_{\mathcal{D}}\varphi(t_{n+1})-\varphi_{h\tau}(t_{n+1})dx\right)^2
\right]\\
&\leq 4C_\varphi^2\left(\tau\sum_{n=0}^{N-1}\|\mathbb{E}\left[X_{h\tau}(t_{n+1})-X(t_{n+1})\right]\|^2\|\varphi_{h\tau}(t_{n+1})\|^2 \right.\\
&\left.\quad+|\mathcal{D}|\tau\sum_{n=0}^{N-1}\|\mathbb{E}\left[Y_{h\tau}^U(t_{n})-Y_{h\tau}(t_{n})\right]\|^2
+C|\mathcal{D}||\mu-\mu_{h\tau}|^2(h^4+\tau^2) \right.\\
&\left.\quad+|\mu-\mu_{h\tau}|^2|\mathcal{D}|\sum_{n=0}^{N-1}\int_{t_n}^{t_{n+1}}\|\varphi(t)-\varphi(t_{n+1})\|^2dt\right)\\
&\leq C\left(\tau\sum_{n=0}^{N-1}\left(\|\mathbb{E}\left[X_{h\tau}(t_{n+1})-X(t_{n+1})\right]\|^2+\|\mathbb{E}\left[Y_{h\tau}^U(t_{n})-Y_{h\tau}(t_{n})\right]\|^2
\right) \right.\\  
&\left.\quad+|\mu-\mu_{h\tau}|^2 (h^4+\tau^2)\right),
\end{align*}
which implies the following estimate for  sufficiently small $h$ and $\tau$
\begin{align*}
&|\mu-\mu_{h\tau}|^2\\
&\leq C\left(\tau\sum_{n=0}^{N-1}\left(\|\mathbb{E}\left[X_{h\tau}(t_{n+1})-X_{h\tau}^{U}(t_{n+1})\right]\|^2+\|\mathbb{E}\left[X_{h\tau}^{U}(t_{n+1})-X(t_{n+1})\right]\|^2\right)\right.\\
&\left.\quad+\tau\sum_{n=0}^{N-1}\left(\|\mathbb{E}\left[Y_{h\tau}^U(t_{n})-Y(t_n)\right]\|^2+\|\mathbb{E}\left[Y(t_n)-Y_{h\tau}(t_{n})\right]\|^2\right)\right).
\end{align*}
Taking the expectation for the FSPDE in (\ref{lisan_fuzhu}), setting $\phi_h=\mathbb{E}\left[X_{h\tau}(t_{n+1})-X_{h\tau}^U(t_{n+1})\right]$ and performing the derivation like (\ref{full_fullfu_X}) in Theorem \ref{full_fullfu_XY} leads to
\begin{equation}\label{weak_Uh_U}
\begin{aligned}
\max_{0\leq n\leq N}\|\mathbb{E}\left[X_{h\tau}(t_{n})-X_{h\tau}^U(t_{n})\right]\|^2\leq \tau\sum_{n=0}^{N-1}\|\mathbb{E}\left[U_{h\tau}(t_{n})-U(t_{n})\right]\|^2.
\end{aligned}
\end{equation}
Using the fact that $Y(t)=-\alpha U(t)$ and $Y_{h\tau}(t)=-\alpha U_{h\tau}(t)$ in (\ref{first_continuous}) and (\ref{fully_first_continuous}) as well as (\ref{weak_Uh_U}) gives
\begin{equation*}
\begin{aligned}
|\mu-\mu_{h\tau}|^2
&\leq C_{\mu}\left(\tau\sum_{n=0}^{N-1}\|\mathbb{E}\left[U_{h\tau}(t_n)-U(t_n)\right]\|^2+\|\mathbb{E}\left[X_{h\tau}^{U}(t_{n+1})-X(t_{n+1})\right]\|^2 \right.\\
&\quad\left.+\|\mathbb{E}\left[Y_{h\tau}^U(t_{n})-Y(t_n)\right]\|^2\right).
\end{aligned}
\end{equation*}
\end{proof}
Based on the previous results, we have the following optimal error estimates of the multiplier, control, state and adjoint state variables.
\begin{theorem}\label{main_order}
Under the regularities in Theorem \ref{fbspde_exact_fu}, let the solutions of (\ref{first_continuous}) and (\ref{fully_first_continuous}) be $\left(X(t),Y(t),Z(t),\mu,U(t)\right)$ and $\left(X_{h\tau}(t),Y_{h\tau}(t),Z_{h\tau}(t),\mu_{h\tau},U_{h\tau}(t)\right)$, respectively, the following error estimates are obtained for sufficiently small $h$ and $\tau$:
\begin{equation}\label{main_theorem}
\begin{aligned}
&\max_{0\leq n\leq N}\mathbb{E}\left[\|U_{h\tau}(t_n)-U(t_n)\|^2\right]+|\mu-\mu_{h\tau}|^2\leq C(h^4+\tau^2),\\
&\max_{0\leq n\leq N}\mathbb{E}\left[\|X(t_n)-X_{h\tau}(t_n)\|^2\right]+\max_{0\leq n\leq N}\mathbb{E}\left[\|Y(t_n)-Y_{h\tau}(t_n)\|^2\right]\leq C(h^4+\tau^2),\\
&\tau\sum_{n=0}^{N-1}\mathbb{E}\left[\|\nabla(X(t_{n+1})-X_{h\tau}(t_{n+1}))\|^2+\|\nabla(Y(t_n)-Y_{h\tau}(t_{n}))\|^2\right]\leq C(h^2+\tau^2).
\end{aligned}
\end{equation}
\end{theorem}
\begin{proof}
The proof is completed in two steps. The first step proves the following estimate: 
\begin{equation*}
\begin{aligned}
\tau\sum_{n=0}^{N-1}\mathbb{E}\left[\|U_{h\tau}(t_n)-U(t_n)\|^2\right]\leq C(h^4+\tau^2),
\end{aligned}
\end{equation*}
and the second step gives the results (\ref{main_theorem}).

\textbf{Step 1.} Setting $\phi_h=Y_{h\tau}(t_n)-Y_{h\tau}^U(t_n)$, $\psi_h=X_{h\tau}(t_{n+1})-X_{h\tau}^{U}(t_{n+1})$ in (\ref{lisan_fuzhu}) and using the $\mathcal{F}_{t_n}$-measurability of $X_{h\tau}(t_{n+1})-X_{h\tau}^{U}(t_{n+1})$ yields
\begin{equation}\label{esti_u_1}
\begin{aligned}
&\tau\sum_{n=0}^{N-1}\mathbb{E}\left[\left(Y_{h\tau}^U(t_n)-Y_{h\tau}(t_n), U_{h\tau}(t_n)-U(t_n)  \right)\right]\\
&=\tau\sum_{n=0}^{N-1}\mathbb{E}\left[\left( X(t_{n+1})-X_{h\tau}(t_{n+1})+\mu-\mu_{h\tau},X_{h\tau}(t_{n+1})-X_{h\tau}^{U}(t_{n+1})  \right)\right].
\end{aligned}
\end{equation}
From (\ref{esti_u_1}), the continuous and discrete first-order optimality conditions we get
\begin{align*}
&\alpha\tau\sum_{n=0}^{N-1}\mathbb{E}\left[\|U_{h\tau}(t_n)-U(t_n)\|^2\right]\\
&=\tau\sum_{n=0}^{N-1}\mathbb{E}\left[\left(Y(t_n)-Y_{h\tau}(t_n),U_{h\tau}(t_n)-U(t_n)\right)\right]\\
&=\tau\sum_{n=0}^{N-1}\mathbb{E}\left[\left(Y(t_n)-Y_{h\tau}^U(t_n),U_{h\tau}(t_n)-U(t_n)\right)
+\left(Y_{h\tau}^U(t_n)-Y_{h\tau}(t_n), U_{h\tau}(t_n)-U(t_n)  \right)\right]\\
&=\tau\sum_{n=0}^{N-1}\mathbb{E}\left[\left(Y(t_n)-Y_{h\tau}^U(t_n),U_{h\tau}(t_n)-U(t_n)\right)\right]\\
&\quad+\tau\sum_{n=0}^{N-1}\mathbb{E}\left[\left( X(t_{n+1})- X_{h\tau}^{U}(t_{n+1}),X_{h\tau}(t_{n+1})-X_{h\tau}^{U}(t_{n+1})\right)\right]\\
&\quad+\tau\sum_{n=0}^{N-1}\mathbb{E}\left[\left(  X_{h\tau}^{U}(t_{n+1})-X_{h\tau}(t_{n+1}),X_{h\tau}(t_{n+1})-X_{h\tau}^{U}(t_{n+1})\right) \right]\\
&\quad+\tau\sum_{n=0}^{N-1}\mathbb{E}\left[\left(\mu-\mu_{h\tau},X_{h\tau}(t_{n+1})-X_{h\tau}^{U}(t_{n+1})\right) \right].
\end{align*}
Employing (\ref{mu_property}), (\ref{lisan_mu_property}) and the property of conditional expectation gives
\begin{equation*}
\begin{aligned}
&\tau\sum_{n=0}^{N-1}\mathbb{E}\left[\left(\mu-\mu_{h\tau},X_{h\tau}(t_{n+1})-X_{h\tau}^{U}(t_{n+1})\right) \right]\\
&=\int_{0}^{T}\int_{\mathcal{D}}(\mu-\mu_{h\tau})\mathbb{E}\left[X_{h\tau}(t)-X(t)\right]dxdt
+\sum_{n=0}^{N-1}\int_{t_n}^{t_{n+1}}\left(\mu-\mu_{h\tau},\mathbb{E}\left[X(t)-X(t_{n+1})\right]\right)dt\\
&\quad+\tau\sum_{n=0}^{N-1}\left(\mu-\mu_{h\tau},\mathbb{E}\left[X(t_{n+1})-X_{h\tau}^{U}(t_{n+1})\right]\right) \\
&\leq 0+\sum_{n=0}^{N-1}\int_{t_n}^{t_{n+1}}\left(\mu-\mu_{h\tau},\mathbb{E}\left[X(t)-X(t_{n+1}) \right]\right)dt\\
&\quad+\tau\sum_{n=0}^{N-1}\left(\mu-\mu_{h\tau},\mathbb{E}\left[X(t_{n+1})-X_{h\tau}^{U}(t_{n+1})\right]\right).
\end{aligned}
\end{equation*}
From Cauchy-Schwartz inequality, Young's inequality and Theorems \ref{full_fullfu_XY}, \ref{estimate_mu}, the following estimate holds
\begin{align}\label{esti_u_2}
&\alpha\tau\sum_{n=0}^{N-1}\mathbb{E}\left[\|U_{h\tau}(t_n)-U(t_n)\|^2\right]+\tau\sum_{n=0}^{N-1}\mathbb{E}\left[\|X_{h\tau}^{U}(t_{n+1})-X_{h\tau}(t_{n+1})\|^2\right] \nonumber\\
&\leq \tau\sum_{n=0}^{N-1}\mathbb{E}\left[\left(Y(t_n)-Y_{h\tau}^U(t_n),U_{h\tau}(t_n)-U(t_n)\right)  +\left(\mu-\mu_{h\tau},X(t_{n+1})-X_{h\tau}^{U}(t_{n+1})\right)\right] \nonumber\\
&\quad+\tau\sum_{n=0}^{N-1}\mathbb{E}\left[\left( X(t_{n+1})- X_{h\tau}^{U}(t_{n+1}),X_{h\tau}(t_{n+1})-X_{h\tau}^{U}(t_{n+1})\right) \right] \nonumber\\
&\quad   +\sum_{n=0}^{N-1}\int_{t_n}^{t_{n+1}}\left(\mu-\mu_{h\tau},\mathbb{E}\left[X(t)-X(t_{n+1})\right]\right) dt \nonumber\\
&\leq \tau\sum_{n=0}^{N-1}\mathbb{E}\left[\frac{1}{\varepsilon_1}\|Y(t_n)-Y_{h\tau}^U(t_n)\|^2+{\varepsilon_1}\|U_{h\tau}(t_n)-U(t_n)\|^2\right. \nonumber\\
&\quad\left.+\frac{1}{\varepsilon_2}\|X(t_{n+1})-X_{h\tau}^{U}(t_{n+1})\|^2+{\varepsilon_2}\|X_{h\tau}(t_{n+1})-X_{h\tau}^{U}(t_{n+1})\|^2  \right] \nonumber\\
&\quad+{|\mathcal{D}|\varepsilon_3}|\mu-\mu_{h\tau}|^2+\frac{1}{\varepsilon_3}\tau\sum_{n=0}^{N-1}\mathbb{E}\left[ \|X(t_{n+1})-X_{h\tau}^{U}(t_{n+1})\|^2 \right] \nonumber\\
&\quad+\sum_{n=0}^{N-1}\int_{t_n}^{t_{n+1}}\left(\mu-\mu_{h\tau},\mathbb{E}\left[X(t)-X(t_{n+1})\right]\right) dt \nonumber\\
&\leq \left(\varepsilon_1+2\varepsilon_2+|\mathcal{D}|C_{\mu}\varepsilon_3\right)\tau\sum_{n=0}^{N-1}\mathbb{E}\left[\|U_{h\tau}(t_n)-U(t_n)\|^2\right] \nonumber\\
&\quad+\left(\frac{1}{\varepsilon_1}+\frac{1}{\varepsilon_2}+C_{\mu}|\mathcal{D}|\varepsilon_3 +\frac{1}{\varepsilon_3}\right)\tau\sum_{n=0}^{N-1}\mathbb{E}\left[\|Y(t_n)-Y_{h\tau}^U(t_n)\|^2 \right.\nonumber\\
&\left.\quad\quad+\|X(t_{n+1})-X_{h\tau}^{U}(t_{n+1})\|^2  \right]
+\sum_{n=0}^{N-1}\int_{t_n}^{t_{n+1}}\left(\mu-\mu_{h\tau},\mathbb{E}\left[X(t)-X(t_{n+1})\right]\right) dt.
\end{align}
From the state equation (\ref{state}), Cauchy-Schwarz inequality and Theorem \ref{estimate_mu} again, it follows that
\begin{equation}\label{esti_u_3}
\begin{aligned}
&\sum_{n=0}^{N-1}\int_{t_n}^{t_{n+1}}\int_{\mathcal{D}}(\mu-\mu_{h\tau})\mathbb{E}\left[X(t)-X(t_{n+1})\right]dxdt\\
&=\sum_{n=0}^{N-1}\int_{\mathcal{D}}\int_{t_n}^{t_{n+1}}(\mu-\mu_{h\tau})\mathbb{E}\left[\int_{t}^{t_{n+1}}-\Delta X(s)-U(s)ds\right]dtdx\\
&\leq \frac{\tau^2}{\varepsilon_4}\int_{0}^{T}\mathbb{E}\left[ \|X(t)\|_{H^2(\mathcal{D})}^2+\|U(t)\|^2\right]dt+\varepsilon_4|\mathcal{D}||\mu-\mu_{h\tau}|^2\\
&\leq C\tau^2+\varepsilon_4|\mathcal{D}|C_{\mu}\left(\tau\sum_{n=0}^{N-1}\mathbb{E}\left[\|U_{h\tau}(t_n)-U(t_n)\|^2+\|X_{h\tau}^{U}(t_{n+1})-X(t_{n+1})\|^2 \right.\right.\\
&\quad\left.\left.+\|Y_{h\tau}^U(t_{n})-Y(t_n)\|^2\right]\right).
\end{aligned}
\end{equation}
The positive constants $\varepsilon_1,\varepsilon_2,\varepsilon_3$ and $\varepsilon_4$ are chosen to satisfy
\begin{equation*}
\begin{aligned}
\varepsilon_1+2\varepsilon_2+|\mathcal{D}|C_{\mu}(\varepsilon_3+\varepsilon_4)<\alpha.
\end{aligned}
\end{equation*}
Combining (\ref{esti_u_2}), (\ref{esti_u_3}) and Theorem \ref{fbspde_exact_fu} yields
\begin{equation}\label{estimate_u_l2}
\begin{aligned}
\tau\sum_{n=0}^{N-1}\mathbb{E}\left[\|U_{h\tau}(t_n)-U(t_n)\|^2\right]\leq C\tau^2+C\left(h^4+\tau^2  \right)\leq C\left(h^4+\tau^2\right).
\end{aligned}
\end{equation}

\textbf{Step 2.} Combining (\ref{estimate_u_l2}), Theorem \ref{fbspde_exact_fu} and Theorem \ref{estimate_mu} leads to
\begin{equation}\label{optimal_mu}
\begin{aligned}
|\mu-\mu_{h\tau}|^2\leq C(h^4+\tau^2).
\end{aligned}
\end{equation}
Then from Theorems \ref{fbspde_exact_fu}, \ref{full_fullfu_XY} and estimates (\ref{estimate_u_l2}), (\ref{optimal_mu}), it holds that
\begin{align}\label{esti_u_4}
&\max_{0\leq n\leq N}\mathbb{E}\left[\|X(t_n)-X_{h\tau}(t_n)\|^2\right]+\max_{0\leq n\leq N}\mathbb{E}\left[\|Y(t_n)-Y_{h\tau}(t_n)\|^2\right] \nonumber\\
&\leq 2\max_{0\leq n\leq N}\mathbb{E}\left[\|X(t_n)-X_{h\tau}^U(t_n)\|^2\right]+
2\max_{0\leq n\leq N}\mathbb{E}\left[\|X_{h\tau}^U(t_n)-X_{h\tau}(t_n)\|^2\right] \nonumber\\
&\quad+2\max_{0\leq n\leq N}\mathbb{E}\left[\|Y(t_n)-Y_{h\tau}^U(t_n)\|^2\right]+
2\max_{0\leq n\leq N}\mathbb{E}\left[\|Y_{h\tau}^U(t_n)-Y_{h\tau}(t_n)\|^2\right] \nonumber\\
&\leq C\left(h^4+\tau^2+\tau\sum_{n=0}^{N-1}\mathbb{E}\left[\|U_{h\tau}(t_n)-U(t_n)\|^2 
+\|X(t_{n+1})-X_{h\tau}^{U}(t_{n+1})\|^2 \right] \right. \nonumber\\
&\quad\left.+|\mu-\mu_{h\tau}|^2 \right) \nonumber\\
&\leq C(h^4+\tau^2).
\end{align}
Analogously, the optimal error estimates in $H_{0}^{1}(\mathcal{D})$-norm are deduced as 
\begin{equation}\label{esti_u_5}
\begin{aligned}
&\tau\sum_{n=0}^{N-1}\mathbb{E}\left[\|\nabla\left(X(t_{n+1})-X_{h\tau}(t_{n+1})\right)\|^2+\|\nabla\left(Y(t_n)-Y_{h\tau}(t_{n})\right)\|^2\right]\\
&\leq 2\tau\sum_{n=0}^{N-1}\mathbb{E}\left[\|\nabla\left(X(t_{n+1})-X_{h\tau}^U(t_{n+1})\right)\|^2+\|\nabla\left(X_{h\tau}^U(t_{n+1})-X_{h\tau}(t_{n+1})\right)\|^2 \right]\\
&\quad+2\tau\sum_{n=0}^{N-1}\mathbb{E}\left[\|\nabla\left(Y(t_n)-Y_{h\tau}^U(t_{n})\right)\|^2+\|\nabla\left(Y_{h\tau}^U(t_{n})-Y_{h\tau}(t_{n})\right)\|^2 \right]\\
&\leq C\left(h^2+\tau^2+\tau\sum_{n=0}^{N-1}\mathbb{E}\left[\|X(t_{n+1})-X_{h\tau}^{U}(t_{n+1})\|^2+\|U_{h\tau}(t_{n})-U(t_{n})\|^2 \right] \right.\\
&\quad\left.+|\mu-\mu_{h\tau}|^2   \right)\\
&\leq C(h^2+\tau^2).
\end{aligned}
\end{equation}
Using (\ref{esti_u_4}) and (\ref{first_continuous}), (\ref{fullyfu_first_continuous}) again yields 
\begin{equation}\label{esti_u_6}
\begin{aligned}
\max_{0\leq n \leq N}\mathbb{E}\left[\|U_{h\tau}(t_n)-U(t_n)\|^2\right]
\leq C\max_{0\leq n \leq N}\mathbb{E}\left[\|Y(t_n)-Y_{h\tau}(t_n)\|^2\right]\leq C(h^4+\tau^2).
\end{aligned}
\end{equation}
Putting the estimates (\ref{optimal_mu}), (\ref{esti_u_4}), (\ref{esti_u_5}), (\ref{esti_u_6}) together gives the desired assertion (\ref{main_theorem}).
\end{proof}
\begin{remark}
Notice that the strong convergence orders are used to bounded the multiplier error in (\ref{optimal_mu}). 
The reason we are allowed to do this is that in additive noise the optimal strong convergence order that is the same as the weak convergence order is achieved, which does not hold in all cases, e.g., when the state equation is driven by multiplicative noise.
In this case, from the error estimate of stochastic differential equation with multiplicative noise it is clear that strong convergence order is only half of the weak convergence order \cite{Zhang}.
\end{remark}

\section{A gradient projection algorithm}
In this section, an efficient gradient projection algorithm is proposed to solve SOCP (\ref{fully_object})-(\ref{fully_state}), the core idea of which is to choose the specific multiplier in each iteration step so that the numerical state satisfies the constraint condition. Further, the convergence analysis of the algorithm is deduced.

Like the deterministic OCP with integral state constraint \cite{integral1,integral2}, for a given $U_{h\tau}^{i}(t)\in \mathbb{U}_{h\tau}$, the computational scheme reads:
\begin{equation}\label{ita_scheme}
\left\{\begin{aligned}
&U_{h\tau}^{i+\frac{1}{2}}(t)=U_{h\tau}^{i}(t)-\rho\left(\alpha U_{h\tau}^{i}(t)+\tilde{Y}_{h\tau}^{i}(t)\right),\\
&U_{h\tau}^{i+1}(t)=\mathbb{R}_h\left(U_{h\tau}^{i+\frac{1}{2}}(t)\right)=U_{h\tau}^{i+\frac{1}{2}}(t)-\rho\mu_{h\tau}^{i}\tilde{M}_{h\tau}(t).
\end{aligned}\right.
\end{equation}
Here $\rho$ is a positive constant and $\mathbb{R}_h$ is the projection operator from $\mathbb{U}_{h\tau}$ to the control domain $\mathbb{U}_{h\tau,\delta}$ given in (\ref{lisan_control_domain}). The projection $\mathbb{R}_h$ satisfies the following definition:
\begin{equation}\label{pro_Rh}
\begin{aligned}
\mathbb{E}\left[\left[V_{h\tau}(t)-\mathbb{R}_hV_{h\tau}(t),W_{h\tau}(t)-\mathbb{R}_hV_{h\tau}(t)    \right]\right]\leq 0,\ \forall W_{h\tau}\in\mathbb{U}_{h\tau,\delta}.
\end{aligned}
\end{equation}
We denote by $X_{h\tau}^{i}(t)\in\mathbb{X}_{h\tau}$ and $\left(Y_{h\tau}^{i}(t),Z_{h\tau}^{i}(t)\right)\in\left(\mathbb{U}_{h\tau}\right)^2$ the solutions of state and adjoint equations solved by $U_{h\tau}^{i}(t)$ and $X_{h\tau}^{i}(t),\mu_{h\tau}^{i}$, respectively. From the definition of adjoint solution, it holds that $Y_{h\tau}^{i}(t)=\tilde{Y}_{h\tau}^{i}(t)+\mu_{h\tau}^{i}\tilde{M}_{h\tau}(t)$, $Z_{h\tau}^{i}(t)=\tilde{Z}_{h\tau}^{i}(t)$, where $\left(\tilde{Y}_{h\tau}^{i}(t),\tilde{Z}_{h\tau}^{i}(t)\right)\in\left(\mathbb{U}_{h\tau}\right)^2$ and $\tilde{M}_{h\tau}(t)\in\mathbb{L}_{h\tau}$ are the solutions to the following equations:
\begin{equation}\label{algo_fenxi}
\left\{\begin{aligned}
&X_{h\tau}^{i}(t_{n+1})-X_{h\tau}^{i}(t_{n})=\tau\left[\Delta_hX_{h\tau}^{i}(t_{n+1})+U_{h\tau}^{i}(t_n)\right]+\Pi_{h}^{1}\sigma(t_n)\Delta W_{n+1},\\
&X_{h\tau}^{i}(0)=\Pi_{h}^{1}X_0,\\
&\tilde{Z}_{h\tau}^{i}(t_n)=\mathbb{E}\left[\left(\tau^{-1}\tilde{Y}_{h\tau}^{i}(t_{n+1})+X_{h\tau}^{i}(t_{n+1})-\Pi_{h}^{1}X_d(t_{n+1})   \right)\Delta W_{n+1}|\mathcal{F}_{t_n}\right],\\
&\tilde{Y}_{h\tau}^{i}(t_{n})=\mathbb{E}\left[ \tilde{Y}_{h\tau}^{i}(t_{n+1})+\tau\Delta_h\tilde{Y}_{h\tau}^{i}(t_{n})+\tau\left(X_{h\tau}^{i}(t_{n+1})-\Pi_{h}^{1}X_d(t_{n+1})\right)|\mathcal{F}_{t_n}\right],\\
&\tilde{Y}_{h\tau}^{i}(T)=0,\\
&\tilde{M}_{h\tau}(t_{n+1})-\tilde{M}_{h\tau}(t_{n})=\tau\left[-\Delta_h\tilde{M}_{h\tau}(t_{n})- 1\right],\\
&\tilde{M}_{h\tau}(T)=0.
\end{aligned}\right.
\end{equation}
Next we discuss how to select $\mu_{h\tau}^{i}$ such that $U_{h\tau}^{i+1}(t)\in\mathbb{U}_{h\tau,\delta}$. First, we denote by $X_{h\tau}^{i+1}(t)\in\mathbb{X}_{h\tau}$ and $X_{h\tau}^{i+\frac{1}{2}}\in\mathbb{X}_{h\tau}$ the state variables solved by $U_{h\tau}^{i+1}(t)$ and $U_{h\tau}^{i+\frac{1}{2}}(t)$ as the control variables, respectively, i.e.,
\begin{equation}\label{algo_X12}
\left\{\begin{aligned}
&X_{h\tau}^{i+1}(t_{n+1})-X_{h\tau}^{i+1}(t_{n})
=\tau\left[\Delta_hX_{h\tau}^{i+1}(t_{n+1})+U_{h\tau}^{i+1}(t_n)\right]+\Pi_{h}^{1}\sigma(t_n)\Delta W_{n+1},\\
&X_{h\tau}^{i+1}(0)=\Pi_{h}^{1}X_0,\\
&X_{h\tau}^{i+\frac{1}{2}}(t_{n+1})-X_{h\tau}^{i+\frac{1}{2}}(t_{n})
=\tau\left[\Delta_hX_{h\tau}^{i+\frac{1}{2}}(t_{n+1})+U_{h\tau}^{i+\frac{1}{2}}(t_n)\right]+\Pi_{h}^{1}\sigma(t_n)\Delta W_{n+1},\\
&X_{h\tau}^{i+\frac{1}{2}}(0)=\Pi_{h}^{1}X_0.
\end{aligned}\right.
\end{equation}
According to the relationship between $U_{h\tau}^{i+1}(t)$ and $U_{h\tau}^{i+\frac{1}{2}}(t)$ in (\ref{ita_scheme}), it can be found that 
\begin{equation*}
\begin{aligned}
X_{h\tau}^{i+1}(t)=X_{h\tau}^{i+\frac{1}{2}}(t)-\rho\mu_{h\tau}^{i}\tilde{Q}_{h\tau}(t), 
\end{aligned}
\end{equation*}
where $\tilde{Q}_{h\tau}(t)\in\mathbb{R}_{h\tau}$ is solved by the following equation
\begin{equation}\label{algo_tildeQ}
\left\{\begin{aligned}
&\tilde{Q}_{h\tau}(t_{n+1})-\tilde{Q}_{h\tau}(t_{n})=\tau\left[\Delta_h\tilde{Q}_{h\tau}(t_{n+1})+\tilde{M}_{h\tau}(t_{n})\right],\\
&\tilde{Q}_{h\tau}(0)=0.
\end{aligned}\right.
\end{equation}
Then we need only select
\begin{equation}\label{mu_ht}
\begin{aligned}
\mu_{h\tau}^{i}=\frac{\max\left\{ \int_{0}^{T}\int_{\mathcal{D}}\mathbb{E}[X_{h\tau}^{i+\frac{1}{2}}(t)]dxdt-\delta,0\right\}}{\rho\int_{0}^{T}\int_{\mathcal{D}}\tilde{Q}_{h\tau}(t)dxdt} 
\end{aligned}
\end{equation}
to assure $U_{h\tau}^{i+1}(t)\in\mathbb{U}_{h\tau,\delta}$. The extremum principle implies that $\rho\int_{0}^{T}\int_{\mathcal{D}}\tilde{Q}_{h\tau}(t)dxdt>0$ for sufficiently small $h$ and $\tau$, which leads to $\mu_{h\tau}^{i}\geq0$.
The algorithm is summarized as follows:
\begin{algorithm}
\caption{Gradient Projection Algorithm}\label{algorithm}
\begin{algorithmic}[1]
\Require Constant $\rho > 0$, initial control value $U_{h\tau}^{0}(t) \in \mathbb{U}_{h\tau}$, and error tolerance $\varepsilon_0$
\Ensure 
\State Setting $error>\varepsilon_0$, $i=0$
\While{$error > \varepsilon_0$}
    \State Solving the first equation in (\ref{algo_fenxi}) using $U_{h\tau}^{i}(t)$ to obtain $X_{h\tau}^{i}(t)$;
    \State Solving the third equation in (\ref{algo_fenxi}) using $X_{h\tau}^{i}(t)$ to obtain $\tilde{Y}_{h\tau}^{i}(t)$;
    \State Computing $U_{h\tau}^{i+\frac{1}{2}}(t)=U_{h\tau}^{i}(t) - \rho\left(\alpha U_{h\tau}^{i}(t) + \tilde{Y}_{h\tau}^{i}(t)\right)$;
    \State Solving the second equation in (\ref{algo_X12}) using $U_{h\tau}^{i+\frac{1}{2}}(t)$ to obtain $X_{h\tau}^{i+\frac{1}{2}}(t)$;
    \State Solving the forth equation in (\ref{algo_fenxi}) to obtain $\tilde{M}_{h\tau}(t)$;
    \State Solving the equation in (\ref{algo_tildeQ}) using $\tilde{M}_{h\tau}(t)$ to obtain $\tilde{Q}_{h\tau}(t)$;
    \State Computing $\mu_{h\tau}^{i}$ using (\ref{mu_ht});
    \State Updating $U_{h\tau}^{i+1}(t)=U_{h\tau}^{i+\frac{1}{2}}(t)-\rho\mu_{h\tau}^{i} \tilde{M}_{h\tau}(t)$;
    \State Computing $error=\left(\tau \sum\limits_{n=0}^{N-1} \|U_{h\tau}^{i+1}(t_n) - U_{h\tau}^{i}(t_n)\|^2\right)^{1/2}$;
    \State Updating $i=i+1$.
\EndWhile
\end{algorithmic}
\end{algorithm}

Next it is shown that $\mathbb{R}_h$ satisfies the definition of projection operator in (\ref{pro_Rh}) for the given multiplier $\mu_{h\tau}^{i}$ in (\ref{mu_ht}).
\begin{theorem}\label{pro_project_u}
For the projection $\mathbb{R}_h$ and multiplier $\mu_{h\tau}^{i}$ given in iterative scheme (\ref{ita_scheme}) and (\ref{mu_ht}), it holds that 
\begin{equation*}
\begin{aligned}
\mathbb{E}\left[\left[ U_{h\tau}^{i+\frac{1}{2}}(t)-\mathbb{R}_hU_{h\tau}^{i+\frac{1}{2}}(t),W_{h\tau}(t)-\mathbb{R}_hU_{h\tau}^{i+\frac{1}{2}}(t)\right]\right]\leq0,\ \forall W_{h\tau}(t)\in\mathbb{U}_{h\tau,\delta}.
\end{aligned}
\end{equation*}
Further, for any $V_{h\tau}(t),P_{h\tau}(t)\in\mathbb{U}_{h\tau}$, it holds that
\begin{equation*}
\begin{aligned}
\int_{0}^{T}\mathbb{E}\left[\|\mathbb{R}_hV_{h\tau}(t)-\mathbb{R}_hP_{h\tau}(t)\|^2\right]dt
\leq\int_{0}^{T}\mathbb{E}\left[\|V_{h\tau}(t)-P_{h\tau}(t)\|^2\right]dt.
\end{aligned}
\end{equation*}
\end{theorem}
\begin{proof}
From the definition of $\mathbb{R}_h$, it holds that
\begin{equation*}
\begin{aligned}
&\mathbb{E}\left[\left[U_{h\tau}^{i+\frac{1}{2}}(t)-\mathbb{R}_hU_{h\tau}^{i+\frac{1}{2}}(t),W_{h\tau}(t)-\mathbb{R}_hU_{h\tau}^{i+\frac{1}{2}}(t)\right]\right]\\
&=\mathbb{E}\left[\left[\rho\mu_{h\tau}^{i}\tilde{M}_{h\tau}(t),W_{h\tau}(t)-U_{h\tau}^{i+\frac{1}{2}}(t)
\right]\right]
+\mathbb{E}\left[\left[\rho\mu_{h\tau}^{i}\tilde{M}_{h\tau}(t),\rho\mu_{h\tau}^{i}\tilde{M}_{h\tau}(t)\right]\right]\\
&=M_1+M_2.
\end{aligned}
\end{equation*}
In order to estimate $M_1$, the following equation is introduced: Find $X_{h\tau}^W(t)\in\mathbb{X}_{h\tau}$ such that
\begin{equation*}
\left\{\begin{aligned}
&X_{h\tau}^W(t_{n+1})-X_{h\tau}^W(t_{n})=\tau\left[\Delta_hX_{h\tau}^W(t_{n+1})+W_{h\tau}(t_n)\right]+\Pi_{h}^{1}\sigma(t_n)\Delta W_{n+1},\\
&X_{h\tau}^W(0)=\Pi_{h}^{1}X_0,
\end{aligned}\right.
\end{equation*}
which implies that $X_{h\tau}^W(t)-X_{h\tau}^{i+\frac{1}{2}}(t)\in\mathbb{R}_{h\tau}$ and it satisfies 
\begin{equation*}
\left\{\begin{aligned}
&X_{h\tau}^W(t_{n+1})-X_{h\tau}^{i+\frac{1}{2}}(t_{n+1})-\left(X_{h\tau}^W(t_{n})-X_{h\tau}^{i+\frac{1}{2}}(t_{n})\right)\\
&=\tau\left[\Delta_h\left(X_{h\tau}^W(t_{n+1})-X_{h\tau}^{i+\frac{1}{2}}(t_{n+1}) \right)+W_{h\tau}(t_{n})-U_{h\tau}^{i+\frac{1}{2}}(t_n)\right],\\
&X_{h\tau}^W(0)-X_{h\tau}^{i+\frac{1}{2}}(0)=0.
\end{aligned}\right.
\end{equation*}
From the above equation it can be derived that
\begin{equation}\label{esti_T1_1}
\begin{aligned}
M_1
&=\tau\sum_{n=0}^{N-1}\int_{\mathcal{D}}\mathbb{E}\left[\rho\mu_{h\tau}^{i}\tilde{M}_{h\tau}(t_{n})\left(
W_{h\tau}(t_n)-U_{h\tau}^{i+\frac{1}{2}}(t_n)  \right) \right]dx\\
&=\rho\mu_{h\tau}^{i}\sum_{n=0}^{N-1}\int_{\mathcal{D}}\mathbb{E}\left[\tilde{M}_{h\tau}(t_{n})\left( X_{h\tau}^W(t_{n+1})-X_{h\tau}^{i+\frac{1}{2}}(t_{n+1})\right.\right.\\
&\quad\left.\left.-X_{h\tau}^W(t_{n})+X_{h\tau}^{i+\frac{1}{2}}(t_{n}) 
-\tau\Delta_h\left(X_{h\tau}^W(t_{n+1})-X_{h\tau}^{i+\frac{1}{2}}(t_{n+1}) \right)
\right)\right]dx.
\end{aligned}
\end{equation}
Employing the definition of $\tilde{M}_{h\tau}(t)$ and the fact that $X_{h\tau}^W(0)-X_{h\tau}^{i+\frac{1}{2}}(0)=0$ yields  
\begin{align}\label{esti_T1_2}
&\sum_{n=0}^{N-1}\int_{\mathcal{D}}\tilde{M}_{h\tau}(t_{n})\left(X_{h\tau}^W(t_{n+1})-X_{h\tau}^{i+\frac{1}{2}}(t_{n+1})-X_{h\tau}^W(t_{n})+X_{h\tau}^{i+\frac{1}{2}}(t_{n}) \right)dx \nonumber\\
&=\sum_{n=0}^{N-1}\int_{\mathcal{D}}\left(\tilde{M}_{h\tau}(t_n)-\tilde{M}_{h\tau}(t_{n+1})\right)\left(X_{h\tau}^W(t_{n+1})-X_{h\tau}^{i+\frac{1}{2}}(t_{n+1})\right)dx  \nonumber\\
&=\sum_{n=0}^{N-1}\int_{\mathcal{D}}\tau\left(\Delta_h\tilde{M}_{h\tau}(t_{n})+1\right)\left(X_{h\tau}^W(t_{n+1})-X_{h\tau}^{i+\frac{1}{2}}(t_{n+1})\right)dx.
\end{align}
Bringing (\ref{esti_T1_2}) into (\ref{esti_T1_1}) gives
\begin{equation}\label{esti_T1_3}
\begin{aligned}
M_1
=\rho\mu_{h\tau}^{i}\tau\sum_{n=0}^{N-1}\int_{\mathcal{D}}\mathbb{E}\left[  X_{h\tau}^W(t_{n+1})-X_{h\tau}^{i+\frac{1}{2}}(t_{n+1})\right] dx
=\mathbb{E}\left[\left[\rho\mu_{h\tau}^{i}, X_{h\tau}^W(t)-X_{h\tau}^{i+\frac{1}{2}}(t) \right] \right].
\end{aligned}
\end{equation}

For the term $M_2$, it follows that $M_2=\tau(\rho\mu_{h\tau}^{i})^2\sum_{n=0}^{N-1}\int_{\mathcal{D}}\left(\tilde{M}_{h\tau}(t_{n})\right)^2 dx$.
From the definitions of $\tilde{M}_{h\tau}(t)$ and $\tilde{Q}_{h\tau}(t)$, it is shown that
\begin{align*}
&\tau\sum_{n=0}^{N-1}\int_{\mathcal{D}}\tilde{M}_{h\tau}(t_{n})\left(-\Delta_h\tilde{Q}_{h\tau}(t_{n+1})-\tilde{M}_{h\tau}(t_{n}) \right)dx\\
&=\sum_{n=0}^{N-1}\int_{\mathcal{D}}\tilde{M}_{h\tau}(t_{n})\left(\tilde{Q}_{h\tau}(t_{n})-\tilde{Q}_{h\tau}(t_{n+1})   \right)dx\\
&=\sum_{n=0}^{N-1}\int_{\mathcal{D}}\tilde{Q}_{h\tau}(t_{n+1})\left(\tilde{M}_{h\tau}(t_{n+1})-\tilde{M}_{h\tau}(t_n)\right)dx\\
&=\tau\sum_{n=0}^{N-1}\int_{\mathcal{D}}\tilde{Q}_{h\tau}(t_{n+1})\left(-\Delta_h\tilde{M}_{h\tau}(t_{n})- 1\right) dx,
\end{align*}
which implies that
\begin{equation*}
\begin{aligned}
M_2
=\tau(\rho\mu_{h\tau}^{i})^2\sum_{n=0}^{N-1}\int_{\mathcal{D}}\tilde{Q}_{h\tau}(t_{n+1}) dx
=\left[(\rho\mu_{h\tau}^{i})^2, \tilde{Q}_{h\tau}(t)\right].
\end{aligned}
\end{equation*}
Using the definition of $\mu_{h\tau}^{i}$ in (\ref{mu_ht}) and the estimates of $M_1$ as well as $M_2$ gives
\begin{equation*}\label{pro_prove_Rh}
\begin{aligned}
&\mathbb{E}\left[\left[U_{h\tau}^{i+\frac{1}{2}}(t)-\mathbb{R}_hU_{h\tau}^{i+\frac{1}{2}}(t), W_{h\tau}(t)-\mathbb{R}_hU_{h\tau}^{i+\frac{1}{2}}(t)\right]\right]\\
&=\mathbb{E}\left[\left[\rho\mu_{h\tau}^{i}, X_{h\tau}^W(t)-X_{h\tau}^{i+\frac{1}{2}}(t) \right] \right]
+\left[(\rho\mu_{h\tau}^{i})^2, \tilde{Q}_{h\tau}(t)\right]\\
&\leq \rho\mu_{h\tau}^{i}\left(\delta-
\int_{0}^{T}\int_{\mathcal{D}}\mathbb{E}\left[X_{h\tau}^{i+\frac{1}{2}}(t)\right]dxdt 
+\max\left\{ \int_{0}^{T}\int_{\mathcal{D}}\mathbb{E}[X_{h\tau}^{i+\frac{1}{2}}(t)]dxdt-\delta,0    \right\} \right)\\
&=0.
\end{aligned}
\end{equation*}
For any $V_{h\tau}(t),P_{h\tau}(t)\in \mathbb{U}_{h\tau}$, by the above inequality, we get
\begin{equation*}
\begin{aligned}
&\mathbb{E}\left[ \left[\mathbb{R}_hV_{h\tau}(t)-\mathbb{R}_hP_{h\tau}(t),\mathbb{R}_hV_{h\tau}(t)-\mathbb{R}_hP_{h\tau}(t)\right]  \right]\\
&=\mathbb{E}\left[\left[\mathbb{R}_hV_{h\tau}(t)-\mathbb{R}_hP_{h\tau}(t),P_{h\tau}(t)-\mathbb{R}_hP_{h\tau}(t)  \right] \right]\\ &\quad+\mathbb{E}\left[\left[\mathbb{R}_hV_{h\tau}(t)-\mathbb{R}_hP_{h\tau}(t),\mathbb{R}_hV_{h\tau}(t)-V_{h\tau}(t)        \right] \right]\\
&\quad+\mathbb{E}\left[\left[\mathbb{R}_hV_{h\tau}(t)-\mathbb{R}_hP_{h\tau}(t),V_{h\tau}(t)-P_{h\tau}(t) \right] \right]\\
&\leq \mathbb{E}\left[\left[\mathbb{R}_hV_{h\tau}(t)-\mathbb{R}_hP_{h\tau}(t),V_{h\tau}(t)-P_{h\tau}(t) \right] \right]\\
&\leq \left(\mathbb{E}\left[\left[\mathbb{R}_hV_{h\tau}(t)-\mathbb{R}_hP_{h\tau}(t),\mathbb{R}_hV_{h\tau}(t)-\mathbb{R}_hP_{h\tau}(t)\right]\right]\right)^{1/2}\\ &\quad\left(\mathbb{E}\left[\left[V_{h\tau}(t)-P_{h\tau}(t),V_{h\tau}(t)-P_{h\tau}(t)\right]\right]\right)^{1/2},
\end{aligned}
\end{equation*}
which implies that
\begin{equation*}
\begin{aligned}
\int_{0}^{T}\mathbb{E}\left[\|\mathbb{R}_hV_{h\tau}(t)-\mathbb{R}_hP_{h\tau}(t)\|^2\right]dt
\leq\int_{0}^{T}\mathbb{E}\left[\|V_{h\tau}(t)-P_{h\tau}(t)\|^2\right]dt.
\end{aligned}
\end{equation*}
\end{proof}

Based on the property of projection $\mathbb{R}_{h}$ in Theorem \ref{pro_project_u}, we give the convergence analysis of Algorithm \ref{algorithm}.
\begin{theorem}
Let the solutions solved by iterative algorithm \ref{algorithm} and fully discrete first-order optimality condition (\ref{fully_first_continuous}) be $\left(X^{i+1}_{h\tau}(t),Y^{i+1}_{h\tau}(t),Z^{i+1}_{h\tau}(t),\mu_{h\tau}^{i+1},U^{i+1}_{h\tau}(t)\right)$ and $\left(X_{h\tau}(t),Y_{h\tau}(t),Z_{h\tau}(t),\mu_{h\tau},U_{h\tau}(t)\right)$, respectively. Then for sufficiently small $h$ and $\tau$, there exists $0<\varrho\leq\rho<\frac{2}{\alpha+2e^{T}}$ such that $0<\lambda<1$ and
\begin{equation*}
\begin{aligned}
&\int_{0}^{T}\mathbb{E}\left[\|U_{h\tau}(t)-U_{h\tau}^{i+1}(t)\|^2\right]dt +|\mu_{h\tau}^{i+1}-\mu_{h\tau}|^2\\
&\quad+\max_{0\leq n\leq N}\mathbb{E}\left[\|X_{h\tau}(t_{n})-X_{h\tau}^{i+1}(t_{n})\|^2\right]+
\int_{0}^{T}\mathbb{E}\left[\|\nabla(X_{h\tau}(t)-X_{h\tau}^{i+1}(t))\|^2\right]dt\\
&\quad+\max_{0\leq n\leq N}\mathbb{E}\left[\|Y_{h\tau}(t_{n})-Y_{h\tau}^{i+1}(t_{n})\|^2\right]+\int_{0}^{T}\mathbb{E}\left[\|\nabla(Y_{h\tau}(t)-Y_{h\tau}^{i+1}(t))\|^2\right]dt\\
&\leq C\lambda^{2(i+1)}\int_{0}^{T}\mathbb{E}\left[\|U_{h\tau}(t)-U_{h\tau}^{0}(t)\|^2\right]dt,
\end{aligned}
\end{equation*}
where
\begin{equation*}
\lambda=\left\{\begin{aligned}
&1-\rho\alpha,\ \mathrm{if}\ 0<\varrho\leq\rho\leq\frac{1}{\alpha+e^T},\\
&\sqrt{F(\rho)},\ \mathrm{if}\ \frac{1}{\alpha+e^T}<\rho<\frac{2}{\alpha+2e^{T}},
\end{aligned}\right.
\end{equation*}
and $F(\rho)=\rho^2(\alpha+1)(\alpha+2e^{T})-\rho(2\alpha+2)+1$.
\end{theorem}
\begin{proof}
We prove the convergence of discrete control, state, multiplier, and adjoint state in four separate steps.

\textbf{Step 1.} Convergence for control variable.
The adjoint process $Y_{h\tau}(t)$ can be split into $\tilde{Y}_{h\tau}(t)$ and $\tilde{M}_{h\tau}(t)$, i.e., $Y_{h\tau}(t)=\tilde{Y}_{h\tau}(t)+\mu_{h\tau}\tilde{M}_{h\tau}(t)$, where $\tilde{Y}_{h\tau}(t)\in\mathbb{U}_{h\tau}$ is solved by $X_{h\tau}(t)$ and $\tilde{M}_{h\tau}(t)\in\mathbb{L}_{h\tau}$ is given in (\ref{algo_fenxi}).
We first show that
\begin{equation}\label{project_u}
\begin{aligned}
U_{h\tau}(t)=\mathbb{R}_h\left( U_{h\tau}(t)-\rho\left(\alpha U_{h\tau}(t)+\tilde{Y}_{h\tau}(t)\right)\right).
\end{aligned}
\end{equation}
For any $W_{h\tau}(t)\in\mathbb{U}_{h\tau,\delta}$, following the derivation of (\ref{esti_T1_1})-(\ref{esti_T1_3}) and using (\ref{fully_first_continuous}) yields 
\begin{equation*}
\begin{aligned}
&\mathbb{E}\left[\left[U_{h\tau}(t)-\rho\left(\alpha U_{h\tau}(t)+\tilde{Y}_{h\tau}(t)\right)-U_{h\tau}(t),W_{h\tau}(t)-U_{h\tau}(t)\right]\right]\\
&=\rho\mathbb{E}\left[\left[Y_{h\tau}(t)-\tilde{Y}_{h\tau}(t),W_{h\tau}(t)-U_{h\tau}(t)\right]\right]\\
&=\mathbb{E}\left[\left[\rho\mu_{h\tau}\tilde{M}_{h\tau}(t),W_{h\tau}(t)-U_{h\tau}(t)\right]\right]\\
&=\rho\mu_{h\tau}\tau\sum_{n=0}^{N-1}\int_{\mathcal{D}}\mathbb{E}\left[X_{h\tau}^{W}(t_{n+1})-X_{h\tau}(t_{n+1})    \right]dx\\
&=\rho\mathbb{E}\left[\left[\mu_{h\tau}, X_{h\tau}^{W}(t)-X_{h\tau}(t)\right]\right]\leq0,
\end{aligned}
\end{equation*}
which implies that (\ref{project_u}) holds based on the (\ref{pro_Rh}).

By using (\ref{ita_scheme}), (\ref{project_u}) and Theorem \ref{pro_project_u}, it follows that
\begin{align*}
&\int_{0}^{T}\mathbb{E}\left[\|U_{h\tau}^{i+1}(t)-U_{h\tau}(t)\|^2\right]dt\\
&=\int_{0}^{T}\mathbb{E}\left[\left\|\mathbb{R}_h\left(U_{h\tau}^{i}(t)-\rho\left(\alpha U_{h\tau}^{i}(t)+\tilde{Y}_{h\tau}^{i}(t)\right)\right) 
-\mathbb{R}_h\left( U_{h\tau}(t)-\rho\left(\alpha U_{h\tau}(t)+\tilde{Y}_{h\tau}(t)\right)\right)\right\|^2\right]dt\\
&\leq\int_{0}^{T}\mathbb{E}\left[\left\|U_{h\tau}^{i}(t)-\rho\left(\alpha U_{h\tau}^{i}(t)+\tilde{Y}_{h\tau}^{i}(t)\right)-\left( U_{h\tau}(t)-\rho\left(\alpha U_{h\tau}(t)+\tilde{Y}_{h\tau}(t)\right)\right) \right\|^2 \right]dt\\
&=\int_{0}^{T}\mathbb{E}\left[\left\| (1-\rho\alpha)\left(U_{h\tau}^{i}(t)-U_{h\tau}(t)\right)-\rho\left(\tilde{Y}_{h\tau}^{i}(t)-\tilde{Y}_{h\tau}(t)\right)\right\|^2 \right]dt\\
&= (1-\rho\alpha)^2\int_{0}^{T}\mathbb{E}\left[\|U_{h\tau}^{i}(t)-U_{h\tau}(t)\|^2 \right]dt+\rho^2\int_{0}^{T}\mathbb{E}\left[\left\|\tilde{Y}_{h\tau}^{i}(t)-\tilde{Y}_{h\tau}(t)\right\|^2 \right]dt\\
&\quad-2(1-\rho\alpha)\rho\mathbb{E}\left[\left[U_{h\tau}^{i}(t)-U_{h\tau}(t), \tilde{Y}_{h\tau}^{i}(t)-\tilde{Y}_{h\tau}(t) \right]\right].
\end{align*}
From the definitions of $X_{h\tau}(t),X_{h\tau}^{i}(t)$ and $\tilde{Y}_{h\tau}^{i}(t),\tilde{Y}_{h\tau}(t)$, for any $\phi_h,\psi_h\in\mathbb{V}_h^1$ we have
\begin{equation*}
\left\{\begin{aligned}
&\left(X_{h\tau}(t_{n+1})-X_{h\tau}^{i}(t_{n+1})-\left(X_{h\tau}(t_{n})-X_{h\tau}^{i}(t_{n})\right), \phi_h\right)\\
&=\tau\left(\nabla X_{h\tau}^{i}(t_{n+1})-\nabla X_{h\tau}(t_{n+1}),\nabla\phi_h \right)
+\tau\left( U_{h\tau}(t_n)-U_{h\tau}^{i}(t_n),\phi_h\right),\\
&X_{h\tau}(0)-X_{h\tau}^{i}(0)=0,\\
&\left(\mathbb{E}\left[\tilde{Y}_{h\tau}(t_{n+1})-\tilde{Y}_{h\tau}^{i}(t_{n+1})|\mathcal{F}_{t_n} \right]-\left(\tilde{Y}_{h\tau}(t_{n})-\tilde{Y}_{h\tau}^{i}(t_{n})\right), \psi_h\right)\\
&=\tau\left(\nabla\tilde{Y}_{h\tau}(t_{n})-\nabla\tilde{Y}_{h\tau}^{i}(t_{n}),\nabla\psi_h\right)
+\tau\left(\mathbb{E}\left[X_{h\tau}^{i}(t_{n+1})-X_{h\tau}(t_{n+1})|\mathcal{F}_{t_n} \right],\psi_h\right),\\
&\tilde{Y}_{h\tau}(T)-\tilde{Y}_{h\tau}^{i}(T)=0.
\end{aligned}\right.
\end{equation*}
Fixing one realization $\omega\in\Omega$, setting $\phi_h=\tilde{Y}_{h\tau}(t_n)-\tilde{Y}_{h\tau}^{i}(t_n)$, $\psi_h=X_{h\tau}(t_{n+1})-X_{h\tau}^{i}(t_{n+1})$ and using the $\mathcal{F}_{t_n}$-measurability of $X_{h\tau}(t_{n+1})-X_{h\tau}^{i}(t_{n+1})$ gives
\begin{equation}\label{conv_u1}
\begin{aligned}
\mathbb{E}\left[\left[U_{h\tau}^{i}(t)-U_{h\tau}(t), \tilde{Y}_{h\tau}^{i}(t)-\tilde{Y}_{h\tau}(t)  \right]\right]=\int_{0}^{T}\mathbb{E}\left[\left\|X_{h\tau}(t)-X_{h\tau}^{i}(t)\right\|^2\right]dt.
\end{aligned}
\end{equation}
Therefore, by (\ref{conv_u1}), the error between $U_{h\tau}^{i+1}(t)$ and $U_{h\tau}(t)$ can be bounded as
\begin{equation*}
\begin{aligned}
&\int_{0}^{T}\mathbb{E}\left[\|U_{h\tau}^{i+1}(t)-U_{h\tau}(t)\|^2\right]dt
+2\rho(1-\rho\alpha)\int_{0}^{T}\mathbb{E}\left[\left\|X_{h\tau}(t)-X_{h\tau}^{i}(t)\right\|^2\right]dt\\
&\leq (1-\rho\alpha)^2\int_{0}^{T}\mathbb{E}\left[\|U_{h\tau}^{i}(t)-U_{h\tau}(t)\|^2 \right]dt+\rho^2\int_{0}^{T}\mathbb{E}\left[\left\|\tilde{Y}_{h\tau}^{i}(t)-\tilde{Y}_{h\tau}(t)\right\|^2 \right]dt.
\end{aligned}
\end{equation*}
Performing the analysis like Theorem \ref{full_fullfu_XY}, one gets 
\begin{align*}
&\max_{0\leq n \leq N}\mathbb{E}\left[\|X_{h\tau}(t_{n})-X_{h\tau}^{i}(t_{n})\|^2\right]
\leq \tau\sum_{n=0}^{N-1}\mathbb{E}\left[\|U_{h\tau}(t_n)-U_{h\tau}^{i}(t_n)\|^2 \right],\\
&\max_{0\leq n \leq N}\mathbb{E}\left[\|\tilde{Y}_{h\tau}(t_{n})-\tilde{Y}_{h\tau}^{i}(t_{n})\|^2\right]
\leq 2e^{T}\tau\sum_{n=0}^{N-1}\mathbb{E}\left[\|X_{h\tau}^{i}(t_{n+1})-X_{h\tau}(t_{n+1})\|^2\right],
\end{align*}
which implies that
\begin{equation}\label{conv_u2}
\begin{aligned}
&\int_{0}^{T}\mathbb{E}\left[\|U_{h\tau}^{i+1}(t)-U_{h\tau}(t)\|^2\right]dt
+2\rho\left(1-\rho\alpha-\rho e^{T}\right)\int_{0}^{T}\mathbb{E}\left[
\left\|X_{h\tau}(t)-X_{h\tau}^{i}(t)\right\|^2\right]dt\\
&\leq  (1-\rho\alpha)^2\int_{0}^{T}\mathbb{E}\left[\left\|U_{h\tau}^{i}(t)-U_{h\tau}(t) \right\|^2 \right]dt.
\end{aligned}
\end{equation}
If $0<\varrho\leq\rho\leq\frac{1}{\alpha+e^{T}}$, it holds that $0<1-\rho\alpha<1$ and $2\rho\left(1-\rho\alpha-\rho e^{T}\right)\geq0$, from (\ref{conv_u2}) it holds that
\begin{equation}\label{conv_u3}
\begin{aligned}
\int_{0}^{T}\mathbb{E}\left[\|U_{h\tau}^{i+1}(t)-U_{h\tau}(t)\|^2\right]dt
&\leq(1-\rho\alpha)^2\int_{0}^{T}\mathbb{E}\left[\|U_{h\tau}^{i}(t)-U_{h\tau}(t)\|^2 \right]dt\\
&\leq (1-\rho\alpha)^{2(i+1)}\int_{0}^{T}\mathbb{E}\left[\|U_{h\tau}^{0}(t)-U_{h\tau}(t)\|^2 \right]dt.
\end{aligned}
\end{equation}
If $\rho>\frac{1}{\alpha+e^{T}}$, we have
\begin{equation*}
\begin{aligned}
\int_{0}^{T}\mathbb{E}\left[\|U_{h\tau}^{i+1}(t)-U_{h\tau}(t)\|^2\right]dt\leq F(\rho)\int_{0}^{T}\mathbb{E}\left[\|U_{h\tau}^{i}(t)-U_{h\tau}(t) \|^2 \right]dt,
\end{aligned}
\end{equation*}
where $F(\rho)=\rho^2(\alpha+1)(\alpha+2e^{T})-\rho(2\alpha+2)+1$. From the properties of quadratic functions, $0<F(\rho)<1$ can be established by selecting $\rho$ that satisfies $\frac{1}{\alpha+e^{T}}<\rho<\frac{2}{\alpha+2e^{T}}$.
Therefore, it's obtained that
\begin{equation}\label{conv_u4}
\begin{aligned}
\int_{0}^{T}\mathbb{E}\left[\|U_{h\tau}^{i+1}(t)-U_{h\tau}(t)\|^2\right]dt
\leq \sqrt{F(\rho)}^{2(i+1)}\int_{0}^{T}\mathbb{E}\left[\|U_{h\tau}^{0}(t)-U_{h\tau}(t)\|^2 \right]dt,
\end{aligned}
\end{equation}
where $0<\sqrt{F(\rho)}<1$. Therefore, combining (\ref{conv_u3}) and (\ref{conv_u4}) gives
\begin{equation}\label{conv_u5}
\begin{aligned}
\int_{0}^{T}\mathbb{E}\left[\|U_{h\tau}^{i+1}(t)-U_{h\tau}(t)\|^2\right]dt\leq \lambda^{2(i+1)}\int_{0}^{T}\mathbb{E}\left[\|U_{h\tau}^{0}(t)-U_{h\tau}(t)\|^2 \right]dt,
\end{aligned}
\end{equation}
where $\lambda=1-\rho\alpha$ when $0<\rho\leq\frac{1}{\alpha+e^{T}}$ and $\lambda=\sqrt{F(\rho)}$ when $\frac{1}{\alpha+e^{T}}<\rho<\frac{2}{\alpha+2e^{T}}$.

\textbf{Step 2.} Convergence for state variable. Performing the similar derivations in Theorem \ref{full_fullfu_XY} for $X_{h\tau}(t)-X^{i+1}_{h\tau}(t)$ and utilizing estimate (\ref{conv_u5}) yields 
\begin{equation}\label{conv_state1}
\begin{aligned}
&\max_{0\leq n\leq N}\mathbb{E}\left[\|X_{h\tau}(t_{n})-X_{h\tau}^{i+1}(t_{n})\|^2\right]+
\int_{0}^{T}\mathbb{E}\left[\|\nabla\left(X_{h\tau}(t)-X_{h\tau}^{i+1}(t)\right)\|^2\right]dt\\
&\leq 2\int_{0}^{T}\mathbb{E}\left[\|U_{h\tau}(t)-U_{h\tau}^{i+1}(t)\|^2\right]dt
\leq 2\lambda^{2(i+1)}\int_{0}^{T}\mathbb{E}\left[\|U_{h\tau}^{0}(t)-U_{h\tau}(t)\|^2 \right]dt.
\end{aligned}
\end{equation}

\textbf{Step 3.} Convergence for multiplier. From the definitions of $Y_{h\tau}(t)$ and $Y_{h\tau}^{i+1}(t)$, it follows that
\begin{equation}\label{conv_adjoint1}
\left\{\begin{aligned}
&\left(\mathbb{E}\left[Y_{h\tau}(t_{n+1})-Y_{h\tau}^{i+1}(t_{n+1})|\mathcal{F}_{t_n}\right]-\left(Y_{h\tau}(t_{n})-Y_{h\tau}^{i+1}(t_{n})\right), \psi_h\right)\\
&=\tau\left(\mathbb{E}\left[X_{h\tau}^{i+1}(t_{n+1})-X_{h\tau}(t_{n+1})+\mu_{h\tau}^{i+1}-\mu_{h\tau}|\mathcal{F}_{t_n}\right],\psi_h\right)\\
&\quad+\tau\left(\nabla Y_{h\tau}(t_{n})-\nabla Y_{h\tau}^{i+1}(t_{n}),\nabla\psi_h\right),\\
&Y_{h\tau}(T)-Y_{h\tau}^{i+1}(T)=0.
\end{aligned}\right.
\end{equation}
Like the derivations in Theorem \ref{estimate_mu}, we set $\psi_h=\varphi_{h\tau}(t_{n+1})$ and $\phi_h=Y_{h\tau}(t_n)-Y_{h\tau}^{i+1}(t_{n})$ in (\ref{varphi_fully}) and it's derived that
\begin{equation}\label{conv_mu1}
\begin{aligned}
&|\mu_{h\tau}^{i+1}-\mu_{h\tau}|^2\\
&\leq C\left(\tau\sum_{n=0}^{N-1}\left(\|\mathbb{E}\left[X_{h\tau}(t_{n+1})-X_{h\tau}^{i+1}(t_{n+1})\right]\|^2+\|\mathbb{E}\left[Y_{h\tau}^{i+1}(t_{n})-Y_{h\tau}(t_{n})\right]\|^2
\right) \right.\\  
&\left.\quad+|\mu_{h\tau}^{i+1}-\mu_{h\tau}|^2 (h^4+\tau^2)\right).
\end{aligned}
\end{equation}
By iterative scheme (\ref{ita_scheme}) and the fact that $Y_{h\tau}^{i+1}(t)=\tilde{Y}_{h\tau}^{i+1}(t)+\mu_{h\tau}^{i+1}\tilde{M}_{h\tau}(t)$, we have
\begin{equation*}
\begin{aligned}
U_{h\tau}^{i+2}(t)
&=U_{h\tau}^{i+1}(t)-\rho\left(\alpha U_{h\tau}^{i+1}(t)+\tilde{Y}_{h\tau}^{i+1}(t)\right)-\rho\mu_{h\tau}^{i+1}\tilde{M}_{h\tau}(t)\\
&=U_{h\tau}^{i+1}(t)-\rho\left(\alpha U_{h\tau}^{i+1}(t)+ Y_{h\tau}^{i+1}(t)\right),
\end{aligned}
\end{equation*}
which implies that 
\begin{equation}\label{conv_mu2}
\begin{aligned}
&\rho^2\int_{0}^{T}\mathbb{E}\left[\|\alpha U_{h\tau}^{i+1}(t)+ Y_{h\tau}^{i+1}(t)\|^2\right]dt\\
&\leq 2\int_{0}^{T}\mathbb{E}\left[  \|U_{h\tau}^{i+1}(t)-U_{h\tau}(t)\|^2\right]dt+2\int_{0}^{T}\mathbb{E}\left[\|U_{h\tau}(t)-U_{h\tau}^{i+2}(t)\|^2\right]dt\\
&\leq 4\lambda^{2(i+1)}\int_{0}^{T}\mathbb{E}\left[\|U_{h\tau}^{0}(t)-U_{h\tau}(t)\|^2 \right]dt.
\end{aligned}
\end{equation}
For sufficiently small $h$ and $\tau$, employing (\ref{fully_first_continuous}), (\ref{conv_u5}), (\ref{conv_state1}), (\ref{conv_mu1}) and (\ref{conv_mu2}) leads to
\begin{align}\label{conv_mu3}
&|\mu_{h\tau}^{i+1}-\mu_{h\tau}|^2  \nonumber\\
&\leq C\left(\max_{0\leq n\leq N}\mathbb{E}\left[\|X_{h\tau}(t_{n})-X_{h\tau}^{i+1}(t_{n})\|^2  \right]
+\int_{0}^{T}\mathbb{E}\left[\|Y_{h\tau}^{i+1}(t)-Y_{h\tau}(t)\|^2\right]dt  \right)  \nonumber\\
&\leq C\left(2\lambda^{2(i+1)}\int_{0}^{T}\mathbb{E}\left[\|U_{h\tau}^{0}(t)-U_{h\tau}(t)\|^2 \right]dt+ 2\int_{0}^{T}\mathbb{E}\left[\|Y_{h\tau}^{i+1}(t)+\alpha U_{h\tau}^{i+1}(t)\|^2\right]dt           \right.\nonumber\\
&\left.\quad+2\int_{0}^{T}\mathbb{E}\left[\|\alpha U_{h\tau}(t)-\alpha U_{h\tau}^{i+1}(t)\|^2\right]dt \right)\nonumber\\
&\leq C\lambda^{2(i+1)}\int_{0}^{T}\mathbb{E}\left[\|U_{h\tau}^{0}(t)-U_{h\tau}(t)\|^2 \right]dt.
\end{align}

\textbf{Step 4.} Convergence for adjoint state. Following a similar derivation as in Theorem \ref{full_fullfu_XY}, applying the estimate to $Y_{h\tau}(t)-Y_{h\tau}^{i+1}(t)$ defined in (\ref{conv_adjoint1}) and combining the estimates (\ref{conv_state1}) as well as (\ref{conv_mu3}) leads to
\begin{equation}\label{conv_adjoint2}
\begin{aligned}
&\max_{0\leq n\leq N}\mathbb{E}\left[\|Y_{h\tau}(t_{n})-Y_{h\tau}^{i+1}(t_{n})\|^2\right]+\int_{0}^{T}\mathbb{E}\left[\|\nabla(Y_{h\tau}(t)-Y_{h\tau}^{i+1}(t))\|^2\right]dt\\
&\leq C\left(\tau\sum_{n=0}^{N-1}\mathbb{E}\left[\|X_{h\tau}(t_{n+1})-X_{h\tau}^{i+1}(t_{n+1})\|^2 \right]+|\mu_{h\tau}-\mu_{h\tau}^{i+1}|^2\right)\\
&\leq C\lambda^{2(i+1)}\int_{0}^{T}\mathbb{E}\left[\|U_{h\tau}^{0}(t)-U_{h\tau}(t)\|^2 \right]dt.
\end{aligned}
\end{equation}
Combining (\ref{conv_u5}), (\ref{conv_state1}), (\ref{conv_mu3}) and (\ref{conv_adjoint2}) yields the final convergence results.
\end{proof}

\section{Numerical experiments}
In this section, numerical examples are presented to verify the accuracy of theoretical analysis and the effectiveness of gradient projection algorithm \ref{algorithm}. 
In the following examples, 
we denote by $R_{X,\tau},R_{X,h}$, $R_{Y,\tau},R_{Y,h}$ and $R_{\mu,\tau},R_{\mu,h}$ the convergence orders of state, adjoint state and multiplier with respect to $\tau$ and $h$ in $|\cdot|$ and $\left(\max\limits_{0\leq n\leq N}\mathbb{E}[\|\cdot\|^2]\right)^{1/2}$ norms. $R_{X,\tau}^{1},R_{X,h}^{1}$ and $R_{Y,\tau}^{1},R_{Y,h}^{1}$ denote the convergence orders of state and adjoint state variables in  $\left(\tau\sum\limits_{n=0}^{N-1}\mathbb{E}[\|\cdot\|_{H_{0}^{1}(\mathcal{D})}^2]\right)^{1/2}$ norm, respectively.
Since the exact solutions of SOCP with stochastic control process and state constraint are often difficult to construct, the integral state-constrained SOCPs with deterministic control variable are considered in one-dimensional and two-dimensional space domain below, and such problems are applied in fields such as finance \cite{app1} and engineering \cite{app2}.

\begin{example}\label{exm1}
The following stochastic optimal control problem with integral state constraint is considered
\begin{equation*}
\begin{aligned}
\min\limits_{\substack{ X(t)\in K\\ U(t) \in L^2\left(0,T;L^2(\mathcal{D})\right) }}
J(X(t),U(t))=\frac{1}{2}\mathbb{E}\left[\int_{0}^{T}\left(\|X(t)-X_d(t)\|^2+\|U(t)\|^2\right)dt\right]
\end{aligned}
\end{equation*}
subject to
\begin{equation*}
\left\{\begin{aligned}
dX(t)&=\left[\Delta X(t)+f(t)+U(t)\right]dt+\beta \sin(\pi x)dW,\ t\in(0,1],\\
X(0)&=0,
\end{aligned}\right.
\end{equation*}
where $\beta$ is a constants.
\end{example}
We set $\mathcal{D}=[0,1]$ and the exact solutions, functions $f(t),X_d(t)$ as well as constraint parameter $\delta$ are given by
\begin{equation*}
\left\{\begin{aligned}
&U(t)=t(T-t)\sin(\pi x),\\
&X(t)=(t+\beta W_t)\sin(\pi x),\\
&f(t)=\sin(\pi x)\left(1+t(t-T)+\pi^2(t+\beta W_t)\right),\\
&\delta=\int_{0}^{1}\int_{\mathcal{D}}\mathbb{E}[X(t)]dxdt=\frac{1}{\pi},\\
&X_d(t)=\sin(\pi x)\left(t-T+2(t+W_t)-\pi^2(t-T)(t+\beta W_t)\right)+\mu,
\end{aligned}\right.
\end{equation*}
where $\beta=0.1$ and multiplier is randomly chosen as $\mu=1$. The number of Monte Carlo simulation paths is $2*10^{3}$ and we set $\varepsilon_0=10^{-6}$ in Algorithm \ref{algorithm}. We first successively set ${\tau}={h}=\frac{1}{40},\frac{1}{45},\frac{1}{50},\frac{1}{60},\frac{1}{70}$. The profiles of numerical control and the expected numerical state are shown in Figure \ref{figure1}. 
According to Theorem \ref{main_order}, the convergence orders of the state, adjoint state and multiplier with respect to space and time is the same, all of which are first-order. The convergence rates with respect to time shown in Figure \ref{figure1} agree with the theoretical analysis.

Then in order to test the variation of the convergence orders with respect to the discretization parameters $\tau$ and $h$, we set ${h}=\frac{1}{10},\frac{1}{15},\frac{1}{20},\frac{1}{25},\frac{1}{30}$, ${\tau}={h^2}$ and ${\tau}={h^4}$, respectively. The numerical results are given in Figure \ref{figure2}, which can be found that the convergence curves of the numerical solutions are parallel to the curves of the standard convergence orders.
To test the role of the state constraint, we choose the constraint parameter 
$$\delta=0.2,0.1,-0.1,-0.2,$$ 
successively, which implies that the explicit exact solutions are unknown. The expected numerical states and integral values with ${\tau}={h}=\frac{1}{40},\frac{1}{45},\frac{1}{50},\frac{1}{60},\frac{1}{70}$ are given in Figure \ref{figure3} and Table \ref{table1}. Clearly the resulting numerical states are all in the constraint set, which validates the effectiveness of our algorithm.

\begin{figure}[!htbp]
\flushleft
\label{1a}
\includegraphics[width=4.5cm,height=4cm]{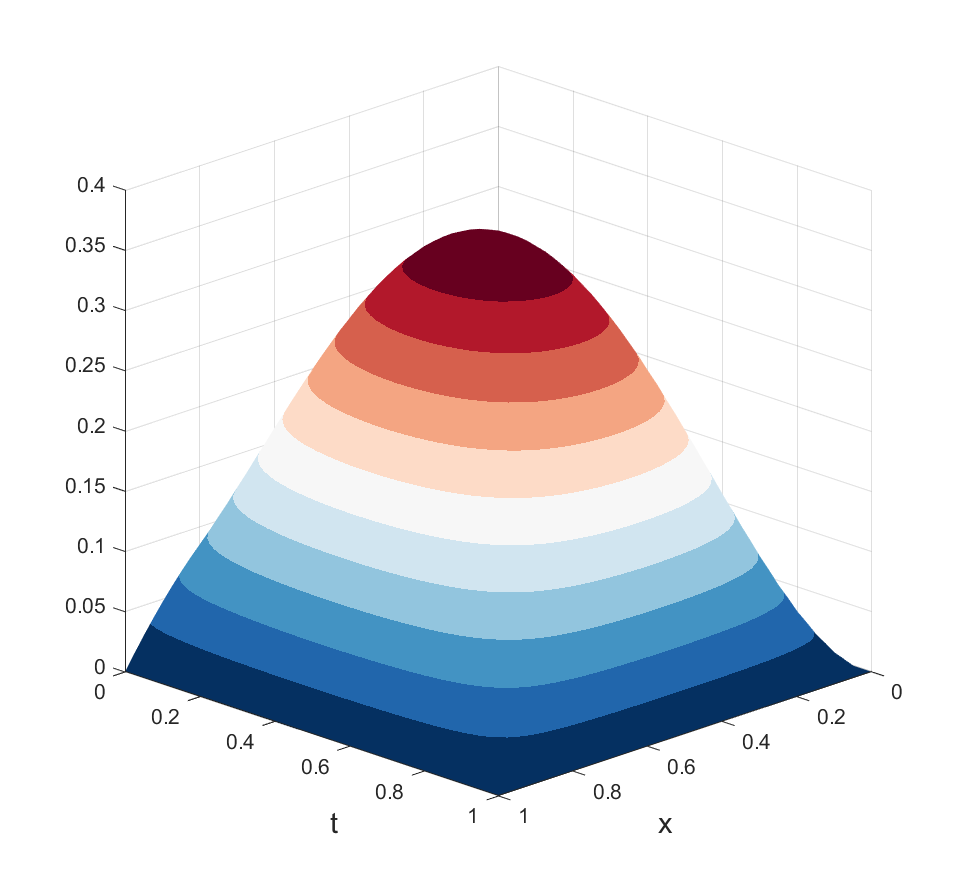}
\hspace{-5mm}
\label{1b}
\includegraphics[width=4.5cm,height=4cm]{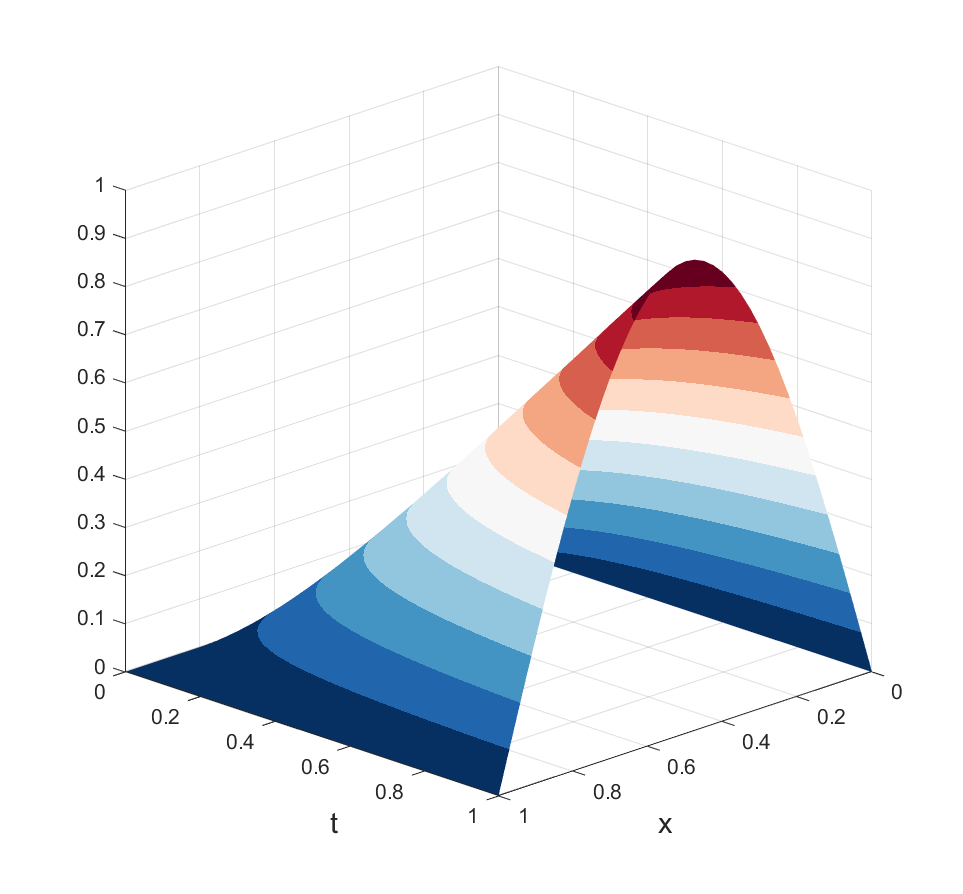}
\hspace{-5mm}
\label{1c}
\includegraphics[width=4.5cm,height=4cm]{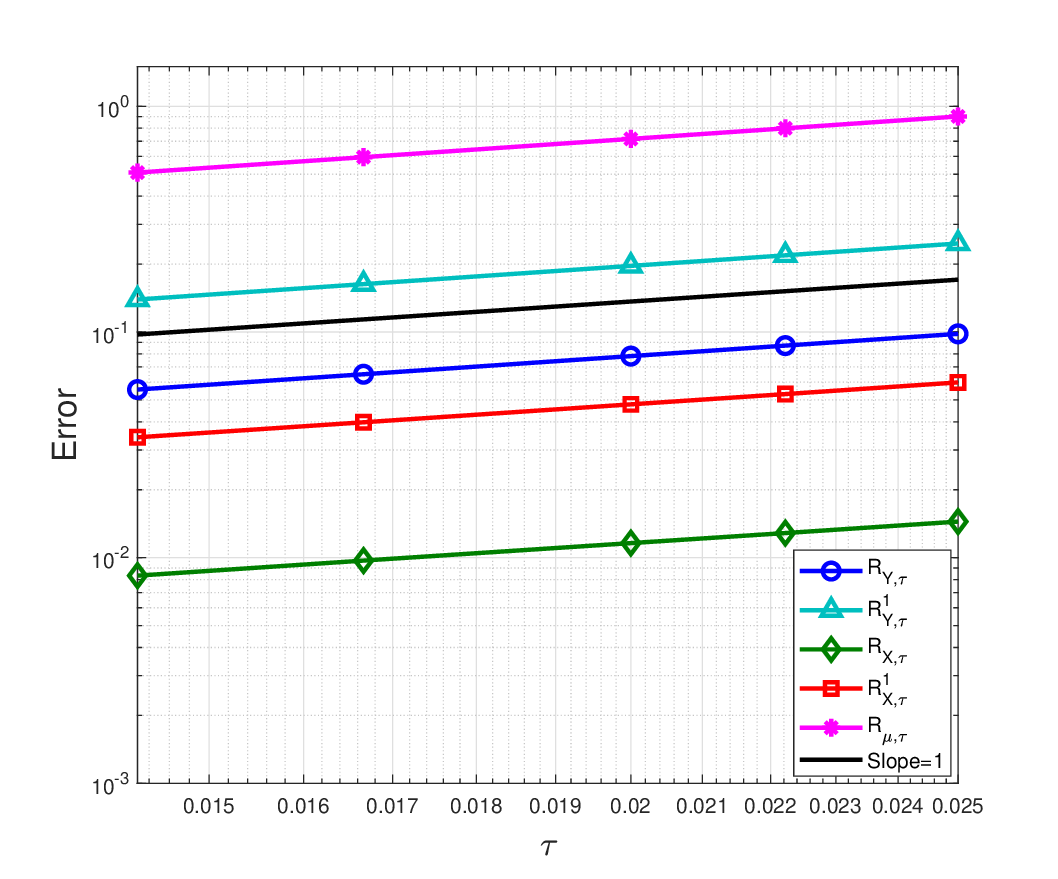}
\hspace{-5mm}
\caption{The profiles of numerical control (Left), expected numerical state (Middle) with ${\tau}={h}=\frac{1}{40}$ and convergence rates (Right) with ${\tau}={h}=\frac{1}{40},\frac{1}{45},\frac{1}{50},\frac{1}{60},\frac{1}{70}$ of Example \ref{exm1}.}
\label{figure1}
\end{figure}

\begin{figure}[!htbp]
\flushleft
\label{2a}
\includegraphics[width=3.5cm,height=3cm]{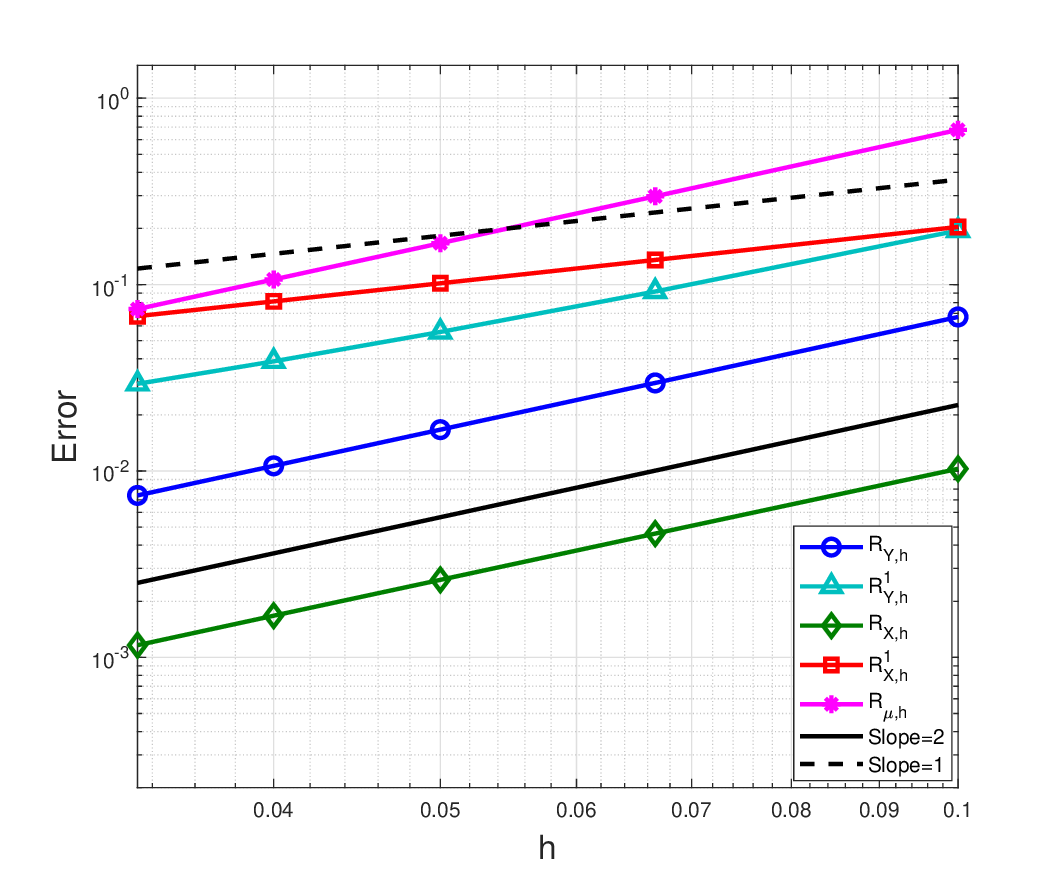}
\hspace{-5mm}
\label{2b}
\includegraphics[width=3.5cm,height=3cm]{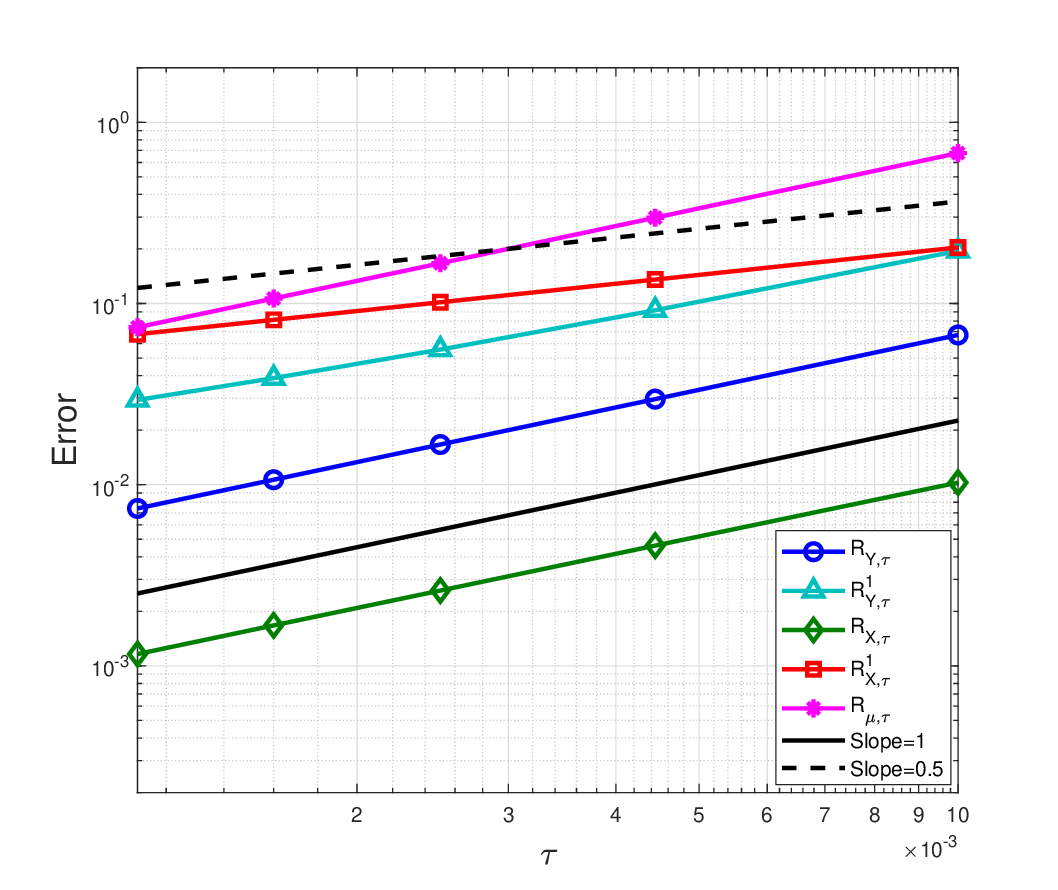}
\hspace{-5mm}
\label{2c}
\includegraphics[width=3.5cm,height=3cm]{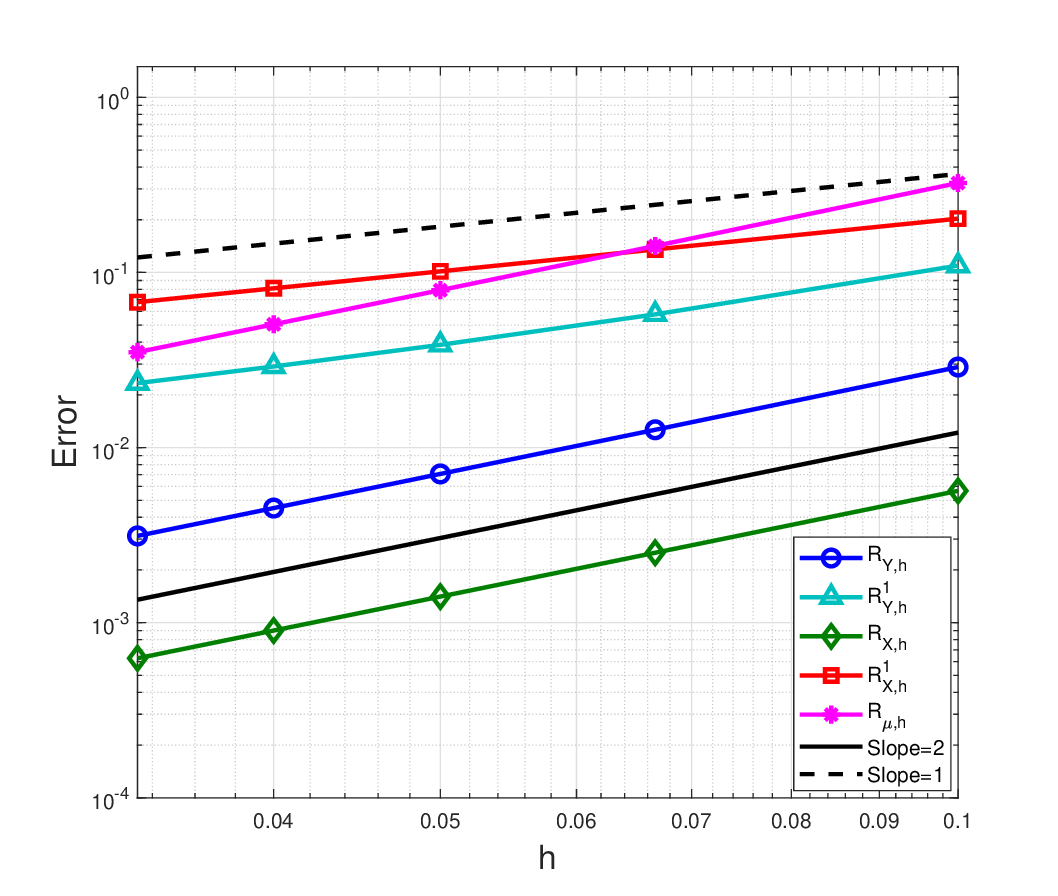}
\hspace{-5mm}
\label{2d}
\includegraphics[width=3.5cm,height=3cm]{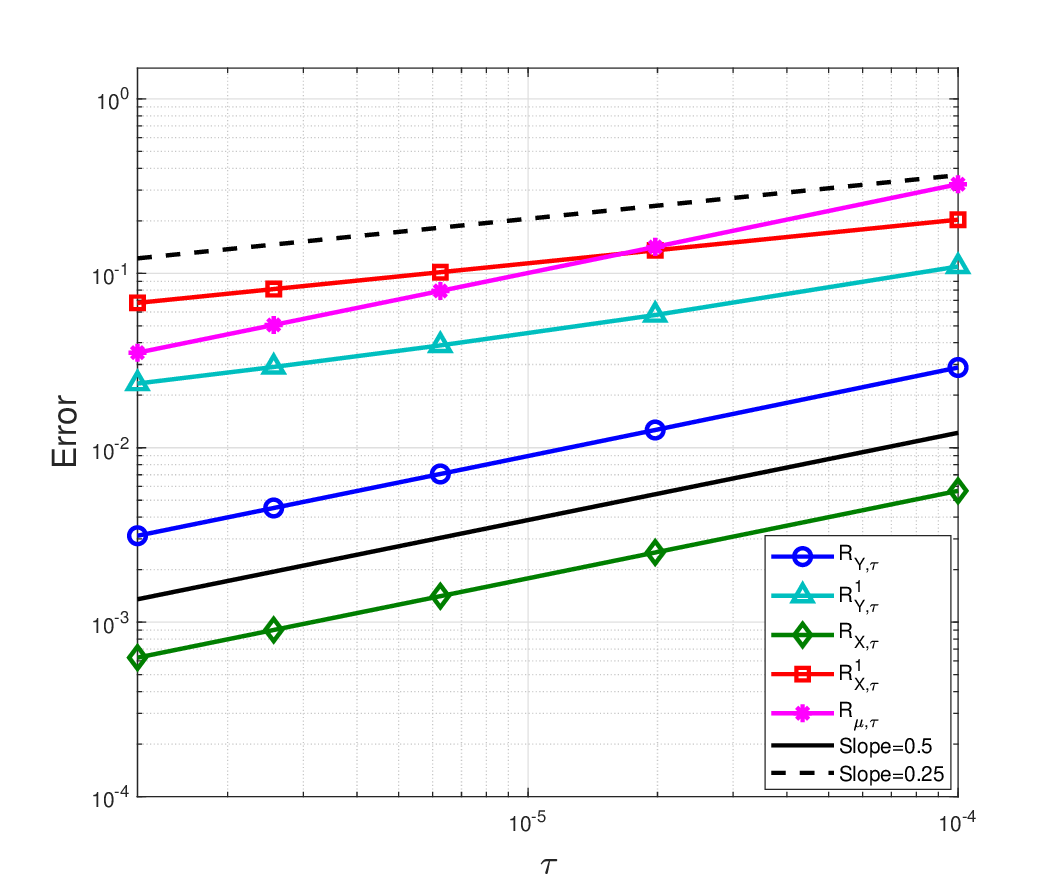}
\hspace{-5mm}
\caption{The convergence rates with ${h}=\frac{1}{10},\frac{1}{15},\frac{1}{20},\frac{1}{25},\frac{1}{30}$, ${\tau}={h^2}$ (Left, Center-Left) and ${\tau}={h^4}$ (Center-Right, Right) of Example \ref{exm1}.}
\label{figure2}
\end{figure}

\begin{figure}[!htbp]
\flushleft
\label{3a}
\includegraphics[width=3.5cm,height=3.5cm]{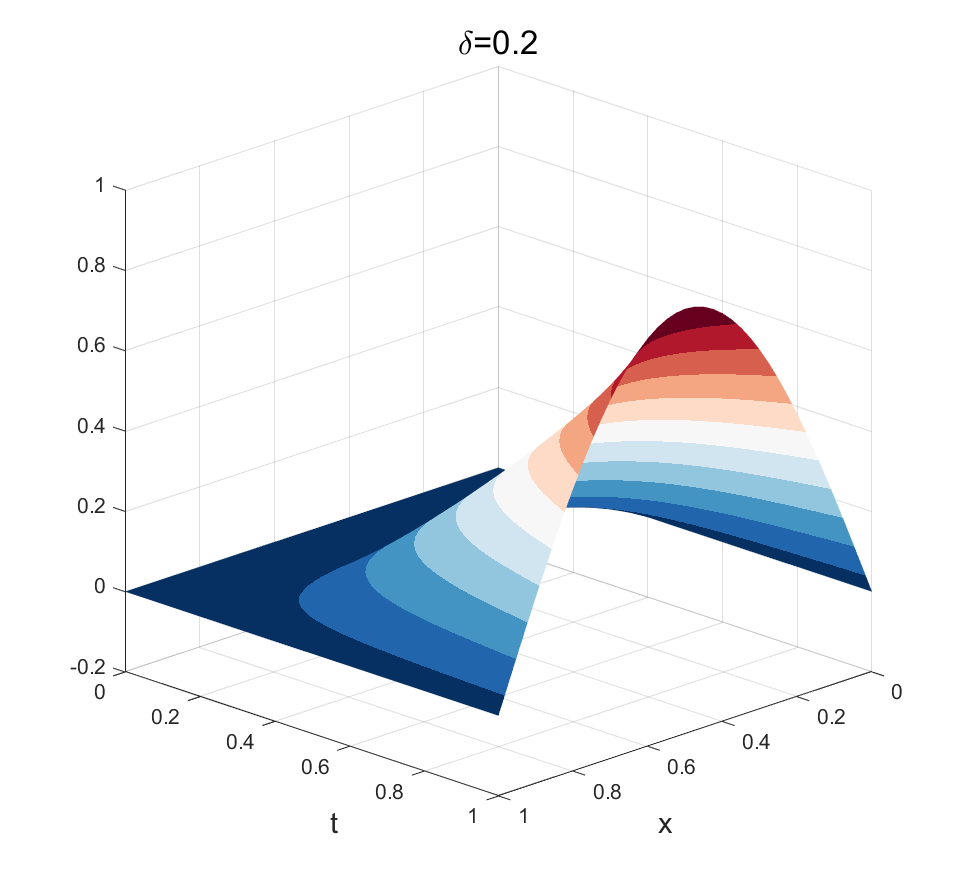}
\hspace{-5mm}
\label{3b}
\includegraphics[width=3.5cm,height=3.5cm]{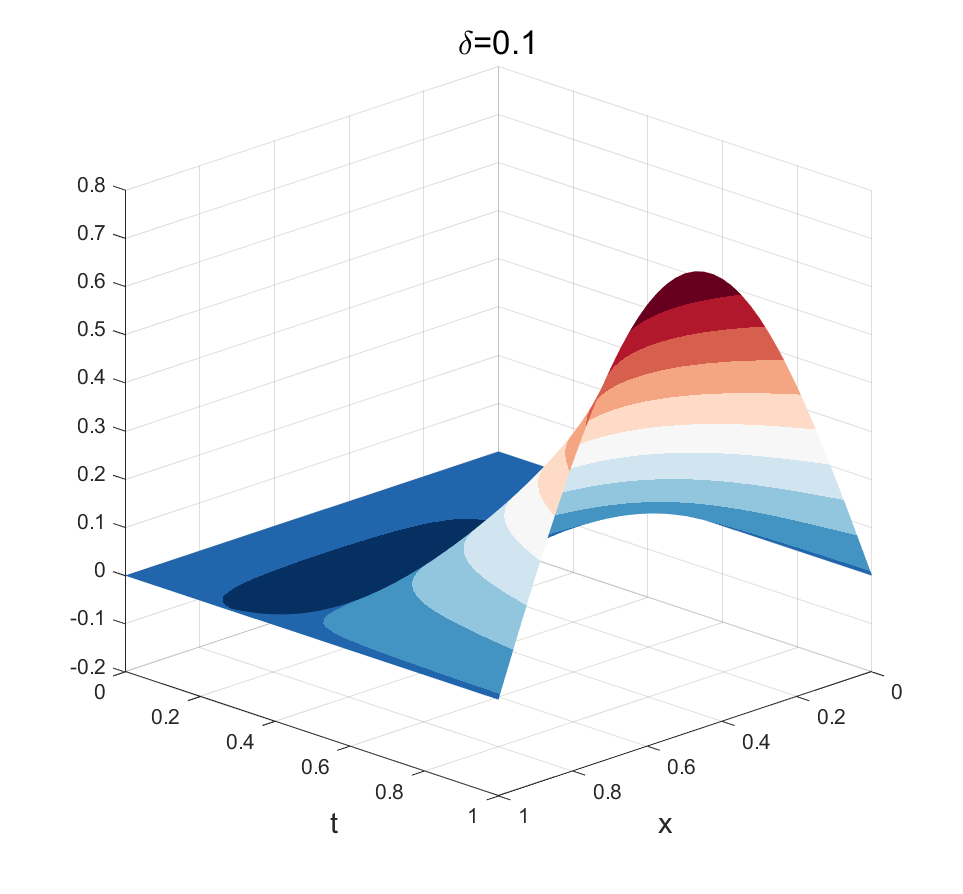}
\hspace{-5mm}
\label{3c}
\includegraphics[width=3.5cm,height=3.5cm]{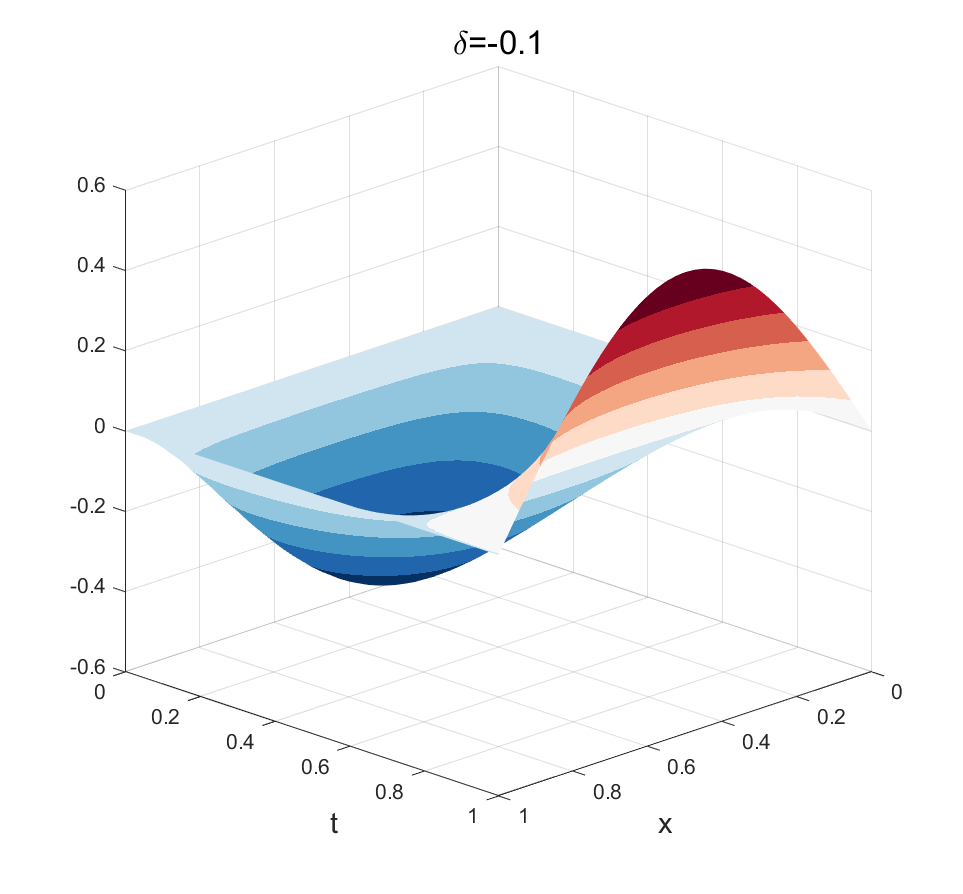}
\hspace{-5mm}
\label{3d}
\includegraphics[width=3.5cm,height=3.5cm]{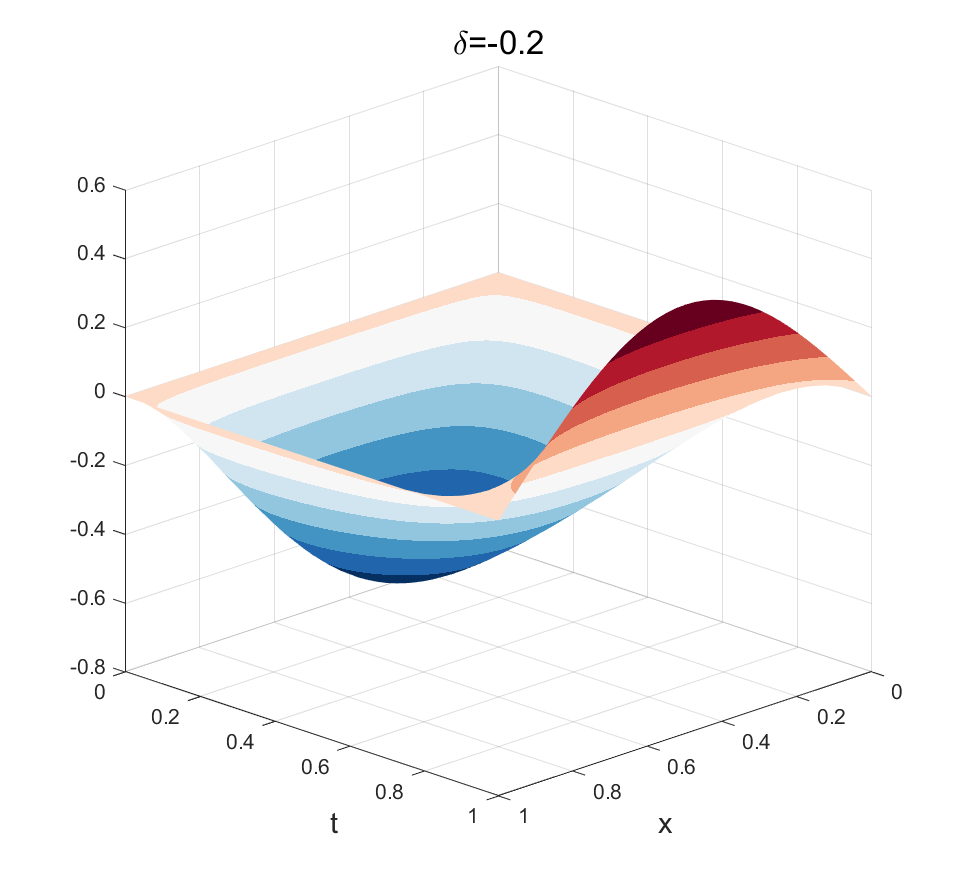}
\hspace{-5mm}
\caption{The profiles of expected numerical states with ${\tau}={h}=\frac{1}{40}$ and $\delta=0.2,0.1,-0.1,-0.2$ of Example \ref{exm1}.}
\label{figure3}
\end{figure}

\begin{table}[h]
\caption{The integral values of numerical states with $\delta = 0.2, 0.1, -0.1, -0.2$ of Example \ref{exm1}.}\label{table1}
\begin{tabular}{@{}cccccc@{}}
\toprule
${h} = {\tau}$  &$\frac{1}{40}$  &$\frac{1}{45}$ &$\frac{1}{50}$ &$\frac{1}{60}$ &$\frac{1}{70}$          \\
\midrule
$\delta = 0.2$        & 1.99913E-1  & 1.99920E-1  & 1.99926E-1  & 1.99934E-1  & 1.99941E-1  \\
$\delta = 0.1$        & 9.98696E-2  & 9.98800E-2  & 9.98887E-2  & 9.99020E-2  & 9.99120E-2  \\
$\delta = -0.1$       & -1.00201E-1 & -1.00185E-1 & -1.00171E-1 & -1.00151E-1 & -1.00135E-1 \\
$\delta = -0.2$       & -2.00232E-1 & -2.00213E-1 & -2.00198E-1 & -2.00174E-1 & -2.00156E-1 \\
\bottomrule
\end{tabular}
\end{table}

\begin{example}\label{exm2}
The objective functional in Example \ref{exm1} is considered, which is subject to the following stochastic parabolic equation with two-dimensional space domain $(x_1,x_2)\in \mathcal{D}=[0,1]\times[0,1]$:
\begin{equation*}
\left\{\begin{aligned}
dX(t)&=\left[\gamma\Delta X(t)+f(t)+U(t)\right]dt+\beta\sin(\pi x)(1+t)^2dW,\ t\in(0,1],\\
X(0)&=X_0,
\end{aligned}\right.
\end{equation*}
where $\gamma$ and $\beta$ are two constants.
\end{example}
The exact solutions and functions $X_0,f(t),X_d(t)$ as well as constraint parameter $\delta$ are given as follows:
\begin{equation*}
\left\{\begin{aligned}
&U(t)=(T-t)(1+\lambda t)\sin(\pi x_1)\sin(\pi x_2)(1+t)^2,\\
&X_0=\sin(\pi x_1)\sin(\pi x_2),\\
&X(t)=(1+\lambda t+\beta W_t)\sin(\pi x_1)\sin(\pi x_2)(1+t)^2,\\
&f(t)=(1+t)^2\sin(\pi x_1)\sin(\pi x_2)\\
&\quad\quad\quad\times\left(2\gamma\pi^2(1+\lambda t+\beta W_t)+(t-T)(1+\lambda t)+\frac{2(1+\lambda t+\beta W_t)}{1+t}+\lambda\right),\\
&\delta=\int_{0}^{1}\int_{\mathcal{D}}\mathbb{E}[X(t)]dxdt=\frac{17\lambda+28}{3\pi^2},\\
&X_d(t)=(1+t)^2\sin(\pi x_1)\sin(\pi x_2)\\
&\quad\quad\quad\times\left((1+\lambda t+\beta W_t)\left(2\gamma\pi^2(T-t)+2+ \frac{2(t-T)}{1+t}\right)+\lambda(t-T)\right)+\mu,
\end{aligned}\right.
\end{equation*}
where $\lambda$ and $\mu$ are the chosen constant and multiplier. We set $\gamma=\lambda=0.2$, $\beta=0.5$, $\mu=0.8$ and the other parameters are given as in Example \ref{exm1}. The cases ${h}=\sqrt{2}\left(\frac{1}{40},\frac{1}{45},\frac{1}{50},\frac{1}{60},\frac{1}{70}\right)$, ${\tau}=\frac{h}{\sqrt{2}}$ and ${h}=\sqrt{2}\left(\frac{1}{10},\frac{1}{15},\frac{1}{20},\frac{1}{25},\frac{1}{30}\right)$, ${\tau}=\frac{h^2}{2}$ are performed, respectively.
The convergence rates of control, state and multiplier are given in Figure \ref{figure4}. The numerical results are similar to the previous example, which validates the accuracy of the theoretical analysis and the effectiveness of our algorithm. 
Again, the different constraint parameters $\delta=1,0.5,-0.5,-1$
are tested separately, and the integral values of the numerical states are reported in Table \ref{table2}, which can be found that our algorithm can reliably capture the solution that satisfies the constraint condition.

\begin{figure}[!htbp]
\flushleft
\label{4a}
\includegraphics[width=4.5cm,height=4cm]{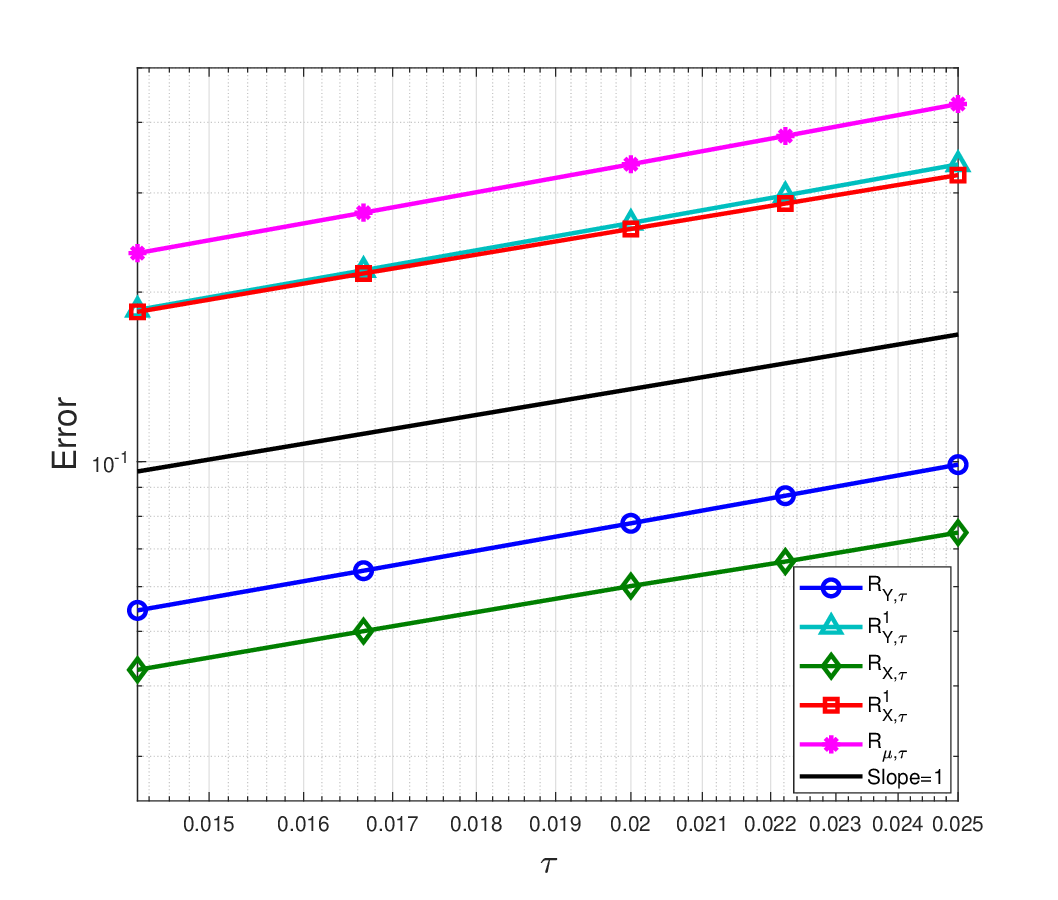}
\hspace{-5mm}
\label{4b}
\includegraphics[width=4.5cm,height=4cm]{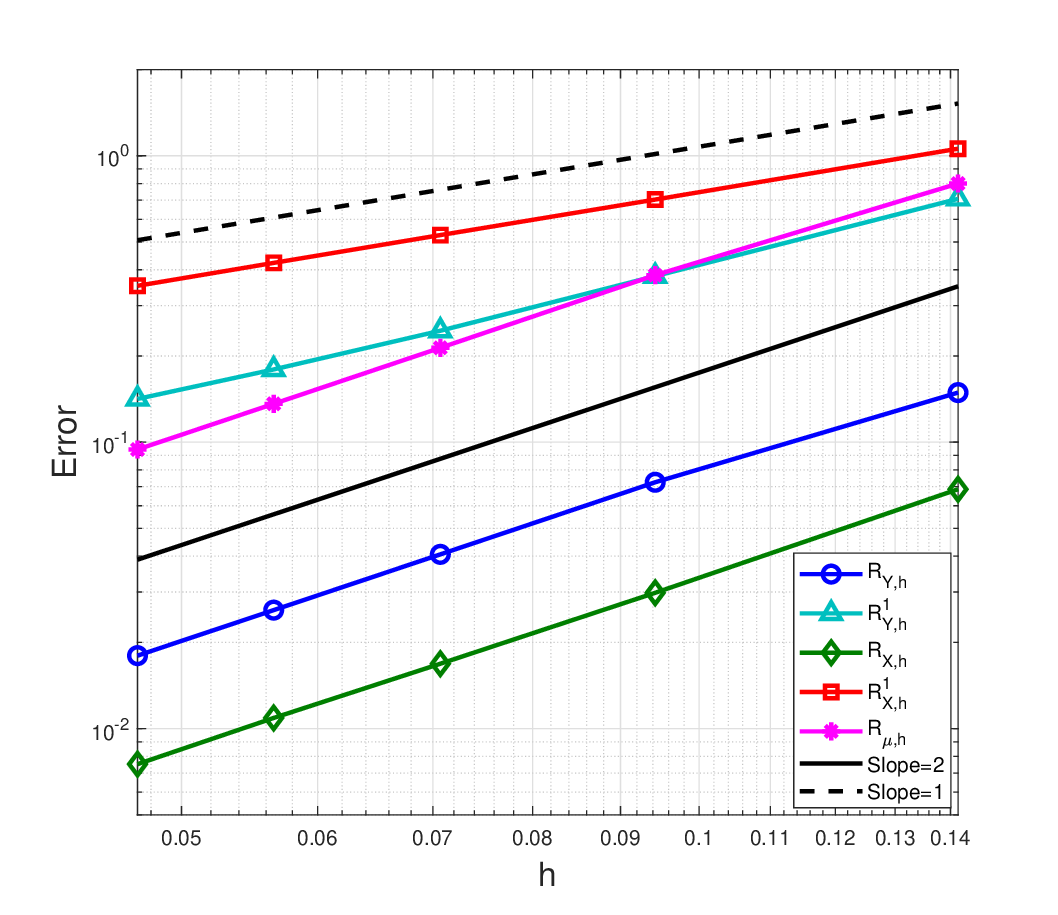}
\hspace{-5mm}
\label{4c}
\includegraphics[width=4.5cm,height=4cm]{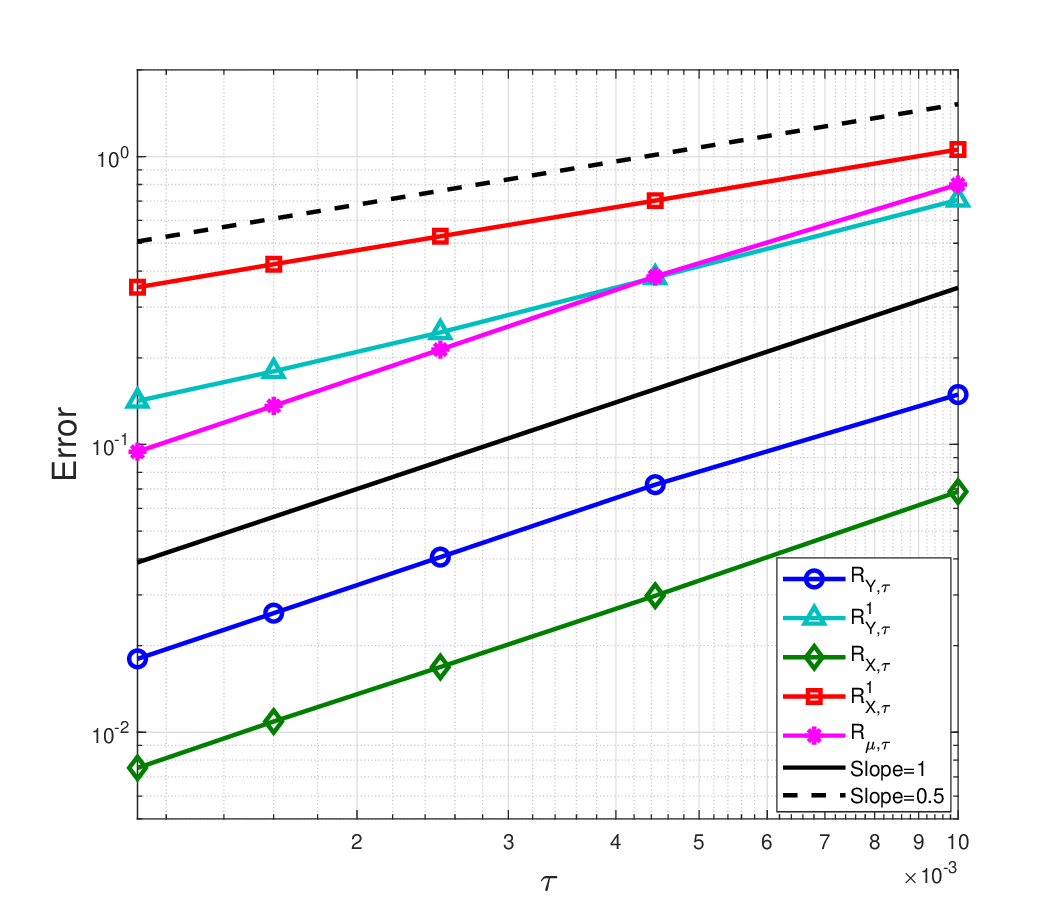}
\hspace{-5mm}
\caption{The convergence rates with ${h}=\sqrt{2}\left(\frac{1}{40},\frac{1}{45},\frac{1}{50},\frac{1}{60},\frac{1}{70}\right)$, ${\tau}=\frac{h}{\sqrt{2}}$ (Left) and ${h}=\sqrt{2}\left(\frac{1}{10},\frac{1}{15},\frac{1}{20},\frac{1}{25},\frac{1}{30}\right)$, ${\tau}=\frac{h^2}{2}$ (Middle, Right) of Example \ref{exm2}.}
\label{figure4}
\end{figure}

\begin{table}[h]
\caption{The integral values of numerical states with $\delta = 1, 0.5, -0.5, -1$ and $h = \sqrt{2}\tau$ of Example \ref{exm2}.}\label{table2}
\begin{tabular}{@{}cccccc@{}}
\toprule
${\tau}$ &$\frac{1}{40}$  &$\frac{1}{45}$ &$\frac{1}{50}$ &$\frac{1}{60}$ &$\frac{1}{70}$ \\
\midrule
$\delta = 1$     & 1.00000E-0  & 1.00000E-0  & 1.00000E-0  & 1.00000E-0  & 1.00000E-0  \\
$\delta = 0.5$   & 5.00000E-1  & 5.00000E-1  & 5.00000E-1  & 5.00000E-1  & 5.00000E-1  \\
$\delta = -0.5$  & -5.00000E-1 & -5.00000E-1 & -5.00000E-1 & -5.00000E-1 & -5.00000E-1 \\
$\delta = -1$    & -1.00000E-0 & -1.00000E-0 & -1.00000E-0 & -1.00000E-0 & -1.00000E-0 \\
\bottomrule
\end{tabular}
\end{table}

\section{Conclusion}
In this paper, the optimal strong error estimates of stochastic parabolic optimal control problem with integral state constraint and additive noise are derived. The SOCP is first discretized based on the time implicit discretization and piecewise linear finite element discretization in space, and the discrete first-order optimality condition is deduced by constructing the Lagrangian functional. 
Then the optimal strong convergence orders of fully discrete forward-backward stochastic parabolic equations are recovered, which is introduced as an auxiliary problem leading to the optimal a prior error estimates of control, state, adjoint state and multiplier. Further, In order to solve discrete SOCP with integral state constraint, an efficient gradient projection algorithm is proposed, the key idea of which is to ensure the state constraint by selecting specific multiplier in each iteration, and a convergence analysis of algorithm is given. Numerical example with one-dimensional and two-dimensional space domains are performed to validates the accuracy of the theoretical analysis and the effectiveness of our algorithm. 

According to \cite{point_state1,point_state2} the integral constraint can be used as a regularization method
to deal with the optimal control problem with pointwise state constraints, and
in our future work we will investigate the strong error estimates of stochastic optimal control problem with expected pointwise state constraint.

\bmhead{Acknowledgements}
The work of the authors was supported by the National Natural Science Foundation of China (12171042,11971259).


\section*{Declarations}
The authors declare that they have no conflict of interest.



\end{document}